\documentclass[11pt,reqno]{amsart}
\usepackage{amsthm,amsfonts,amssymb,euscript,hyperref}
\usepackage{graphicx}
\usepackage{comment}
\def\alphab{\underline{\alpha}}
\def\betab{\underline{\beta}}
\def\chib{\underline{\chi}}
\def\chibh{\hat{\underline{\chi}}}
\def\chih{\hat{\chi}}
\def\curl{\text{curl}\,}
\def\divergence{\text{div}\,}
\def\Divergence{\text{Div}\,}
\def\D{\mathcal{D}}
\def\etab{\underline{\eta}}
\def\Hb{\underline{H}}
\def\Lb{\underline{L}}
\def\mub{\underline{\mu}}
\def\kappab{\underline{\kappa}}
\def\tr{\text{tr}}
\def\omegab{\underline{\omega}}

\def\tensor{\widehat{\otimes}}
\def\ub{\underline{u}}
\def\O{\mathcal{O}}
\def\omegabt{\omegab^\dagger}
\def\omegat{\omega^\dagger}
\def\omegabpair{\langle \omegab \rangle}
\def\omegapair{\langle \omega \rangle}
\def\thetab{\underline{\theta}}
\def\phib{\underline{\phi}}
\def\phif{\,^{(4)}\phi}
\def\phibf{\,^{(4)}\!\underline{\phi}}
\def\phit{\,^{(3)}\phi}
\def\phibt{\,^{(3)}\!\underline{\phi}}
\def\varphib{\underline{\varphi}}

\def\L{\mathcal{L}}
\def\Lh{\widehat{\mathcal{L}}}

\def\pih{\hat{\pi}}

\def\doubleint{\int\!\!\!\!\!\int}
\def\Dslash{\mbox{$D \mkern -13mu /$\, }}

\def\OSzeroinfinity{{}^{(S)}\mathcal{O}_{0,\infty}}
\def\OSzerop{\,^{(S)}\mathcal{O}_{0,}}
\def\OSzerofour{\,^{(S)}\mathcal{O}_{0,4}}
\def\OSzerotwo{{}^{(S)}\mathcal{O}_{0,2}}
\def\OSonep{{}^{(S)}\mathcal{O}_{1,p}}
\def\OSonetwo{{}^{(S)}\mathcal{O}_{1,2}}
\def\OSonefour{{}^{(S)}\mathcal{O}_{1,4}}

\def\OHtwo{{}^{(H)}\!\mathcal{O}}
\def\OHbtwo{{}^{(\Hb)}\!\mathcal{O}}

\def\Ozeroinfinity{\mathcal{O}_{0,\infty}}
\def\Ozerofour{\mathcal{O}_{0,p}}
\def\Ozerofour{\mathcal{O}_{0,4}}
\def\Ozerotwo{\mathcal{O}_{0,2}}

\def\Oonetwo{\mathcal{O}_{1,2}}
\def\Oonefour{\mathcal{O}_{1,4}}
\def\Ozero{{\mathcal{O}^{(0)}}}

\def\Rinitial{{\mathcal{R}^{(0)}}}
\def\Finitial{{\mathcal{F}^{(0)}}}
\def\Oinitial{{\mathcal{O}^{(0)}}}

\def\R{\mathcal{R}}
\def\Rb{\underline{\mathcal{R}}}
\def\F{\mathcal{F}}
\def\Fb{\underline{\mathcal{F}}}

\def\Rzero{\mathcal{R}_0}
\def\Rone{\mathcal{R}_1}
\def\Rzerob{\underline{\mathcal{R}}_0}
\def\Roneb{\underline{\mathcal{R}}_1}

\def\Fzero{\mathcal{F}_0}
\def\Fone{\mathcal{F}_1}
\def\Ftwo{\mathcal{F}_2}
\def\Fzerob{\underline{\mathcal{F}}_0}
\def\Foneb{\underline{\mathcal{F}}_1}
\def\Ftwob{\underline{\mathcal{F}}_2}

\def\Izero{\mathcal{I}_0}

\def\Done{\D_1}
\def\Donestar{^*\!\D_1}
\def\Dtwo{\D_2}
\def\Dtwostar{^*\!\D_2}
\def\nablastar{^*\!\nabla}
\def\Wstar{{}^*W}
\def\Fstar{{}^*F}
\def\trchibt{\widetilde{\tr \chib}}
\newtheorem*{attention}{Attention}
\newtheorem*{theoremA}{Theorem A}
\newtheorem*{theoremB}{Theorem B}
\newtheorem*{theoremC}{Theorem C}
\newtheorem*{maintheorem}{Main Theorem}
\newtheorem*{mainCor1}{First Corollary}
\newtheorem*{mainCor2}{Second Corollary}
\newtheorem{theorem}{Theorem}[section]
\newtheorem{lemma}[theorem]{Lemma}
\newtheorem{proposition}[theorem]{Proposition}
\newtheorem{corollary}[theorem]{Corollary}
\newtheorem{definition}[theorem]{Definition}
\newtheorem{remark}[theorem]{Remark}

\numberwithin{equation}{section}

\begin{document}
\title[Formation of Black Holes]{Dynamical formation of black holes due to the condensation of matter field}
\author[Pin Yu]{Pin Yu}
\address{Mathematical Sciences Center}
\address{Tsinghua University\\ Beijing, China}
\email{pin@math.tsinghua.edu.cn}
\thanks{This work was partly done when the author was a graduate student at Princeton University and afterwards when he was visiting Harvard University. He would like to thank these institutions for their hospitality. The author is deeply indebted to Professor \emph{Sergiu Klainerman} and \emph{Igor Rodnianski} for many fruitful discussions on the problem. The author also would like to acknowledge Professor \emph{Demetrios Christodoulou} and \emph{Shing-Tung Yau} for explaining the insights and continuous encouragements.}

\begin{abstract}
The purpose of the paper is to understand a mechanism of evolutionary formation of trapped surfaces when there is an electromagnetic field coupled to the background space-time. Based on the \emph{short pulse} ansatz, on a given finite outgoing null hypersurface which is free of trapped surfaces, we exhibit an open set of initial data $(\chih,\alpha_F)$ for Einstein equations coupled with a Maxwell field, so that a trapped surface forms along the Einstein-Maxwell flow.

On one hand, this generalizes the black-hole-formation results of Christodoulou \cite{Ch} and Klainerman-Rodnianski \cite{K-R-09}. In fact, by switching off the electromagnetic field in our main theorem, we can retrieve their results in vacuum. On the other hand, this shows that the formation of black hole can be purely due to the condensation of Maxwell field on the initial null hypersurface where there is no incoming gravitational energy.
\end{abstract}
\maketitle

\section{Introduction}\label{introduction}
The famous singularity theorem of Penrose (see \cite{H-E}) states that if in addition to the dominant energy condition, the space-time has a trapped surface, namely a two dimensional space-like sphere whose outgoing and incoming expansions are negative, thus the space-time is future causally geodesically incomplete (we usually say that the space-time contains a \emph{singularity}). The \emph{weak cosmic censorship} conjecture asserts there is no naked singularity under reasonable physical assumptions. In other words, singularities need to be hidden from an observer at infinity by the event horizon of a black hole. Thus, by combining these two claims, if one can exhibit a trapped surface in a space-time, then one can predict the existence of black holes. In other words, although many supplementary conditions are required, we regard the existence of a trapped surface as the presence of a black hole.

A major challenge in general relativity is to understand how trapped surfaces actually form due to the focusing of gravitational waves. In a recent remarkable breakthrough \cite{Ch}, Christodoulou solved this long standing problem. He discovered a mechanism which is responsible for the dynamical formation of trapped surfaces in vacuum space-times. In the monograph \cite{Ch}, in addition to the Minkowskian flat data on a incoming null hypersurface, Christodoulou identified an open set of initial data (this is the \emph{short pulse} ansatz) on a outgoing null hypersurfaces. Based on the techniques developed by himself and Klainerman in the proof of the global stability of the Minkowski space-times \cite{Ch-K}, he managed to understand the whole picture of how the various estimates on geometric quantities propagates along the evolution. Once those estimates are established in a large region of the space-time, the actual formation of trapped surfaces is easy to demonstrate. Christodoulou also proved a version of the same result for the short pulse data prescribed on past null infinity. He showed that strongly focused gravitational waves, coming in from past null infinity, lead to a trapped surface. This miraculous work provides the first global \emph{large data} result in general relativity (without symmetry assumptions) and opens the gate for many new developments on dynamical problems related to black holes. The methods undoubtedly have many future applications in both general relativity and other partial differential equations.

In \cite{K-R-09}, Klainerman and Rodnianski extend aforementioned result which significantly simplifies the proof of Christodoulou (from about six hundred pages to one hundred and twenty). They enlarge the admissible set of initial conditions and show that the corresponding propagation estimates are much easier to derive. The relaxation of the propagation estimates are just enough to guarantee that a trapped surface still forms. Based on the trace estimates developed in a sequence of work \cite{K-R-04}, \cite{K-R-05}, \cite{K-R-05-Rough}, \cite{K-R-05-Microlocolized-Rough}, \cite{K-R} and \cite{K-R-Sharp} towards the critical local well-posedness for Einstein vacuum equations, they reduce the number of derivatives needed in the argument from two derivatives on curvature (in Christodoulou's proof) to just one. More importantly, Klainerman and Rodnianski introduce a parabolic scaling in \cite{K-R-09} which is incorporated into Lebesgue norms and Sobolev norms. These new techniques allow them to capture the hidden \emph{smallness} of the nonlinear interactions among different small or large components of various geometric objects. The result of Klainerman and Rodnianski can be easily localized with respect to angular sectors which leads to further developments, see \cite{K-R-10} for details. We remark that Klainerman and Rodnianski only considered the problem on a finite region. The question from past null infinity can be solved in a similar manner as in \cite{Ch} once one understand the picture on a finite region. The problem from past null infinity has also been studied in a recent work \cite{R-T} by Reiterer and Trubowitz.

The aforementioned works are all investigating vacuum space-times. When some matter field is presenting, the formation of black holes has attracted a lot of interest. We mention only two of them which are more related to the present paper. In \cite{S-Y}, based on their work on positive mass theorem, especially the resolution of Jang's equation, Schoen and Yau proved the existence of a trapped surface when matter is condensed in a small region, see \cite{Y} for an improvement. Their result is restricted to the initial 3-slice hence is not dynamical. Another work is an earlier paper \cite{Ch-91} of Christodoulou. He studied the evolutionary formation of singularities for the Einstein-scalar field system. The radial symmetry is assumed and the mechanism of the formation of trapped surface is quite different to his most recent work \cite{Ch}. And moreover, he could get more precise information on the gravitational collapse.

The purpose of the present paper is, under the framework of Christodoulou and Klainerman-Rodnianski, to understand the formation of charged black holes, namely the evolutionary formation of trapped surfaces when there is an electromagnetic field coupled to the background space-time.  On one hand, we will generalize the black-hole-formation results of Christodoulou and Klainerman-Rodnianski. In fact, by switching off the electromagnetic field in our main theorem, we can retrieve their results in vacuum. On the other hand, we will show that the formation of black hole can purely due to condensation of Maxwell field on the initial null hypersurface, namely, we can set the incoming gravitational energy to be zero.


\subsection{Framework} We recall double-null-foliation formalism and we refer the reader to Chapter 3 of \cite{K-N} for more precise descriptions. We use $\D = \D(u_*,\ub_*)$ to denote the underlying space-time and use $g$ to denote the background metric. We assume that $\D$ is spanned by a double null foliation generated by two optical functions $u$ and $\ub$ and we also assume that $u$ and $\ub$ increase towards the future, $0 \leq u \leq u_*$ and $0 \leq \ub \leq \ub_*$. We use $H_u$ to denote the outgoing null hypersurfaces generated by the level surfaces of $u$ and use ${\Hb}_{\ub}$ to denote the incoming null hypersurfaces generated by the level surfaces of $\ub$. We use $S_{u,\ub}$ to denote the space-like two surface $H_u \cap \Hb_{\ub}$. We denote by $H_u^{(\ub_1,\ub_2)}$ the region of $H_u$ defined by $\ub_1 \leq \ub \leq\ub_2$; similarly, we can define $\Hb_{\ub}^{(u_1,u_2)}$.

We require that when $\ub \leq 0$, this part of initial null cone $H_0$ is a flat light cone in Minkowski space-time.

\begin{minipage}[!t]{0.5\textwidth}
  \includegraphics[width = 3 in]{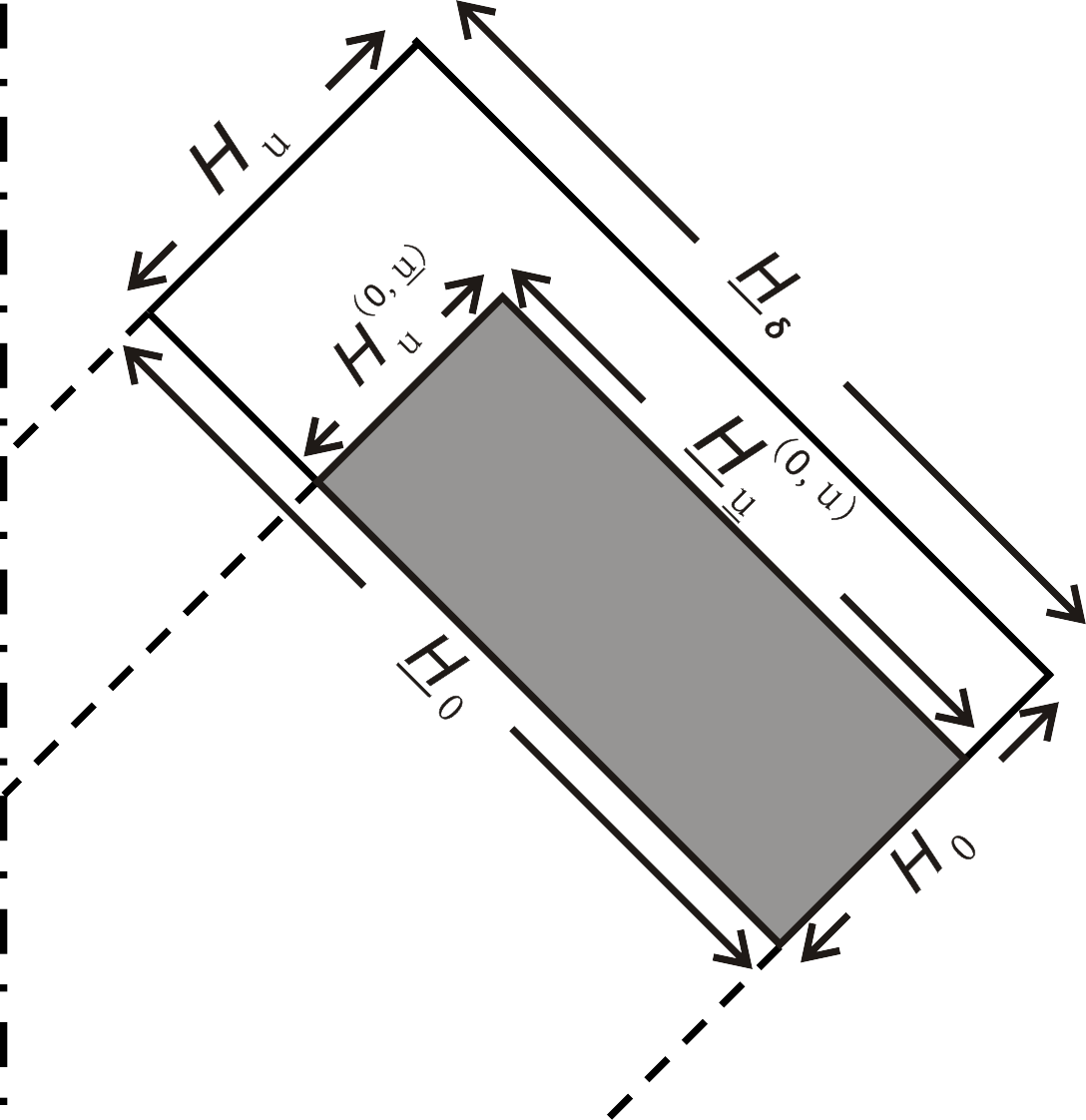}
\end{minipage}
\hspace{0.0\textwidth}
\begin{minipage}[!t]{0.4\textwidth}
The shaded region represents the domain $\D(u,\ub)$.  The function $\ub$ is in fact defined from $-u_*$ to $\delta$. When $\ub \leq 0$, this part of $H_0$ is assumed to be a flat light cone in Minkowski space-time with vertex located at $\ub = -u_*$. We shall show that the trapped surface forms when $\ub \sim 1$ and $u = \delta$. We also require that $u_*$ is a fixed number which is at least $2$, say, we may take $u_* = 2$.
\end{minipage}

Let $(L,\Lb)$ be the null geodesic generators of the double null foliation and we define the lapse function $\Omega$ by $ g(L,\Lb) = -\frac{2}{\Omega^2}$. We also need normalized null pair $(e_3, e_4)$, i.e. $e_3 = \Omega \Lb$, $e_4 = \Omega L$ and $g(e_3,e_4)= -2$. On a given two sphere $S_{u,\ub}$ we choose an orthonormal frame $(e_1,e_2)$. We call $(e_1, e_2, e_3, e_4)$ a \emph{null frame}. \footnote{\, Throughout the paper, we use Greek letters $\alpha, \beta, \cdots$ to denote an index from $1$ to $4$ and Latin letters $a, b, \cdots$ to denote an index from $1$ to $2$. Repeated indices should always be understood as summations.}

We use $D$ to denote Levi-Civita connection of the metric $g$ and we define the \emph{connection coefficients},
\begin{align*}
 \chi_{ab}&= g(D_b e_4, e_b), \eta_a = -\frac{1}{2}g(D_3 e_a, e_4),  \omega = -\frac{1}{4} g(D_4 e_3, e_4),\\
\chib_{ab}&=  g(D_b e_3, e_b),  \etab_a = -\frac{1}{2}g(D_4 e_a, e_3), \omegab = -\frac{1}{4}g(D_3 e_4, e_3), \zeta_a  = \frac{1}{2} g(D_a e_4, e_3).
\end{align*}
where $D_a = D_{e_a}$. On $S_{u.\ub}$ we use $\nabla$ to denote the induced connection and we use $\nabla_3$ and $\nabla_4$ to denote the projections to $S_{u.\ub}$ of $D_3$ and $D_4$. Those $\nabla$ derivatives are called \emph{horizontal derivatives} (we use the word \emph{horizontal tensor fields} for tensor fields depending only on components along $S_{u,\ub}$, see \cite{Ch}). 

For a Weyl field $W$, the null decomposition adapted to the null frame $\{e_\alpha\}$ is \footnote{\, When $W$ is equal to Weyl curvature tensor, we use $\alpha, \alphab, \beta, \betab, \rho, \sigma$ to denote its null components.}
\begin{align*}
 \alpha(W)_{ab} &= W(e_a, e_4, e_b, e_4), \quad  \beta(W)_a = \frac{1}{2}W(e_a,e_4,e_3,e_4), \quad \rho(W) = \frac{1}{4} W(e_4,e_3,e_4,e_3), \\
 \alphab(W)_{ab} &= W(e_a, e_3, e_b, e_3) \quad \betab(W)_a = \frac{1}{2}W(e_a,e_3,e_3,e_4), \quad \sigma(W) = \frac{1}{4}{} ^*W(e_4,e_3,e_4,e_3).
\end{align*}
where ${}^*W$ is the space-time Hodge dual of $W$. 

A \emph{Maxwell field} shall always refer to a two form $F_{\alpha\beta}$ satisfying \emph{Maxwell equations},
\begin{equation}\label{Maxwell_Equations}
D_{[\gamma}F_{\alpha\beta]} = 0, \quad D^\alpha F_{\alpha\beta} = 0.
\end{equation}
The Maxwell field $F_{\alpha\beta}$ is coupled to the background geometry, namely, in addition to \eqref{Maxwell_Equations}, $F_{\alpha\beta}$ satisfies Einstein field equations,
\begin{equation}\label{Einstein_Equations}
 R_{\alpha\beta}-\frac{1}{2}R g_{\alpha\beta} = T_{\alpha\beta},
\end{equation}
where $R_{\alpha\beta}$ is Ricci curvature, $R$ is scalar curvature and the $T_{\alpha\beta}$ is the
energy-momentum tensor associated to $F_{\alpha\beta}$ defined as
\begin{equation*}
T_{\alpha\beta} = F_{\alpha\mu}F_{\beta}{}^{\mu} - \frac{1}{4} g_{\alpha\beta}F_{\mu\nu}F^{\mu\nu}.
\end{equation*}
We remark that the energy-momentum tensor $T_{\alpha\beta}$ can also be written as
\begin{equation*}
T_{\alpha\beta} = \frac{1}{2}(F_{\alpha}{}^\mu F_{\beta\mu}+{}^*F_{\alpha}{}^\mu {}^*F_{\beta\mu}),
\end{equation*}
where ${}^*F$ denotes the space-time Hodge dual of $F$. The symmetry between $F$ and ${}^* F$ in above expression plays an important role when we derive energy estimates.

We also decompose $F$ in a given null frame into null components, \footnote{\, We use shorthand symbols $\alpha_F$, $\beta_F$, $\rho_F$ and $\sigma_F$ to denote these components.}
\begin{align*}
 {\alpha(F)}_{a}= F(e_a, e_4), {\alphab(F)}_{a} = F(e_a, e_3), {\rho(F)} = \frac{1}{2} F(e_3,e_4),{\sigma(F)} =  \frac{1}{2}{} ^*F(e_3,e_4)=F(e_1,e_2).
\end{align*}

Using null components, \eqref{Einstein_Equations} are equivalent to following null structure equations, see \cite{B-Z}, \cite{Ch} or \cite{Ch-K} for details. \footnote{\, Thanks to \eqref{Einstein_Equations}, we can also use $R_{\mu\nu}$ to denote the energy momentum tensor of $F$.}
\begin{equation}\label{NSE_L_tr_chi}
 \nabla_4 \tr \chi + \frac{1}{2}(\tr \chi)^2 = -|\chih|^2-2\omega \tr \chi - T_{44},
\end{equation}
\begin{equation}\label{NSE_L_chih}
 \nabla_4 \chih + \tr \chi \,\chih = -2\omega \chih -\alpha,
\end{equation}
\begin{equation}\label{NSE_Lb_tr_chib}
 \nabla_3 \tr \chib + \frac{1}{2}(\tr \chib)^2 = -|\chibh|^2-2\omegab \tr \chib - T_{33},
\end{equation}
\begin{equation}\label{NSE_Lb_chibh}
 \nabla_3 \chibh + \tr \chib \, \chibh = -2\omegab \chibh -\alphab,
\end{equation}
\begin{equation}\label{NSE_L_eta}
 \nabla_4 \eta = -\chi \cdot (\eta -\etab)-\beta - \frac{1}{2}T_{b4},
\end{equation}
\begin{equation}\label{NSE_Lb_etab}
 \nabla_3 \etab = -\chib \cdot (\etab -\eta)+\betab + \frac{1}{2}T_{b3},
\end{equation}
\begin{equation}\label{NSE_L_omegab}
 \nabla_4 \omegab = 2\omega \omegab +\frac{3}{4}|\eta-\etab|^2-\frac{1}{4}(\eta -\etab)\cdot(\eta +\etab)-\frac{1}{8}|\eta +\etab|^2 + \frac{1}{2}\rho + \frac{1}{4} T_{43},
\end{equation}
\begin{equation}\label{NSE_Lb_omega}
 \nabla_3 \omega = 2\omegab \omega +\frac{3}{4}|\eta-\etab|^2+\frac{1}{4}(\eta -\etab)\cdot(\eta +\etab)-\frac{1}{8}|\eta +\etab|^2 + \frac{1}{2}\rho + \frac{1}{4} T_{34},
\end{equation}
\begin{equation}\label{NSE_L_tr_chib}
 \nabla_4 \tr \chib + \frac{1}{2}\tr \chi \, \tr \chib=2 \omega \tr \chib + 2 \divergence \etab  + 2|\etab|^2 + 2\rho - \chih \cdot \chibh,
\end{equation}
\begin{equation}\label{NSE_Lb_tr_chi}
 \nabla_3 \tr \chi + \frac{1}{2}\tr \chib \, \tr \chi=2 \omegab \tr \chi + 2 \divergence \eta + 2|\eta|^2 + 2\rho - \chih \cdot \chibh,\footnote{\, $\divergence$ denotes divergence operator on $S_{u, \ub}$; $\Divergence$ denotes the space-time divergence.}
\end{equation}
\begin{equation}\label{NSE_L_chibh}
 \nabla_4 \chibh + \frac{1}{2} \tr \chi \chibh = \nabla \tensor \etab +2\omega \chibh -\frac{1}{2}\tr \chib \chih +\etab \tensor \etab + \hat{T}_{ab},
\end{equation}
\begin{equation}\label{NSE_Lb_chih}
 \nabla_3 \chih + \frac{1}{2} \tr \chib \chih = \nabla \tensor \eta +2\omegab \chih -\frac{1}{2}\tr \chi \chibh +\eta \tensor \eta + \hat{T}_{ab}.
\end{equation}
\begin{equation}\label{NSE_div_chih}
 \divergence \chih =\frac{1}{2}\nabla \tr \chi -\frac{1}{2}(\eta-\etab)\cdot(\chih -\frac{1}{2}\tr \chi \delta_{ab})-\beta+ \frac{1}{2}T_{4b},
\end{equation}
\begin{equation}\label{NSE_div_chibh}
 \divergence \chibh =\frac{1}{2}\nabla \tr \chib -\frac{1}{2}(\etab-\eta)\cdot(\chibh -\frac{1}{2}\tr \chib \delta_{ab})+\betab+ \frac{1}{2}T_{3b},
\end{equation}
\begin{equation}\label{NSE_curl_eta}
 \curl\eta = \chibh \wedge \chih +\sigma \epsilon_{ab},
\end{equation}
\begin{equation}\label{NSE_curl_etab}
 \curl \etab = -\chibh \wedge \chih - \sigma \epsilon_{ab},
\end{equation}
\begin{equation}\label{NSE_gauss}
 K =-\frac{1}{4}\tr \chi \tr \chib + \frac{1}{2} \chih \cdot \chibh -\rho + \frac{1}{4}T_{43}.\footnote{\, $K$ is the Gauss curvature of  $S_{u,\ub}$.}
\end{equation}

The second Bianchi equations are equivalent to null Bianchi equations,
\begin{equation}\label{NBE_Lb_alpha}
 \nabla_3 \alpha + \frac{1}{2}\tr \chib \alpha =\nabla \tensor \beta + 4\omegab \alpha - 3(\chih \rho + ^*\!\chih \sigma)+(\zeta + 4\eta)\tensor \beta +\frac{1}{2}(D_3 R_{44}-D_4 R_{43})\delta_{ab},\footnote{\, In applications, we always eliminate $\zeta$ by $\zeta = \frac{1}{2}(\eta -\etab)$.}
\end{equation}
\begin{equation}\label{NBE_L_beta}
\nabla_4 \beta + 2 \tr \chi \beta = \divergence \alpha - 2\omega \beta +\eta \cdot \alpha - \frac{1}{2}(D_b R_{44}-D_4 R_{4b}),
\end{equation}
\begin{equation}\label{NBE_Lb_beta}
\nabla_3 \beta +  \tr \chib \beta = \nabla \rho + ^*\! \nabla \sigma + 2\omegab \beta + 2\chih \cdot \betab + 3(\eta \rho + ^*\!\eta \sigma) +\frac{1}{2}(D_b R_{34}-D_4 R_{3b}),
\end{equation}
\begin{equation}\label{NBE_L_sigma}
 \nabla_4 \sigma + \frac{3}{2} \tr \chi \sigma = -\divergence ^*\! \beta + \frac{1}{2}\chibh \cdot ^*\!\alpha - \zeta \cdot ^*\! \beta -2\etab \cdot ^*\!\beta-\frac{1}{4}(D_\mu R_{4\nu}-D_\nu R_{4\mu})\epsilon^{\mu\nu}{}_{34},
\end{equation}
\begin{equation}\label{NBE_Lb_sigma}
 \nabla_3 \sigma + \frac{3}{2} \tr \chib \sigma = -\divergence ^*\! \betab + \frac{1}{2}\chih \cdot ^*\!\alphab - \zeta \cdot ^*\! \betab -2\eta \cdot ^*\!\betab+\frac{1}{4}(D_\mu R_{3\nu}-D_\nu R_{3\mu})\epsilon^{\mu\nu}{}_{34},
\end{equation}
\begin{equation}\label{NBE_L_rho}
 \nabla_4 \rho + \frac{3}{2} \tr \chi \rho = \divergence \beta -\frac{1}{2}\chibh \cdot \alpha + \zeta \cdot \beta + 2\etab \cdot \beta -\frac{1}{4}(D_3 R_{44}-D_4 R_{34}),
\end{equation}
\begin{equation}\label{NBE_Lb_rho}
 \nabla_3 \rho + \frac{3}{2} \tr \chib \rho = -\divergence \betab-\frac{1}{2}\chih \cdot \alphab + \zeta \cdot \betab - 2\eta \cdot \betab +\frac{1}{4}(D_3 R_{34}-D_4 R_{33}),
\end{equation}
\begin{equation}\label{NBE_L_betab}
\nabla_4 \betab +  \tr \chi \betab = -\nabla \rho + ^*\! \nabla \sigma + 2\omega \betab + 2\chibh \cdot \beta - 3(\etab \rho - ^*\!\etab \sigma) -\frac{1}{2}(D_b R_{43}-D_3 R_{4b}),
\end{equation}
\begin{equation}\label{NBE_Lb_betab}
\nabla_3 \betab + 2 \tr \chib \, \betab = -\divergence \alphab - 2\omegab \betab +\etab \cdot \alphab + \frac{1}{2}(D_b R_{33}-D_3 R_{3b}),
\end{equation}
\begin{equation}\label{NBE_L_alphab}
 \nabla_4 \alphab + \frac{1}{2}\tr \chi \alphab =-\nabla \tensor \betab + 4\omega \alphab - 3(\chibh \rho - ^*\!\chibh \sigma)+(\zeta - 4\etab)\tensor \betab +\frac{1}{2}(D_4 R_{33}-D_3 R_{34})\delta_{ab}.
\end{equation}

The Maxwell equations \eqref{Maxwell_Equations} are equivalent to null Maxwell equations,
\begin{equation}\label{NM_L_alphab}
\nabla_4 {\alphab_F}  + \frac{1}{2}\tr\chi {\alphab_F} = -\nabla {\rho_F} -^*\!\nabla {\sigma_F}-2\,^*\!\etab \cdot {\sigma_F} - 2\etab \cdot {\rho_F} + 2\omega {\alphab_F}-\chibh \cdot {\alpha_F}, \footnote{\, $^* \nabla_a =\epsilon_{ab}\nabla^b$.}
\end{equation}
\begin{equation}\label{NM_Lb_alpha}
\nabla_3 {\alpha_F}  + \frac{1}{2}\tr\chib {\alpha_F} = -\nabla {\rho_F} + ^*\!\nabla {\sigma_F}-2\,^*\!\etab \cdot {\sigma_F} + 2\etab \cdot {\rho_F} + 2\omegab {\alpha_F}-\chih \cdot {\alphab_F},
\end{equation}
\begin{equation}\label{NM_L_rho}
\nabla_4 {\rho_F} = - \divergence {\alpha_F} -\tr\chi {\rho_F} -(\eta-\etab)\cdot{\alpha_F},
\end{equation}
\begin{equation}\label{NM_L_sigma}
\nabla_4 {\sigma_F} = - \curl {\alpha_F} -\tr\chi {\sigma_F} +(\eta-\etab)\cdot \, ^*\!{\alpha_F},
\end{equation}
\begin{equation}\label{NM_Lb_rho}
\nabla_3 {\rho_F} = \divergence {\alphab_F} + \tr\chib {\rho_F} + (\eta-\etab)\cdot{\alphab_F},
\end{equation}
\begin{equation}\label{NM_Lb_sigma}
\nabla_3 {\sigma_F} = -\curl {\alphab_F} -\tr\chib {\sigma_F} +(\eta-\etab)\cdot \, ^*\!{\alphab_F}.
\end{equation}

\subsection{Scale Invariant Norms}
We recall the concept of \emph{signature} and \emph{scale} for very horizontal tensor field introduced in \cite{K-R-09}. Let $\phi$ be a horizontal tensor,
we use $N_a(\phi)$, $N_3(\phi)$ and $N_4(\phi)$ to denote the number of times $(e_a)_{i=1,2}$, respectively $e_3$ and $e_4$ appearing in the definition of $\phi$. We define $sgn{\phi}$ the \emph{signature} of $\phi$ to be $sgn(\phi) = N_4(\phi) + \frac{1}{2} N_a(\phi) -1$.
We also assign a \emph{scale} $sc(\phi)$ for $\phi$ to be $sc(\phi) = -sgn(\phi)+ \frac{1}{2}$.
For horizontal derivatives, we also assign signatures via the following convention, $sgn(\nabla_4\phi)= sgn(\phi)+1$, $sgn(\nabla \phi) = sgn(\phi)+\frac{1}{2}$ and $sgn(\nabla_3 \phi)= sgn(\phi)+ 0$. For product of two horizontal fields, we define the signature to be the sum, i.e. $sgn(\phi_1 \cdot \phi_2) = sgn(\phi_1)+sgn(\phi_2)$.


We introduce the \emph{scale invariant norms} \`a la Klainerman and Rodnianski. Along the outgoing null hypersurfaces $H_u^{(0,\ub)}$ and incoming null hypersurfaces $\Hb_{\ub}^{(0,u)}$, we define scale invariant $L^2$ norms,
\begin{equation*}
 \|\phi\|_{L^2_{(sc)}(H_u^{(0,\ub)})} =  \delta^{-sc(\phi)-1}\|\phi\|_{L^2(H_u^{(0,\ub)})}, \quad \|\phi\|_{L^2_{(sc)}(\Hb_{\ub}^{(0,u)})} = \delta^{-sc(\phi)-\frac{1}{2}} \|\phi\|_{L^2(\Hb_{\ub}^{(0,u)})}.
\end{equation*}
On two surface $S_{u,\ub}$, we define the scare invariant $L^p$ norms,\footnote{\, We use shorthand notation $\|\phi\|_{L^p_{(sc)}(u,\ub)}$ for this norm.}
\begin{equation*}
 \|\phi\|_{L^p_{(sc)}(S_{u,\ub})} = \delta^{-sc(\phi)-\frac{1}{p}}  \|\phi\|_{L^p(S_{u,\ub})}.
\end{equation*}
By definition, scale invariant norms come up naturally with a small parameter $\delta$. Roughly speaking, it reflects the \emph{smallness} of the nonlinear interactions in our problem. One key aspect of this principle is the scale invariant H\"older's inequality,
\begin{equation}\label{Holder}
\|\phi_1 \cdot \phi_2 \|_{L^{p}_{(sc)}(S_{u,\ub})} \leq \delta^{\frac{1}{2}}\|\phi_1\|_{L^{p_1}_{(sc)}(S_{u,\ub})} \|\phi_2\|_{L^{p_2}_{(sc)}(S_{u,\ub})},
\end{equation}
where $\frac{1}{p} =  \frac{1}{p_1} +  \frac{1}{p_2}$. Similar inequalities hold along incoming and outgoing hypersurfaces. We emphasize the rule of thumb for treating the nonlinear terms.
\begin{quotation}
For a product of terms each use of H\"older's inequality gains a $\delta^{\frac{1}{2}}$.
\end{quotation}
In this sense, nonlinear terms are better than linear terms. Other other hand, there are nonlinear terms where we do \emph{not} gain any power in $\delta$. If $f$ is a bounded (in usual sense) scalar function (bounded by a universal constant), the best we expect is $\|f \cdot \phi \|_{L^{p}_{(sc)}} \lesssim \|\phi\|_{L^{p}_{(sc)}}$. \footnote{\, This is the case for $f = \tr\chib_0 =\frac{4}{2r_0 + \ub - u}$ appearing frequently in the paper.}

We now introduce schematic notations $\psi$, $\psi_g$, $\Psi$, $\Psi_g$, $\Psi_4$, $\Upsilon$, $\Upsilon_g$ and $\Upsilon_4$. We use $\psi$, $\Psi$ and $\Upsilon$ to denote any connection coefficient \footnote{\, Instead of considering $\tr\chib$, we always use $\trchibt$ defined by $\trchibt = \tr\chib-\tr\chib_0$.}, null curvature component and null Maxwell component respectively; We use $\psi_g$, $\Psi_g$ and $\Upsilon_g$ to denote a \emph{good} connection coefficient, null curvature component and null Maxwell component respectively, i.e. $\psi_g \neq \chih, \chibh$, $\Psi_g \neq \alpha$ and $\Upsilon_g \neq \alpha_F$.  We require $\Psi_4 \neq \alphab$ and $\Upsilon_4 \neq \alphab_F$.

For $p = 2$ or $4$, $S = S_{u,\ub}$, $H=H_u^{(0,\ub)}$ and $\Hb = \Hb_{\ub}^{(0,u)}$, we introduce a family of scale invariant norms for connection coefficients, \footnote{\, We use shorthand notations $\|\psi\|$ to denote the sum for all the possible $\|\psi\|'s$ and $\|(\psi,\psi',\psi'',\cdots)\| = \|\psi\|+\|\psi'\|+\|\psi''\|+ \cdots$}
\begin{align*}
 \OSzeroinfinity (u,\ub) &= \|\psi\|_{L^\infty_{(sc)}(S)}, \OSzerop (u,\ub) =\delta^{\frac{1}{p}}\|(\chih, \chibh)\|_{L^p_{(sc)}(S)} + \|\psi_g\|_{L^p_{(sc)}(S)},\\
\OSonep (u,\ub) &=\|\nabla \psi\|_{L^p_{(sc)}(S)}, \OHtwo(u,\ub)=\|\nabla^2 \psi\|_{L^2_{(sc)}(H)}, \OHbtwo(u,\ub)=\|\nabla^2 \psi \|_{L^2_{(sc)}(\Hb)};
\end{align*}
For curvature components,
\begin{align*}
\Rzero(u,\ub) &= \delta^{\frac{1}{2}}\|\alpha\|_{L^2_{(sc)}(H)} + \|(\beta, \rho, \sigma, \betab)\|_{L^2_{(sc)}(H)}, \Rone(u,\ub) = \delta^{\frac{1}{2}}\|\nabla_4 \alpha\|_{L^2_{(sc)}(H)} + \|\nabla \Psi_4\|_{L^2_{(sc)}(H)},\\
\Rzerob(u,\ub)&= \delta^{\frac{1}{2}}\|\beta\|_{L^2_{(sc)}(\Hb)} + \|(\rho, \sigma, \betab, \alphab)\|_{L^2_{(sc)}(\Hb)}, \Roneb(u,\ub) = \delta^{\frac{1}{2}}\|\nabla_3 \alphab\|_{L^2_{(sc)}(\Hb)} + \|\nabla \Psi_g\|_{L^2_{(sc)}(\Hb)};
\end{align*}
For Maxwell components,
\begin{align*}
\Fzero(u,\ub) & = \delta^{\frac{1}{2}}\|\alpha_F\|_{L^2_{(sc)}(H)} + \|(\rho_F, \sigma_F)\|_{L^2_{(sc)}(H)}, \Fzerob(u,\ub) = \delta^{\frac{1}{2}}\|(\rho_F, \sigma_F)\|_{L^2_{(sc)}(\Hb)} + \|\alphab_F\|_{L^2_{(sc)}(\Hb)},\\
\Fone(u,\ub) & = \delta^{\frac{1}{2}}\|\nabla_4 \alpha_F\|_{L^2_{(sc)}(H)} + \|\nabla \Upsilon_4\|_{L^2_{(sc)}(H)}, \Foneb(u,\ub) = \delta^{\frac{1}{2}}\|\nabla_3 \alphab_F\|_{L^2_{(sc)}(\Hb)} + \|\nabla \Upsilon_g\|_{L^2_{(sc)}(\Hb)},\\
\Ftwo(u,\ub) & = \delta^{\frac{1}{2}} \|\nabla^2_4 \Upsilon_4\|_{L^2_{(sc)}(H)}  + \|\nabla^2 \Upsilon_4\|_{L^2_{(sc)}(H)}, \Ftwob(u,\ub) = \delta^{\frac{1}{2}} \|\nabla^2_3 \Upsilon_g\|_{L^2_{(sc)}(\Hb)} + \|\nabla^2 \Upsilon_g\|_{L^2_{(sc)}(\Hb)}.
\end{align*}

In those definitions, some terms come up with a $\delta^{\frac{1}{2}}$, for example, $\alpha$ in $\Rzero$. Those terms are understood to \emph{cause a loss of $\delta^{-\frac{1}{2}}$}. We call them \emph{anomalies} or being \emph{anomalous}. By convention, we do not regard $\beta$ and $\beta_F$  as anomalies, since they are not anomalous on outgoing lightcones. In sequel, we shall encounter other terms which cause a loss of $\delta^{-\frac{1}{2}}$. By abuse of language, we will call them \emph{anomalies} too. And, for those non-anomalous terms, in schematic notations, we usually use a $g$ as a subindex, for example, we use $(\nabla \Psi)_g$ to denote $\nabla \alpha$.

To rectify global $L^\infty$ estimate on anomalies, we need localized norms. Let ${}^{\delta}\!S_{u,\ub} \subset S_{u,\ub}$ be a patch of $S_{u,\ub}$ obtained by transporting a disc $S_\delta \subset S_{u,0}$ of radius $\delta^{\frac{1}{2}}$ along integral curves of $L$ and let ${}^{\delta}\!H_u \subset H_u$ be a piece of the hypersurface $H^{(0,\delta)}_u$ obtained by evolving a disc $S_\delta \subset S_{u,0}$ of radius $\delta^{\frac{1}{2}}$ along the integral curves of $L$, we set
\begin{align*}
\Ozerofour^{\delta}[\chih](u,\ub) &= \sup_{{}^{\delta}\!S_{u,\ub} \subset S_{u,\ub}} \|\chih\|_{L^4_{(sc)}({}^{\delta}\!S_{u,\ub})}, \quad \Ozerofour^{\delta}[\chibh](u,\ub) = \sup_{{}^{\delta}\!S_{u,\ub} \subset S_{u,\ub}} \|\chibh\|_{L^4_{(sc)}({}^{\delta}\!S_{u,\ub})},\\
\Rzero^{\delta}[\alpha](u) &= \sup_{{}^{\delta}\!H_u \subset H_u} \|\alpha\|_{L^2_{(sc)}({}^{\delta}\!H_u)}, \quad \Fzero^{\delta}[\alpha_F](u) = \sup_{{}^{\delta}\!H_u \subset H_u} \|\alpha_F\|_{L^2_{(sc)}({}^{\delta}\!H_u)}.
\end{align*}

Let $\OSzerofour = \sup_{u,\ub} \OSzerofour(u,\ub)$ and $\Rzero = \sup_{u,\ub} \Rzero(u,\ub)$; other norms are defined in a similar manner. Finally, we introduce total norms,
\begin{align*}
 \mathcal{O} &= \Ozeroinfinity + \Ozerotwo +\Ozerofour + \Oonetwo + \Oonefour + \OHtwo + \OHbtwo,\\
 \R &= \Rzero+\Rone,  \Rb = \Rzerob + \Roneb, \F = \Fzero + \Fone + \Ftwo, \Fb = \Fzerob + \Foneb + \Ftwob.
\end{align*}

 We use $\Oinitial$, $\Rinitial$ and $\Finitial$ to denote the corresponding total norms on initial surface $H_0$.

\subsection{Main Results}
Let $\psi_i =(\chih, \alpha_F)$. We define initial data norm,
\begin{align*}
\Izero &= \delta^{\frac{1}{2}}\|\psi_i \|_{L^{\infty}(H_0)} +\delta^{\frac{1}{2}} \sup_{0\leq \ub \leq \delta}[ \sum_{k=0}^{2}\|(\delta \nabla_4)^k \psi_i\|_{L^2(S_{0,\ub})}+ \sum_{k=0}^{1}\sum_{m=0}^3 \|(\delta^{\frac{1}{2}}\nabla)^m(\delta \nabla_4)^k \psi_i \|_{L^2(S_{0,\ub})}].
\end{align*}
Our assumption on initial data is
\begin{equation}\label{initial_ansatz_1}
\Izero < \infty.
\end{equation}
This ansatz (when $\alpha_F =0$) was discovered in \cite{K-R-09} and it is larger than those in \cite{Ch}. Unfortunately, the formation of trapped surfaces can not be derived from ansatz \eqref{initial_ansatz_1}. Following \cite{K-R-09}, we shall present necessary modifications in {\bf{Main Theorem}}. Nevertheless, ansatz \eqref{initial_ansatz_1} implies
\begin{proposition} Under the ansatz \eqref{initial_ansatz_1}, along initial outgoing hypersurface $H_0$, if $\delta$ is sufficiently small, we have
\begin{equation*}
\Oinitial +\Rinitial + \Finitial  \lesssim \Izero.
\end{equation*}
\end{proposition}
Thanks to this proposition, we can replace ansatz \eqref{initial_ansatz_1} by
 \begin{equation}\label{initial_ansatz_1_prime}
\Oinitial +\Rinitial + \Finitial  \lesssim \Izero < \infty.
\end{equation}

We omit the proof of this Proposition since the it is very much similar to the vacuum case. In the vacuum case, the reader can find a proof in Chapter 2 of Christodoulou's book \cite{Ch}. In general, the proof is mainly to chase null structure equations, null Bianchi equations and null Maxwell equations one by one along $H_0$. The proof actually shows that $(\chih,\alpha_F)$ can be freely prescribed along $H_0$ for Einstein-Maxwell system \eqref{Einstein_Equations}.

This first result of the paper addresses the propagation of \eqref{initial_ansatz_1_prime} to later time, this is the content of {\bf{Theorem C}}. Roughly speaking, modulo the electromagnetic field, curvatures are one derivative of connection coefficients. Thus, once one controls connection coefficients, one expects to control the curvature components. On the other hand, one can gain the control of connection coefficients provided bounds on curvatures. This later statement is recorded in {\bf{Theorem A}}. Together with {\bf{Theorem B}} which, at the end of Section \ref{Section_Trace_Rotation}, addresses estimates on angular momentums, {\bf{Theorem B}} serves as an intermediate step toward {\bf{Theorem C}}.
\begin{attention}
We use $C$ to denote constants depending only on $\Ozero$, $\R$, $\Rb$, $\F$ and $\Fb$.
\end{attention}
\begin{theoremA}\label{theoremA}
Assume that $\Ozero$, $\R$, $\Rb$, $\F$ and $\Fb$ are finite. If $\delta$ is sufficiently small, we have
\begin{equation*}
\O \lesssim C, \quad \OSzerofour[\chibh] \lesssim \Ozero + C \delta^\frac{1}{4}.
\end{equation*}
\end{theoremA}
\begin{theoremC}\label{theoremC}
Under ansatz \eqref{initial_ansatz_1_prime}, if $\delta$ is sufficiently small, we have
 \begin{equation*}
 \R + \Rb + \F + \Fb \lesssim \Izero.
 \end{equation*}
\end{theoremC}
Finally, we state the main theorem of the paper and its immediate corollaries.
\begin{maintheorem}
In addition to ansatz \eqref{initial_ansatz_1_prime}, we assume that
\begin{equation}\label{initial_ansatz_2}
\sup_{0\leq \ub \leq \delta}\sum_{k=2}^4 \delta^{\frac{1}{2}}\|(\delta^{\frac{1}{2}})^k (\chih,\alpha_F)\|_{L^2(S_{0,\ub})} \leq \varepsilon,
\end{equation}
for sufficiently small $\varepsilon$ such that $0< \delta \ll \varepsilon$. On $H_0$, we also assume
\begin{equation}\label{initial_ansatz_3_formation_ansatz}
 (1+ C_0 \delta^{\frac{1}{2}})\frac{2(r_0-u)}{r_0^2}<\int_0^\delta(|\chih(0,\ub)|^2 + |\alpha_F(0,\ub)|^2) d\ub < \frac{2(r_0-\delta)}{r_0^2},
\end{equation}
where $C_0$ is a universal constant and $r_0$ is the maximal radius of the flat part of $H_0$. Then, if $\delta$ is sufficiently small,  $H_0$ is free of trapped surfaces and a trapped surface forms at $(u,\ub)=(1,\delta)$.
\end{maintheorem}
We can set shear $\chih \equiv 0$ on $H_0$ but $\alpha_F$ large enough to form trapped surfaces. This can be regarded as formation of black hole purely due to the condensation of matter,
\begin{mainCor1}
There exists an initial data set for Maxwell field such that the shear of $H_0$ is trivial, $H_0$ is free of trapped surfaces and a trapped surface forms at $(u,\ub)=(1,\delta)$.
\end{mainCor1}
We can also switch off electromagnetic field to retrieve earlier results for vacuum space-times,
\begin{mainCor2}(\emph{Christodoulou} \cite{Ch}; \emph{Klainerman and Rodnianski}, \cite{K-R-09})
For vacuum, there exists an initial data set on $H_0$ such that $H_0$ is free of trapped surfaces and a trapped surface forms at $(u,\ub)=(1,\delta)$.
\end{mainCor2}





\section{Theorem A - Estimates on connection coefficients}
We prove {\bf{Theorem A}} in this section. The proof is based on following bootstrap assumption,
\begin{equation}\label{bootstrap}
 \OSzeroinfinity + \|(\alpha_F, \rho_F, \sigma_F, \alphab_F)\|_{L^{\infty}_{(sc)}}\leq \Delta_0
\end{equation}
where $\Delta_0$ is a large positive number. We will show that if $\delta$ is sufficiently small, all the statements in the theorem hold. In particular, $\OSzeroinfinity \lesssim C$. We also show $\|\Upsilon\|_{L^{\infty}_{(sc)}}\lesssim C$. Since $C$ is independent of $\Delta_0$, we obtain a better bound in \eqref{bootstrap} therefore close the bootstrap argument. Once we prove {\bf{Theorem A}}, we can use the estimates derived in the proof throughout the paper. We emphasize that \emph{everything} in sequel is based on bootstrap assumption \eqref{bootstrap}. The proof is organized as follows: we first show some preliminary estimates, and then prove step by step $\OSzerofour$,$\OSonetwo$, $\OSonefour$ and second derivatives estimates on connection coefficients.

\subsection{Preliminary Estimates}
In this section, we provide some preliminary estimates based on \eqref{bootstrap} without detailed proof. The reader can find details in Section 4 of \cite{K-R-09}.
\begin{lemma} If $\delta$ is sufficiently small, more precisely if $\delta^\frac{1}{2} \Delta_0  \leq 1$, we have
 \begin{equation*}
   \frac{1}{4} \leq \Omega \leq 4.
 \end{equation*}
\end{lemma}
The proof is straightforward by integrating $\omegab = -\frac{1}{2}\nabla_3 \Omega$. For $p=2$ or $4$, we have following integral estimates,
\begin{lemma}\label{Integral_Estimates}
  For a given horizontal tensor field $\psi$, if $\delta^{\frac{1}{2}} \Delta_0$ is sufficiently small, we have the following scale invariant estimates:
\begin{equation}\label{integralestimate1sc}
 \|\psi\|_{L_{(sc)}^p(u,\ub)} \lesssim \|\psi\|_{L_{(sc)}^p(u,0)} + \int_0^{\ub} \delta^{-1}\|\nabla_4 \psi \|_{L_{(sc)}^p(u,\ub')}d \ub'
\end{equation}
\begin{equation}\label{integralestimate2sc}
 \|\psi\|_{L_{(sc)}^p(u,\ub)} \lesssim \|\psi\|_{L_{(sc)}^p(0,\ub)} + \int_0^{u} \|\nabla_3 \psi \|_{L_{(sc)}^p(u',\ub)}d u'
\end{equation}
\end{lemma}
We define the transported coordinates along double null foliation. For a given local coordinate system $(\theta_1,\theta_2)$ on $S_{u,0} \subset H_u$,\footnote{\, On $S_{u,0}$, the metric is the round metric with constant Gauss curvature.} we parameterize points along the outgoing null geodesics starting from $S_{u,0}$ by corresponding $(\theta_1,\theta_2)$ and the affine parameter $\ub$; similarly, for a given local coordinate system $(\underline{\theta}_1,\underline{\theta}_2)$ on $S_{0,\ub}\subset \Hb_{\ub}$, we parameterize points along the incoming null geodesics starting from $S_{0,\ub}$ by the corresponding $(\theta_1,\theta_2)$ and the affine parameter $u$. Corresponding metrics are denoted by $\gamma_{ab}$ and $\underline{\gamma}_{ab}$. We have estimates on the geometry of $S_{u,\ub}$,
\begin{proposition}
Let $\gamma^0_{ab}$ be the standard metric on $\mathbb{S}^2$. If $\delta$ is sufficiently small, we have
\begin{equation*}
\|(\gamma_{ab} - \gamma^0_{ab}, \underline{\gamma}_{ab} - \gamma^0_{ab}, \nabla_3 \theta^a, \nabla_4 \underline{\theta}^a)\|_{L^{\infty}} \leq \delta^{\frac{1}{2}}\Delta_0, \|\nabla \theta^a\|_{L^{\infty}} + \|\nabla \underline{\theta}^a\|_{L^{\infty}} \lesssim 1.
\end{equation*}
For Christoffel symbols $\Gamma_{abc}$ and $\underline{\Gamma}_{abc}$ whose signatures are $\frac{1}{2}$, we have
\begin{equation*}
\|(\Gamma_{abc}, \underline{\Gamma}_{abc})\|_{L^2(u,\ub)} \lesssim \OSzerofour + \OSonetwo, \quad \|(\partial_d \Gamma_{abc}, \partial_d \underline{\Gamma}_{abc})\|_{L^2(u,\ub)} \lesssim \OSonefour.
\end{equation*}
\end{proposition}
We can follow exactly strategy in Section 4.5 and Section 4.9 of \cite{K-R-09} to prove it. The control of the geometry of $S_{u,\ub}$ at the level of metric and Christoffel symbols allows us to derive the standard Sobolev inequalities, elliptic estimates on Hodge systems and trace inequalities.
\begin{lemma}\label{sobolev_trace_estimates}
 For a horizontal tensor $\psi$, $S=S_{u,\ub}$, $H=H_u$ and $\Hb = \Hb_{\ub}$, if $\delta$ is sufficiently small, we have
\begin{equation}\label{sobolev_L4}
 \|\psi\|_{L^4_{(sc)}(S)} \lesssim \|\psi\|^{\frac{1}{2}}_{{L^2_{(sc)}(S)}}\|\nabla \psi\|^{\frac{1}{2}}_{{L^2_{(sc)}(S)}}+ \delta^{\frac{1}{4}}\|\psi\|_{{L^2_{(sc)}(S)}},
\end{equation}
\begin{equation}\label{sobolev_Linfinity}
 \|\psi\|_{L^{\infty}_{(sc)}(S)} \lesssim \|\psi\|^{\frac{1}{2}}_{L^4_{(sc)}(S)}\|\nabla \psi\|^{\frac{1}{2}}_{L^4_{(sc)}(S)}+ \delta^{\frac{1}{4}}\|\psi\|_{L^4_{(sc)}(S)},
\end{equation}
\begin{equation}\label{trace_H}
 \|\psi\|_{L^4_{(sc)}(S)} \lesssim (\delta^{\frac{1}{2}}\|\psi\|_{L^2_{(sc)}(H)}+\|\nabla \psi\|_{L^2_{(sc)}(H)})^{\frac{1}{2}}(\delta^{\frac{1}{2}}\|\psi\|_{L^2_{(sc)}(H)}+\|\nabla_4 \psi\|_{L^2_{(sc)}(H)})^{\frac{1}{2}},
\end{equation}
\begin{equation}\label{trace_Hb}
 \|\psi\|_{L^4_{(sc)}(S)} \lesssim (\delta^{\frac{1}{2}}\|\psi\|_{L^2_{(sc)}(\Hb)}+\|\nabla \psi\|_{L^2_{(sc)}(\Hb)})^{\frac{1}{2}}(\delta^{\frac{1}{2}}\|\psi\|_{L^2_{(sc)}(\Hb)}+\|\nabla_3 \psi\|_{L^2_{(sc)}(\Hb)})^{\frac{1}{2}},
\end{equation}
\begin{equation}\label{sobolev_local}
 \|\psi\|_{L^{\infty}_{(sc)}(S)} \lesssim \sup_{{}^{\delta}\!S_{u,\ub} \subset S_{u,\ub}}(\|\nabla \psi\|_{L^4_{(sc)}({}^{2\delta}\!S_{u,\ub})}+ \|\psi\|_{L^4_{(sc)}({}^{2\delta}\!S_{u,\ub})}).
\end{equation}
\end{lemma}
Last inequality allows us to rectify the anomalies of $\chih$, $\chibh$ and $\alpha_F$ in $L^\infty_{(sc)}$ estimates. We also have following interpolation estimates,
\begin{lemma}\label{interpolation_estimates} If $\delta$ is sufficiently small, we have
\begin{equation}\label{sobolev_L4_along_outgoing}
\int_{0}^{\ub} \|\Psi\|_{L^4_{(sc)}(u, \ub')}d\ub' \lesssim \delta \|\Psi\|^{\frac{1}{2}}_{L^2_{(sc)}(H_u^{(0,\ub)})}\|\nabla\Psi\|^{\frac{1}{2}}_{L^2_{(sc)}(H_u^{(0,\ub)})}+\delta^{\frac{5}{4}}\|\Psi\|_{L^2_{(sc)}(H_u^{(0,\ub)})},
\end{equation}
\begin{equation}\label{sobolev_L4_along_incoming}
\int_{0}^{u} \|\Psi\|_{L^4_{(sc)}(u', \ub)}d u' \lesssim  \|\Psi\|^{\frac{1}{2}}_{L^2_{(sc)}(\Hb_{\ub}^{(0,u)})}\|\nabla\Psi\|^{\frac{1}{2}}_{L^2_{(sc)}(\Hb_{\ub}^{(0,u)})}+\delta^{\frac{1}{4}}\|\Psi\|_{L^2_{(sc)}(\Hb_{\ub}^{(0,u)})}.
\end{equation}
\end{lemma}
Let $\mathcal{D} \in \{\Done,\Donestar,\Dtwo, \Dtwostar\}$\footnote{\, See Appendix \ref{HodgeOperators} for definitions.} be a Hodge operator, we have following elliptic estimates,
\begin{lemma}\label{estimates_for_hogde_systems}
On $S=S_{u, \ub}$, for a solution $\psi$ of the Hodge system $\D \psi = F$, if $\delta$ is sufficiently small, we have
\begin{equation*}
 \|\nabla \psi\|_{{L^2_{(sc)}(S)}} \lesssim \delta^{\frac{1}{2}}\|K\|_{{L^2_{(sc)}(S)}}\|\psi\|_{{L^2_{(sc)}(S)}} + \|F\|_{{L^2_{(sc)}(S)}},
\end{equation*}
\begin{equation*}
 \|\nabla \psi\|_{{L^2_{(sc)}(S)}} \lesssim \delta^{\frac{1}{2}}\|K\|^{\frac{1}{2}}_{{L^2_{(sc)}(S)}}\|\psi\|_{L^4_{(sc)}(S)} + \|F\|_{{L^2_{(sc)}(S)}},
\end{equation*}
\begin{equation}\label{elliptic_Hogde_2nd_order}
 \|\nabla^2 \psi\|_{{L^2_{(sc)}(S)}} \lesssim \delta^{\frac{1}{2}}\|K\|_{{L^2_{(sc)}(S)}}\|\psi\|_{L^{\infty}_{(sc)}} + \delta^{\frac{1}{4}}\|K\|^{\frac{1}{2}}_{{L^2_{(sc)}(S)}}\|\nabla \psi\|_{L^{4}_{(sc)}(S)} + \|\nabla F\|_{{L^2_{(sc)}(S)}}.
\end{equation}
\end{lemma}

We provide localized estimates on $\alpha$ and $\alpha_F$.
\begin{proposition}If $\delta$ is sufficiently small, then we have
\begin{equation*}
 \Fzero^{\delta}[\alpha_F] \lesssim \Fzero^{\delta}[\alpha](0) + \F, \, \Rzero^{\delta}[\alpha] \lesssim \Rzero^{\delta}[\alpha](0) + \R + \F.
\end{equation*}
\end{proposition}
\begin{proof}
By direct integration, we have
\begin{align*}
 \|\alpha_F\|_{L^2_{(sc)}({}^\delta\!S_{u,\ub})} &\lesssim  \|\alpha_F\|_{L^2_{(sc)}({}^\delta\!S_{0,\ub})} + \int_0^u \|\nabla_3 \alpha_F + \frac{1}{2}\tr\chib \cdot \alpha_F\|_{L^2_{(sc)}({}^\delta\!S_{u',\ub})}
\end{align*}
We use \eqref{NM_Lb_alpha} to replace the last integrand, use H\"older inequality and schematic form of \eqref{bootstrap}.
\begin{align*}
 \|\alpha_F\|_{L^2_{(sc)}({}^\delta\!S_{u,\ub})} &\lesssim  \|\alpha_F\|_{L^2_{(sc)}({}^\delta\!S_{0,\ub})} + \int_0^u  \|(\psi \cdot \Upsilon_g,  \nabla \Upsilon, \chih \cdot \alphab_F)\|_{L^2_{(sc)}({}^\delta\!S_{u',\ub})}\\
&\lesssim \|\alpha_F\|_{L^2_{(sc)}({}^\delta\!S_{0,\ub})} +  \Foneb+\int_0^u \delta^{\frac{1}{2}}\|\psi\|_{L^\infty_{(sc)}} (\|\Upsilon_g\|_{L^2_{(sc)}({}^\delta\!S_{u',\ub})} +  \|\alpha_F\|_{L^2_{(sc)}({}^\delta\!S_{u',\ub})}) \\
&\lesssim \|\alpha_F\|_{L^2_{(sc)}({}^\delta\!S_{0,\ub})} + \Foneb + \delta^{\frac{1}{2}}\Delta_0 \Fzerob + \delta^{\frac{1}{2}}\Delta_0 \int_0^u  \|\alpha_F\|_{L^2_{(sc)}({}^\delta\!S_{u',\ub})}.
\end{align*}
Last term can be absorbed by Gronwall's inequality. This yields the desired estimates for $\alpha_F$.

The proof for $\alpha$ is more or less the same. The key is to use \eqref{NBE_L_alphab}. We also have to pay attention to $\Upsilon$ terms form \eqref{NBE_L_alphab}. By applying null Maxwell equations to convert the null derivatives of a component into either tangential derivatives or lower order terms, those terms can be written as
\begin{align*}
 \int_0^u  \|\psi \cdot \Upsilon \cdot \Upsilon\|_{L^2_{(sc)}({}^\delta\!S_{u',\ub})} + \|\nabla \Upsilon \cdot \Upsilon\|_{L^2_{(sc)}({}^\delta\!S_{u',\ub})}.
\end{align*}
This allows one to use H\"{o}lder's inequality and \eqref{bootstrap} to conclude.
\end{proof}

\subsection{Zeroth Derivative Estimates}
In this section, we derive $\OSzerofour$ as well as $\OSzerotwo$ estimates on connection coefficients and Maxwell field.

\subsubsection{$L^4$ Estimates on connection coefficients}
We first derive estimates on $\omegab$, $\eta$, $\chih$ and $\tr\chi$, then we derive estimates on $\omega$, $\etab$, $\chibh$ and $\trchibt$.

For $\omegab$, in view of \eqref{NSE_L_omegab}, we have $\nabla_4 \omegab = \psi_g \cdot \psi + \Psi_g + \Upsilon_g \cdot {\Upsilon}$, therefore,
\begin{align*}
\|\nabla_4 \omegab\|_{L^4_{(sc)}(u, \ub)} & \lesssim \|\psi_g \cdot \psi\|_{L^4_{(sc)}(u, \ub)} + \|\Psi_g\|_{L^4_{(sc)}(u, \ub)} + \| \Upsilon_g \cdot {\Upsilon}\|_{L^4_{(sc)}(u, \ub)}\\
&\lesssim \delta^{\frac12}  \|\psi\|_{L^\infty_{(sc)}(u, \ub)}  \|\psi_g\|_{L^4_{(sc)}(u, \ub)} +  \|\Psi_g\|_{L^4_{(sc)}(u, \ub)} +  \delta^{\frac12}  \|\Upsilon\|_{L^\infty_{(sc)}(u, \ub)}  \|\Upsilon_g\|_{L^4_{(sc)}(u, \ub)}\\
&\lesssim  \delta^{\frac12} \Delta_0 \cdot(\OSzerofour  + \|\Upsilon_g\|_{L^4_{(sc)}(u, \ub)}) + \|\Psi_g\|_{L^4_{(sc)}(u, \ub)}.
\end{align*}
Thus, by integrating along outgoing directions,
\begin{align*}
\|\omegab\|_{L^4_{(sc)}(u, \ub)} &\lesssim \|\omegab\|_{L^4_{(sc)}(u, 0)} + \int_{0}^{\ub} \delta^{-1} \|\nabla_4 \omegab\|_{L^4_{(sc)}(u, \ub')} \\
&\lesssim \Ozero +   \delta^{\frac12} \Delta_0  \cdot \OSzerofour + \int_{0}^{\ub} \delta^{-1} \|\Psi_g\|_{L^4_{(sc)}(u, \ub')} d\ub' +  \Delta_0 \int_{0}^{\ub} \delta^{-\frac12} \|\Upsilon\|_{L^4_{(sc)}(u, \ub')}.
\end{align*}
In view of \eqref{sobolev_L4_along_outgoing}, one concludes
\begin{align}\label{Os4 estimate for omegab}
\|\omegab\|_{L^4_{(sc)}(u, \ub)} &\lesssim \Ozero +  \delta^{\frac12} \Delta_0 \cdot \, \OSzerofour + \Rzero^{\frac{1}{2}}\Rone^{\frac{1}{2}}+\delta^{\frac{1}{4}}\Rzero +\delta^{\frac{1}{2}}\Delta_0 (\Fzero^{\frac{1}{2}}\Fone^{\frac{1}{2}}+\delta^{\frac{1}{4}}\Fzero)\notag\\
&\lesssim \Ozero +  \delta^{\frac12} \Delta_0 \cdot \, \OSzerofour + C.
\end{align}

For $\eta$, similarly, we can show
\begin{equation}\label{Os4 estimate for eta}
\|\eta\|_{L^4_{(sc)}(u, \ub)} \lesssim \Ozero +   \delta^{\frac12} \Delta_0 \cdot \, \OSzerofour + C.
\end{equation}

For $\chih$, according to \eqref{NSE_L_chih}, we have $\nabla_4 \chih = \psi_g \cdot \psi + \alpha$. In view of \eqref{sobolev_L4_along_outgoing} and the fact that $\chih$ has trivial data along $\Hb_{0}$, we have
\begin{align*}
\|\chih\|_{L^4_{(sc)}(u, \ub)} &\lesssim \|\chih\|_{L^4_{(sc)}(u, 0)}+   \delta^{\frac12} \Delta_0 \, \OSzerofour + \int_{0}^{\ub} \delta^{-1} \|\alpha\|_{L^4_{(sc)}(u, \ub')} \\
&\lesssim \delta^{\frac12} \Delta_0 \, \OSzerofour+  \|\alpha\|^{\frac{1}{2}}_{L^2_{(sc)}(H_u^{(0,\ub)})}\|\nabla \alpha\|^{\frac{1}{2}}_{L^2_{(sc)}(H_u^{(0,\ub)})}+\delta^{\frac{1}{4}}\|\alpha\|_{L^2_{(sc)}(H_u^{(0,\ub)})}.
\end{align*}
We multiply both sides by $\delta^{\frac{1}{4}}$ to derive
\begin{equation}\label{Os4 estimate for chih}
\OSzerofour[\chih] \lesssim \delta^{\frac{3}{4}} \Delta_0 \, \OSzerofour + \Rzero[\alpha]^{\frac{1}{2}}\Rone[\alpha]^{\frac{1}{2}}+\Rzero[\alpha].
\end{equation}

We also need localized estimates for $\chih$. The procedure is similar,
\begin{align*}
\|\chih\|_{L^4_{(sc)}({}^\delta\!S_{u,\ub})} &\lesssim \delta^{\frac12} \Delta_0 \OSzerofour + \int_{0}^{\ub} \delta^{-1} \|\alpha\|_{L^4_{(sc)}({}^\delta\!S_{u,\ub'})}\\
&\lesssim \delta^{\frac12} \Delta_0 \OSzerofour + \|\alpha\|_{L^2_{(sc)}({}^{2\delta}H_u^{(0,\ub)})} + \|\nabla \alpha\|_{L^2_{(sc)}({}^{2\delta}H_u^{(0,\ub)})},
\end{align*}
Hence,
\begin{equation}\label{Os4 estimate for chih localized}
\|\chih\|_{L^4_{(sc)}({}^\delta\!S_{u,\ub})} \lesssim \delta^{\frac12} \Delta_0 \OSzerofour + \Rzero^\delta[\alpha] + \Rone[\alpha].
\end{equation}

For $\tr \chi$, in view of \eqref{NSE_L_tr_chi}, we derive
\begin{align*}
\|\tr \chi\|_{L^4_{(sc)}(u, \ub)} &\lesssim \Ozero +   \delta^{\frac12} \Delta_0 \, \OSzerofour +  \int_{0}^{\ub} \delta^{-1} \||\chih|^2 +|\alpha_F|^2 \|_{L^4_{(sc)}(u, \ub')}\\
 &\lesssim \Ozero +   \delta^{\frac12} \Delta_0 \OSzerofour +  \delta^{-\frac{1}{2}}\Delta_0\int_{0}^{\ub} \|\chih \|_{L^4_{(sc)}(u, \ub')} + \|\alpha_F \|_{L^4_{(sc)}(u, \ub')}.
\end{align*}
Combining \eqref{Os4 estimate for chih} and \eqref{sobolev_L4_along_outgoing} for $\alpha_F$, we derive
\begin{equation}\label{Os4 estimate for trchi}
\|\tr \chi\|_{L^4_{(sc)}(u, \ub)} \lesssim \Ozero +   \delta^{\frac{1}{4}} \Delta_0 \OSzerofour +C.
\end{equation}

For $\etab$, in view \eqref{NSE_Lb_etab}, we have $\nabla_3 \etab
= \tr \chib_0 \cdot \etab +\psi_g \cdot \psi + \Psi_g + \Upsilon_g \cdot \Upsilon
$. We bound $\tr \chib_0$ by a constant to derive
\begin{equation*}
\|\nabla_3 \etab\|_{L^4_{(sc)}(u, \ub)} \lesssim  \|\etab\|_{L^4_{(sc)}(u, \ub)} + \delta^{\frac12} \Delta_0 (\OSzerofour + \| \Upsilon\|_{L^4_{(sc)}(u, \ub)})+  \|\Psi_g\|_{L^4_{(sc)}(u, \ub)}.
\end{equation*}
We also observe that $\Psi_g \neq \beta$, thus,
\begin{align*}
\|\etab\|&_{L^4_{(sc)}(u, \ub)} \lesssim \|\etab\|_{L^4_{(sc)}(0, \ub)} + \int_{0}^{u} \|\nabla_3 \etab\|_{L^4_{(sc)}(u', \ub)}\\
&\lesssim \Ozero +   \delta^{\frac12} \Delta_0  \OSzerofour + \int_0^u \|\etab\|_{L^4_{(sc)}(u', \ub)} + \|\Psi_g\|_{L^4_{(sc)}(u', \ub)} +\Delta_0 \delta^{\frac12} \|\Upsilon\|_{L^4_{(sc)}(u', \ub)}.
\end{align*}
Thanks to Gronwall's inequality, we remove $\eta$ on the right-hand side,
\begin{equation}\label{Os4 estimate for etab}
\|\etab\|_{L^4_{(sc)}(u, \ub)} \lesssim \Ozero +   \delta^{\frac12} \Delta_0  \OSzerofour + C.
\end{equation}

For $\omega$, similarly, we derive
\begin{equation}\label{Os4 estimate for omega}
\|\omega\|_{L^4_{(sc)}(u, \ub)} \lesssim \Ozero +   \delta^{\frac12} \Delta_0 \OSzerofour + C.
\end{equation}

For $\chibh$, according to \eqref{NSE_Lb_chibh}, $\nabla_3 \chibh= \tr\chib_0 \, \chibh  + \psi_g \cdot \psi -\alphab$, thus,
\begin{equation*}
\|\nabla_3 \chibh\|_{L^4_{(sc)}(u, \ub)} \lesssim \|\chibh\|_{L^4_{(sc)}(u, \ub)} + \delta^{\frac12} \Delta_0 \OSzerofour +  \|\alphab\|_{L^4_{(sc)}(u, \ub)}.
\end{equation*}
As in \eqref{Os4 estimate for etab}, by Gronwall's inequality, we ignore $\|\chibh\|_{L^4_{(sc)}(u, \ub)}$. Therefore,
\begin{align*}
\|\chibh\|_{L^4_{(sc)}(u, \ub)} & \lesssim \|\chibh\|_{L^4_{(sc)}(0, \ub)}+   \delta^{\frac12} \Delta_0  \OSzerofour  + \int_{0}^{u} \|\alphab\|_{L^4_{(sc)}(u', \ub)}\lesssim \delta^{-\frac{1}{4}}\Ozero+   \delta^{\frac12} \Delta_0 \OSzerofour  + C.
\end{align*}
Thus,
\begin{equation}\label{Os4 estimate for chibh}
\OSzerofour[\chibh] \lesssim \Ozero+   \delta^{\frac{3}{4}} \Delta_0 \OSzerofour  + C\delta^{\frac{1}{4}}.
\end{equation}
We repeat above procedure to derive localized estimates,
\begin{equation}\label{Os4 estimate for omegab localized}
\|\chibh\|_{L^4_{(sc)}({}^\delta\!S_{u,\ub})}\lesssim \|\chibh\|_{L^4_{(sc)}({}^\delta\!S_{0,\ub})} + \delta^{\frac{1}{2}} \Delta_0 \OSzerofour +C.
\end{equation}

For $\trchibt$, since $\nabla_3 u = \Omega^{-1}$ and $\nabla_3 \ub =0$, we have $\nabla_3 \tr\chib_0  = -\frac{1}{4\Omega}(\tr\chib_0) ^2$. Thus, \eqref{NSE_Lb_tr_chib} implies,
\begin{equation*}
\nabla_3 \trchibt = -\tr \chib_0 \cdot \trchibt - \frac{1}{2\Omega}(\Omega-\frac{1}{2}) (\tr\chib_0) ^2 - 2\tr\chib_0 \omegab -\frac{1}{2}(\trchibt)^2-2\omegab \trchibt -|\chibh|^2-|\alphab_F|^2.
\end{equation*}
For the second term, we have $\frac{1}{2\Omega}(\Omega-\frac{1}{2}) = \frac{1}{4}\int_0^u \omegab $, therefore,
\begin{equation*}
\| \frac{1}{2\Omega}(\Omega-\frac{1}{2}) \|_{L^4_{(sc)}(u, \ub)} \lesssim \int_0^u \|\omegab\|_{L^4_{(sc)}(u', \ub)}.
\end{equation*}
The double anomaly $|\chibh|^2$ causes a loss of $\delta^{\frac{1}{4}}$,
\begin{equation*}
\||\chibh|^2\|_{L^4_{(sc)}(u, \ub)} \lesssim \delta^{\frac{1}{2}}  \|\chibh\|_{L^{\infty}_{(sc)}(u, \ub)} \|\chibh\|_{L^4_{(sc)}(u, \ub)}\lesssim \delta^{\frac{1}{4}}\Delta_0 \, \OSzerofour.
\end{equation*}
Combined with the bound of $\omegab$ in \eqref{Os4 estimate for omegab}, we have
\begin{align*}
\|\nabla_3 \trchibt \|_{L^4_{(sc)}(u, \ub)} &\lesssim \|\trchibt \|_{L^4_{(sc)}(u, \ub)}+ \|\omegab\|_{L^4_{(sc)}(u, \ub)}+\int_0^u \|\omegab\|_{L^4_{(sc)}(u', \ub)}  + \delta^{\frac{1}{4}}\Delta_0 \OSzerofour\\
&\lesssim  \|\trchibt \|_{L^4_{(sc)}(u, \ub)} + \Ozero +\delta^{\frac{1}{4}}\Delta_0  \OSzerofour + C.
\end{align*}
Thanks to Gronwall's inequality, we ignore $\|\trchibt \|_{L^4_{(sc)}(u, \ub)}$ on the right-hand side. Thus,
\begin{equation}\label{Os4 estimate for trchibt}
\|\trchibt \|_{L^4_{(sc)}(u, \ub)} \lesssim  \Ozero +\delta^{\frac{1}{4}}\Delta_0  \OSzerofour + C.
\end{equation}

We add \eqref{Os4 estimate for omegab}, \eqref{Os4 estimate for eta}, \eqref{Os4 estimate for chih}, \eqref{Os4 estimate for trchi}, \eqref{Os4 estimate for etab}, \eqref{Os4 estimate for omega}, \eqref{Os4 estimate for chibh} and \eqref{Os4 estimate for trchibt} together to derive
\begin{equation*}
\OSzerofour \lesssim \Ozero +\delta^{\frac{1}{4}}\Delta_0  \OSzerofour + C.
\end{equation*}
If $\delta$ is sufficiently small, $\delta^{\frac{1}{4}}\Delta_0  \OSzerofour$ is absorbed by the left hand side. Next proposition summarizes the estimates derived in this subsection,
\begin{proposition}\label{OZERO_FOUR_ESTIMATES} If $\delta$ is sufficiently small, we have
\begin{equation*}
\OSzerofour \lesssim C.
\end{equation*}
Moreover, we have
\begin{equation}\label{ozerofour_for_chih}
\OSzerofour[\chih] \lesssim   \Rzero[\alpha]^{\frac{1}{2}}\Rone[\alpha]^{\frac{1}{2}}+\Rzero[\alpha] + \delta^{\frac{1}{4}} C, \,\, \OSzerofour[\chibh] \lesssim \Ozero+   \delta^{\frac{1}{4}} C,
\end{equation}
and their localized versions
\begin{equation}\label{localized_L_4_chih}
\|\chih\|_{L^4_{(sc)}({}^\delta\!S_{u,\ub})}\lesssim \delta^{\frac{1}{4}} C  + \Rzero^\delta[\alpha] + \Rone[\alpha], \,\,\|\chibh\|_{L^4_{(sc)}({}^\delta\!S_{u,\ub})}\lesssim \|\chibh\|_{L^4_{(sc)}({}^\delta\!S_{0,\ub})} + C.
\end{equation}
\end{proposition}

\subsubsection{$L^4$ Estimates on Maxwell Field}
We introduce $\OSzerofour$ norms for Maxwell field,
\begin{align*}
\OSzerofour[\alpha_F] (u,\ub) &=\delta^{\frac{1}{4}}\|\alpha_F\|_{L^4_{(sc)}(u,\ub)}, \!
\OSzerofour[(\rho_F, \sigma_F, \alphab_F)] (u,\ub)  =\|(\rho_F, \sigma_F, \alphab_F)\|_{L^4_{(sc)}(u,\ub)}.
\end{align*}

For $\Upsilon_g$, in view of \eqref{NM_L_alphab}, \eqref{NM_L_rho} and \eqref{NM_L_sigma}, we have $\nabla_4 \Upsilon_g = \nabla \Upsilon + \psi \cdot \Upsilon$.
In view of \eqref{bootstrap} and Proposition \ref{OZERO_FOUR_ESTIMATES}, we bound $\Upsilon$ in $L^\infty_{(sc)}$ and $\psi$ in  $L^4_{(sc)}$. Thus,
\begin{equation*}
 \|\nabla_4 \Upsilon_g\|_{L^4_{(sc)}(u, \ub)} \lesssim \|\nabla \Upsilon\|_{L^4_{(sc)}(u, \ub)} + \delta^{\frac{1}{2}}\Delta_0 \| \psi\|_{L^4_{(sc)}(u, \ub)} \lesssim \|\nabla \Upsilon\|_{L^4_{(sc)}(u, \ub)} + C \delta^{\frac{1}{4}}\Delta_0.
\end{equation*}
Thus, by direct integration,
\begin{align}\label{maxwell_L4_1}
 \|\Upsilon_g\|_{L^4_{(sc)}(u, \ub)} &\lesssim \|\Upsilon\|_{L^4_{(sc)}(u, 0)} + C \delta^{\frac{1}{4}}\Delta_0+ \delta^{-1}\int_0^{\ub} \|\nabla \Upsilon\|_{L^4_{(sc)}(u, \ub')}\notag\\
&\lesssim \Ozero + C \delta^{\frac{1}{4}}\Delta_0 + \Ftwo[\Upsilon]^{\frac{1}{2}}\Fone[\Upsilon]^{\frac{1}{2}} + \delta^{\frac{1}{4}}\Fone[\Upsilon] \lesssim C.
\end{align}

For $\alpha_F$, we rewrite \eqref{NM_Lb_alpha} as $ \nabla_3 \alpha_F = \tr\chib_0 \cdot \alpha_F + \nabla \Upsilon + \psi \cdot \Upsilon$. Thus,
\begin{align*}
 \|\nabla_3 \alpha_F\|_{L^4_{(sc)}(u, \ub)} &\lesssim \| \alpha_F\|_{L^4_{(sc)}(u, \ub)} + \|\nabla \Upsilon\|_{L^4_{(sc)}(u, \ub)} + \delta^{\frac{1}{2}}\Delta_0 \|\psi\|_{L^4_{(sc)}(u, \ub)}\\
&\lesssim \| \alpha_F\|_{L^4_{(sc)}(u, \ub)} + \|\nabla \Upsilon\|_{L^4_{(sc)}(u, \ub)} + C\delta^{\frac{1}{4}}\Delta_0.
\end{align*}
We ignore $\| \alpha_F\|_{L^4_{(sc)}(u, \ub)}$ on the right-hand side by Gronwall's inequality to derive
\begin{align*}
 \|\alpha_F\|_{L^4_{(sc)}(u, \ub)} &\lesssim \|\alpha_F\|_{L^4_{(sc)}(0, \ub)} + \int_0^{u} \|\nabla \Upsilon\|_{L^4_{(sc)}(u', \ub)}  + C \delta^{\frac{1}{4}}\Delta_0\\
&\lesssim \|\alpha_F\|_{L^4_{(sc)}(0, \ub)}  + \Ftwob[\Upsilon]^{\frac{1}{2}}\Foneb[\Upsilon]^{\frac{1}{2}} + \delta^{\frac{1}{2   }}\Foneb[\Upsilon] + C\delta^{\frac{1}{4}} \Delta_0.
\end{align*}
This implies
\begin{equation}\label{maxwell_L4_2}
 \OSzerofour[\alpha_F](u,\ub) \lesssim \OSzerofour[\alpha_F](0,\ub)  + C \delta^{\frac{1}{4}}.
\end{equation}
We can also localize the proof to derive
\begin{equation}\label{maxwell_L4_3}
 \|\alpha_F\|_{L^4_{(sc)}({}^\delta\!S_{u,\ub})} \lesssim \|\alpha_F\|_{L^4_{(sc)}({}^\delta\!S_{0,\ub})}  + C.
\end{equation}
Next proposition summarizes the estimates derived in this subsection,
\begin{proposition}
If $\delta$ is sufficiently small, we have
\begin{equation}
 \OSzerofour[\alpha_F, \rho_F, \sigma_F, \alphab_F]\lesssim C.
\end{equation}
Moreover, we have
\begin{equation}\label{local alpha_F}
 \|\alpha_F\|_{L^4_{(sc)}({}^\delta\!S_{u,\ub})} \lesssim \delta^{-\frac{1}{4}}\Ozero  + C, \,\,  \|\alpha_F\|_{L^4_{(sc)}({}^\delta\!S_{u,\ub})} \lesssim \|\alpha_F\|_{L^4_{(sc)}({}^\delta\!S_{0,\ub})}  + C.
\end{equation}
\end{proposition}

\subsubsection{$L^2$ Estimates}
We move on to $\OSzerotwo$ estimates on $\psi$, $\Psi$ and $\Upsilon$. It is straightforward since we have already established $\Ozerofour$ estimates. Since the procedure is standard, we will only sketch the  proof.
\begin{proposition}
If $\delta$ is sufficiently small, we have
\begin{equation*}
\|(\Psi_g, \Upsilon_g, K)\|_{{L^2_{(sc)}(u,\ub)}} + \Ozerotwo\lesssim C, \,\, \|(\alpha,\alpha_F)\|_{{L^2_{(sc)}(u,\ub)}} \lesssim \delta^{-\frac{1}{2}}C.
\end{equation*}
\end{proposition}
\begin{proof}
For $\alpha$, in view of \eqref{NBE_Lb_alpha}, we have $\nabla_3 \alpha = \tr\chib_0 \cdot \alpha + \nabla \Psi + \psi \cdot \Psi + \Upsilon \cdot \nabla \Upsilon$.
We proceed exactly as before:  integrating along $\Hb_{\ub}$ and removing $\tr\chib_0 \cdot \alpha$ by Gronwall's inequality. This yields
\begin{equation*}
 \|\alpha\|_{{L^2_{(sc)}(S_{u,\ub})}} \lesssim  \|\alpha\|_{L^2_{(sc)}(0, \ub)} + C.
\end{equation*}
For $\|\alpha\|_{L^2_{(sc)}(0, \ub)}$, we can bound it as follows,
\begin{equation*}
\|\alpha\|_{L^2_{(sc)}(0, \ub)} \lesssim \delta^{-1}\int_0^{\ub}\|\nabla_4 \alpha\|_{L^2_{(sc)}(0, \ub')} \lesssim \delta^{-\frac{1}{2}} \Rone[\alpha](0),
\end{equation*}
which yields the estimates on $\alpha$.

For $\Psi_g$, in view of \eqref{NBE_L_beta}, \eqref{NBE_L_sigma},  \eqref{NBE_L_rho}, \eqref{NBE_L_betab} and \eqref{NBE_L_alphab},  $\nabla_4 \Psi_g = \nabla \Psi + \psi \cdot \Psi + \Upsilon\cdot \nabla \Upsilon$. This allows one to derive the bound for $\Psi_g$.

For Gaussian curvature $K$, according to \eqref{NSE_gauss}, we have $K = \rho + \psi \cdot \psi + \Upsilon \cdot \Upsilon$. Therefore,
\begin{equation*}
\| K \|_{{L^2_{(sc)}(S_{u,\ub})}} \lesssim \|\rho\|_{{L^2_{(sc)}(S_{u,\ub})}} + \delta^{\frac{1}{2}}\|(\psi, \Upsilon)\|^2_{{L^4_{(sc)}(S_{u,\ub})}}\lesssim C + \OSzerofour^2 +\OSzerofour^2[\Upsilon]^2 \lesssim C.
\end{equation*}

For $\psi_4 \in \{\tr\chi,\, \chih, \,\omegab, \,\eta\}$, the corresponding null structure equations read as $\nabla_4 \psi_4 = \psi \cdot\psi + \Psi +\Upsilon\cdot \Upsilon$. Thus,
\begin{equation*}
\|\nabla_4 \psi\|_{L^2_{(sc)}(u, \ub)} \lesssim \delta^{\frac{1}{2}}\|(\psi, \Upsilon)\|^2_{L^4_{(sc)}(u, \ub)} + \|\Psi\|_{L^2_{(sc)}(u, \ub)} \lesssim C + \|\Psi\|_{L^2_{(sc)}(0, \ub)}.
\end{equation*}
We then integrate over $H_u$ to derive
\begin{equation*}
\|\psi\|_{{L^2_{(sc)}(S_{u,\ub})}} \lesssim  C + \|\Psi\|_{L^2_{(sc)}(H_u^{(0,\ub)})}.
\end{equation*}
One key observation is, in above expressions,  $\Psi$ is not anomalous unless $(\psi, \Psi) = (\chih, \alpha)$. Thus,
\begin{equation*}
\|(\tr\chi, \, \omegab, \, \eta)\|_{{L^2_{(sc)}(S_{u,\ub})}} \lesssim  C^2 + \Rzero, \quad \|\chih\|_{{L^2_{(sc)}(S_{u,\ub})}} \lesssim  C^2 + \delta ^{-\frac{1}{2}}\Rzero[\alpha].
\end{equation*}

We also prove estimates on $\chibh$, $\trchibt$, $\omega$ and $\etab$ in the same manner.  The derivation of estimates on $\Upsilon$ is also similar. This completes the proof.
\end{proof}

\subsection{First Derivative Estimates}

\subsubsection{Transport-Hodge System for connection coefficients}
We use $(\Theta, \psi)$ to denote one of pairs
\begin{equation*}
\{(\nabla \tr\chi, \chih), (\nabla \trchibt, \chibh); (\mu, \eta), (\mub,  \etab); (\kappa, \langle \omega \rangle), (\kappab, \langle \omegab \rangle) \}.
\end{equation*}
The definition of $\mu$, $\mub$, $\langle \omega \rangle$ and $\langle \omegab \rangle$ will be given in the context. We will show that $(\Theta, \psi)$ satisfies a Transport-Hodge type system which has following schematic form,
\begin{equation}\label{L_transport_one_derivative}
\nabla_4 \Theta = \psi \cdot (\nabla \psi + \Psi + \Upsilon\cdot\Upsilon) + \tr \chib_0 \cdot \psi \cdot \psi_g +\psi \cdot \psi \cdot \psi_g + \nabla\Upsilon \cdot \Upsilon,
\end{equation}
\begin{equation}\label{Lb_transport_one_derivative}
\nabla_3 \Theta =\tr \chib_0 \cdot \nabla \psi + \psi \cdot (\nabla \psi + \Psi + \Upsilon\cdot\Upsilon) + \tr \chib_0 \cdot \psi \cdot \psi_g +\psi \cdot \psi \cdot \psi_g + \nabla\Upsilon \cdot \Upsilon,
\end{equation}
\begin{equation}\label{Hodge_one_derivative}
\D \psi = \Theta + \Psi_g + \tr\chib_0 \cdot \psi_g + \psi \cdot \psi + \Upsilon \cdot \Upsilon. \footnote{\, $\D$ is a Hodge operator.}
\end{equation}

For $(\Theta, \psi) = (\nabla \tr\chi, \chih)$, in view of \eqref{NSE_div_chih}, $\psi = \chih$ satisfies \eqref{Hodge_one_derivative}. For $\Theta = \nabla \tr \chi$, we commute $\nabla$ with \eqref{NSE_L_tr_chi} to derive
\begin{align*}
\nabla_4 \Theta &= [\nabla_4, \nabla] \tr \chi + \nabla \nabla_4 \tr \chi =\psi_g\cdot \nabla_4 \tr \chi + \psi \cdot \nabla \tr \chi + \nabla \nabla_4 \tr \chi\\
&=\psi_g\cdot (\psi \cdot\psi + \Upsilon \cdot\Upsilon) + \psi\cdot \nabla \psi + \nabla (\psi\cdot \psi + \Upsilon \cdot\Upsilon).
\end{align*}
which is clearly in the form of \eqref{L_transport_one_derivative}.

For $(\Theta, \psi) = (\nabla \trchibt, \chibh)$, similarly, we can show that it satisfies \eqref{Lb_transport_one_derivative} and \eqref{Hodge_one_derivative}.

For $(\Theta, \psi) = (\mu, \eta)$, we first define the mass aspect function
\begin{equation}\label{mu}
\mu = -\divergence \eta - \rho.
\end{equation}
In view of \eqref{NSE_L_eta}, we have $\nabla_4 \eta = \psi_g \cdot \psi -\beta + \Upsilon\cdot \Upsilon$. We remark that we must keep tracking the exact coefficient for $\beta$. Thus, we can derive
\begin{align*}
\nabla_4 (\divergence \eta) &=[\nabla_4, \divergence] \eta + \divergence \nabla_4 \eta =\psi \cdot \nabla \eta + \psi_g \cdot \nabla_4 \eta + (\Psi_g + \Upsilon\cdot\Upsilon)\cdot\eta +\psi_g \cdot \psi \cdot \eta +\divergence \nabla_4 \eta\\
&=\psi \cdot \nabla \psi + \psi \dot (\Psi_g + \Upsilon \cdot\Upsilon)+\psi_g \cdot\nabla_4 \eta + \divergence \nabla_4 \eta\\
&=\psi\cdot \nabla \psi + \psi\cdot(\Psi_g + \Upsilon \cdot \Upsilon)+\psi_g \cdot\psi\cdot \psi + \divergence (\psi_g  \cdot\psi -\beta + \Upsilon\cdot\Upsilon).
\end{align*}
Thus,
\begin{equation}\label{mu1}
\nabla_4 (\divergence \eta) = -\divergence \beta + \psi \cdot\nabla \psi + \psi\cdot(\Psi_g + \Upsilon\cdot\Upsilon)+\psi_g \cdot\psi\cdot \psi +\Upsilon \cdot \nabla \Upsilon.
\end{equation}
In view of \eqref{NBE_L_rho}, we have
\begin{equation}\label{mu2}
\nabla_4 \rho = \divergence \beta + \psi\cdot \Psi + \Upsilon \cdot \nabla \Upsilon.
\end{equation}
We also record the exact coefficient for $\beta$ in\eqref{mu2}. Adding up \eqref{mu}, \eqref{mu1} and \eqref{mu2},
\begin{equation*}
\nabla_4 \mu = \psi \cdot (\nabla \psi + \Psi_g +\Upsilon \cdot \Upsilon ) + \psi \cdot \psi \cdot \psi +\Upsilon \cdot \nabla \Upsilon.
\end{equation*}
Hence, $\Theta = \mu$ satisfies \eqref{L_transport_one_derivative}. We also combine \eqref{mu} with \eqref{NSE_curl_eta} to derive,
\begin{equation*}
\divergence \eta =-\mu-\rho = -\Theta - \Psi_g , \,\, \curl\eta = \chibh \wedge \chih +\sigma = \Psi_g + \psi \cdot \psi.
\end{equation*}
So $\psi = \eta$ satisfies \eqref{Hodge_one_derivative}.

For $(\Theta, \psi) = (\mub, \etab)$, we define the mass aspect function $\mub = -\divergence \etab - \rho$. Similarly, we can show that $(\mub, \etab)$ satisfies \eqref{Lb_transport_one_derivative} and \eqref{Hodge_one_derivative}.

For $(\Theta, \psi) = (\kappab, \langle \omegab \rangle)$ or $(\kappa, \langle \omega \rangle)$, we first define $\langle \omega \rangle$ and $\langle \omegab \rangle$. We introduce two auxiliary functions $\omegabt$ and $\omegat$ defined by following transport equations
\begin{equation}
\nabla_3 \omegat = \frac{1}{2}\sigma, \,\omegat = 0 \,\,\text{on}\,\, H_0;\, \nabla_4 \omegabt = \frac{1}{2}\sigma,\, \omegabt = 0 \,\,\text{on}\,\, \Hb_0.
\end{equation}
with initial data $\omegat = 0$ on $H_0$ and $\omegabt = 0$ on $\Hb_0$. We then define
 \begin{equation}
 \langle \omegab \rangle = (\omegab, \omegabt), \quad \langle \omega \rangle = (-\omega, \omegat).
 \end{equation}
\begin{remark}
We enlarge the set of connection coefficients by adding two non-anomalous scalars $\omegabt$ and $\omegat$. We also extend $\mathcal{O}$-norms as well as \eqref{bootstrap} to include them. It is easy to show that the $\Ozerofour$ and $\Ozerotwo$ estimates hold for $\omegabt$ and $\omegat$.
\end{remark}
Since $\Donestar$ acts as $\Donestar \omegabpair = -\nabla \omegab + \nablastar \omegabt$, $\Donestar \omegapair = -\nabla \omega + \nablastar \omegat$,
in view of \eqref{NSE_L_omegab}, we derive
\begin{align}\label{kappab_1}
\nabla_4 \,\Donestar \omegabpair
&=\Donestar (\frac{1}{2}\rho + \psi_g \cdot \psi_g + \Upsilon \cdot \Upsilon, \frac{1}{2}\sigma) + \psi \cdot \nabla(\omegab, \omegabt) + \psi_g \cdot(\nabla_4 \omegab + \nabla_4 \omegabt)\notag\\
& =\frac{1}{2}\,\Donestar (\rho, \sigma) + \psi\cdot(\nabla \psi +\Psi_g +\Upsilon \cdot \Upsilon)  +\psi_g \cdot \psi\cdot \psi +\Upsilon\cdot \nabla \Upsilon.
\end{align}
We have to keep tracking of the exact coefficient of $\Donestar (\rho, \sigma)$. In view of \eqref{NBE_L_betab},
\begin{equation}\label{kappab_2}
\nabla_4 \betab = \Donestar(\rho,\sigma)+\psi \cdot \Psi+\Upsilon \cdot \nabla \Upsilon.
\end{equation}
If we introduce mass aspect function
\begin{equation}\label{kappab}
\kappab = \Donestar \omegabpair -\frac{1}{2}\betab = -\nabla \omegab + \nablastar \omegabt  -\frac{1}{2}\betab.
\end{equation}
By suitably adding up \eqref{kappab_1} and \eqref{kappab_2}, we derive
\begin{equation}
\nabla_4 \kappab =  \psi \cdot (\nabla \psi +\Psi_g +\Upsilon \cdot \Upsilon)  +\psi_g \cdot \psi \cdot \psi +\Upsilon \cdot \nabla \Upsilon,
\end{equation}
which is clearly of type \eqref{L_transport_one_derivative}. By definition, we also have
\begin{equation}
\Donestar \omegabpair = \kappab + \frac{1}{2}\betab,
\end{equation}
which is clearly of type \eqref{Hodge_one_derivative}. Hence, we showed that $(\kappab, \langle \omegab \rangle)$ satisfies \eqref{L_transport_one_derivative} and \eqref{Hodge_one_derivative}.

Similarly, by using the mass aspect function
\begin{equation}\label{kappa}
\kappa = \Donestar \omegapair -\frac{1}{2}\beta= \nabla \omega + \nablastar \omegat  -\frac{1}{2}\beta,
\end{equation}
we can show that $(\kappa, \langle \omega \rangle)$ satisfies \eqref{Lb_transport_one_derivative} and \eqref{Hodge_one_derivative}.

\subsubsection{First Derivative Estimates on connection coefficients}
For $\Theta \in \{\nabla \tr\chi, \mu, \kappab\}$, they satisfy \eqref{L_transport_one_derivative}. On $S_{u,\ub}$, we bound $\psi$ or $\Upsilon$ in $L^\infty_{(sc)}$ norm to derive
\begin{align*}
\|\nabla_4 \Theta\|_{{L^2_{(sc)}}} & \lesssim \delta^{\frac{1}{2}}\Delta_0(\|\nabla \psi\|_{{L^2_{(sc)}}}+ \|\nabla \Upsilon\|_{{L^2_{(sc)}}} + \|\Psi\|_{{L^2_{(sc)}}} + \|\psi_g\|_{{L^2_{(sc)}}}+ \delta^{\frac{1}{2}}\Delta_0\|\Upsilon\|_{{L^2_{(sc)}}})\\
&\lesssim \delta^{\frac{1}{2}}\Delta_0 \|\nabla \psi\|_{{L^2_{(sc)}(S_{u,\ub})}} +\delta^{\frac{1}{2}}\Delta_0 \|\nabla \Upsilon\|_{{L^2_{(sc)}(S_{u,\ub})}}+ C.
\end{align*}
Since $\Theta$'s are trivial along $\Hb_0$, we have
\begin{align}\label{est_L}
\|\Theta\|_{{L^2_{(sc)}(S_{u,\ub})}} &\lesssim \int_{0}^{\ub} \delta^{-1} \|\nabla_4 \Theta\|_{L^2_{(sc)}(u, \ub')} \lesssim C + \delta^{\frac{1}{2}} \Delta_0 \cdot \Oonetwo  +\delta^{\frac{1}{2}} \Delta_0 \int_{0}^{\ub} \delta^{-1} \|\nabla \Upsilon\|_{L^2_{(sc)}(u, \ub')}\notag\\
&\lesssim C + \delta^{\frac{1}{2}}\Delta_0 \cdot \Oonetwo.
\end{align}

For $\Theta \in \{\nabla \trchibt, \mub, \kappa\}$, similarly, we use \eqref{Lb_transport_one_derivative} to derive estimates. The presence of $\tr \chib_0 \cdot \nabla \psi$ in \eqref{Lb_transport_one_derivative} leads to a bit more complicated estimates,
\begin{equation}\label{est_Lb}
\|\Theta\|_{{L^2_{(sc)}(S_{u,\ub})}} \lesssim C + \delta^{\frac{1}{2}}\Delta_0 \cdot \Oonetwo + \int_{0}^{u} \| \nabla \psi \|_{L^2_{(sc)}(u', \ub)} .
\end{equation}

We turn to the Hodge system \eqref{Hodge_one_derivative}, in view of Lemma \ref{estimates_for_hogde_systems}, on $S_{u,\ub}$, we have
\begin{align*}
\| \nabla \psi \|_{{L^2_{(sc)}}} &\lesssim \delta^{\frac{1}{4}}\| K \|^{\frac{1}{2}}_{{L^2_{(sc)}}} \| \psi \|_{{L^4_{(sc)}}} + \| \D \psi \|_{{L^2_{(sc)}}}\lesssim  C + \|\Theta\|_{{L^2_{(sc)}}}  +  \|(\Psi_g,\psi_g)\|_{{L^2_{(sc)}}} + \delta^{\frac{1}{2}}\|(\psi, \Upsilon)\|^2_{{L^4_{(sc)}}}.
\end{align*}
Hence,
\begin{equation} \label{est_hodge}
\| \nabla \psi \|_{{L^2_{(sc)}(S_{u,\ub})}} \lesssim  C + \| \Theta \|_{{L^2_{(sc)}(S_{u,\ub})}}.
\end{equation}
We combine \eqref{est_L}, \eqref{est_Lb} and \eqref{est_hodge} to deduce,
\begin{equation*}
\|\Theta\|_{{L^2_{(sc)}(S_{u,\ub})}} \lesssim C + \delta^{\frac{1}{2}}\Delta_0 \Oonetwo+ \int_{0}^{u} \| \nabla \psi \|_{L^2_{(sc)}(u', \ub)} \lesssim C + \int_{0}^{u}  \| \Theta \|_{L^2_{(sc)}(u', \ub)}+ \delta^{\frac{1}{2}}\Delta_0 \cdot \Oonetwo.
\end{equation*}
Thanks to Gronwall's inequality, we can remove the integral on the right-hand side.
Combined with \eqref{est_hodge} again, we deduce
\begin{equation*}
\Oonetwo \lesssim C + \delta^{\frac{1}{2}}\Delta_0 \cdot \Oonetwo.
\end{equation*}
When $\delta$ is sufficiently small, this yields the following proposition,
\begin{proposition} \label{Oone_estimates}
If $\delta^{\frac{1}{4}} \Delta_0$ is sufficiently small, for $\Theta \in \{\mu, \mub, \kappab, \kappa\}$, we have
\begin{equation*}
\| \Theta \|_{{L^2_{(sc)}(S_{u,\ub})}} + \OSonetwo \lesssim C.
\end{equation*}
\end{proposition}
As a byproduct of above argument,  we have
\begin{corollary}\label{onederivativeonhhb}
We have
\begin{equation*}
\|(\nabla_4 \mu, \nabla_4 \kappab)\|_{L^2_{(sc)}(H_u)} + \|(\nabla_3\mub, \nabla_3 \kappa) \|_{L^2_{(sc)}{(\Hb_{\ub})}} \lesssim C.
\end{equation*}
\end{corollary}

\subsubsection{Improved One Derivative Estimates on Maxwell Field}

For $\Upsilon_g$, in view of \eqref{NM_L_alphab}, \eqref{NM_L_rho} and \eqref{NM_L_sigma}, we have $\nabla_4 \Upsilon_g  = \nabla \Upsilon + \psi \cdot \Upsilon$. We commute with $\nabla$ to derive
\begin{equation*}
 \nabla_4 \nabla \Upsilon_g = \psi \cdot \nabla \Upsilon_g + (\beta + \psi_g \cdot \psi + \Upsilon \cdot \Upsilon)\cdot \Upsilon_g + \psi_g \cdot \nabla \Upsilon  + \nabla^2 \Upsilon.
\end{equation*}
Integrating this equation along $H_u$, it is routine to derive
\begin{equation*}
 \|\nabla(\rho_F, \sigma_F, \alphab_F)\|_{L^{2}_{(sc)}(u,\ub)} \lesssim C + \delta^{\frac{1}{2}}\Delta_0 \|\nabla(\alpha_F, \rho_F, \sigma_F, \alphab_F)\|_{L^{2}_{(sc)}(u,\ub)}.
\end{equation*}
Combined with Gronwall's inequality, similar argument along $\Hb_{\ub}$ yields
\begin{equation*}
 \|\nabla(\alpha_F, \rho_F, \sigma_F)\|_{L^{2}_{(sc)}(u,\ub)} \lesssim C + \delta^{\frac{1}{2}}\Delta_0 \|\nabla(\alpha_F, \rho_F, \sigma_F, \alphab_F)\|_{L^{2}_{(sc)}(u,\ub)}.
\end{equation*}
Putting those estimates together, we have
\begin{equation}\label{nabla4nablaF}
\|\nabla \Upsilon\|_{L^{2}_{(sc)}(u,\ub)} +  \|\nabla_3 \nabla \Upsilon_4 \|_{L^{2}_{(sc)}(H_u)} + \|\nabla_4 \nabla\Upsilon_g\|_{L^{2}_{(sc)}(\Hb_{\ub})} \lesssim C.
\end{equation}

We move on to the $L^\infty_{(sc)}$ bound on $\Upsilon$. For $\Upsilon_g$, in view of \eqref{sobolev_Linfinity},
\begin{equation*}
 \|\Upsilon_g\|_{L^{\infty}_{(sc)}} \lesssim \|\Upsilon_g\|^{\frac{1}{2}}_{L^{4}_{(sc)}(u,\ub)}\|\nabla \Upsilon_g\|^{\frac{1}{2}}_{L^{4}_{(sc)}(u,\ub)} + \delta^{\frac{1}{4}} \|\Upsilon_g\|^{\frac{1}{2}}_{L^{4}_{(sc)}(u,\ub)}
\lesssim C \|\nabla \Upsilon_g\|^{\frac{1}{2}}_{L^{4}_{(sc)}(u,\ub)} +  C\delta^{\frac{1}{4}}.
\end{equation*}
According to \eqref{trace_H} and \eqref{nabla4nablaF},
\begin{equation*}
 \|\nabla \Upsilon_g\|_{L^{4}_{(sc)}(u,\ub)} \lesssim C(C\delta^{\frac{1}{2}}+\|\nabla_4 \nabla \Upsilon_g\|_{L^2_{(sc)}(H_u)})^{\frac{1}{2}} \lesssim C.
\end{equation*}
Thus, $ \|\Upsilon_g|_{L^{\infty}_{(sc)}}$ is bounded by $C$. For $\|\alpha_F\|_{L^\infty_{(sc)}}$, in view of \eqref{trace_Hb} and \eqref{nabla4nablaF}, we have
\begin{equation*}
 \|\nabla \alpha_F\|_{L^{4}_{(sc)}(u,\ub)} \lesssim C(C\delta^{\frac{1}{2}}+\|\nabla_3 \nabla \alpha_F\|_{L^2_{(sc)}(H_u)})^{\frac{1}{2}} \lesssim C.
\end{equation*}
To rectify the anomalous behavior of $\alpha_F$, we have to use localized Sobolev estimates \eqref{sobolev_local},
\begin{equation*}
 \|\alpha_F\|_{L^{\infty}_{(sc)}(u,\ub)} \lesssim \sup_{{}^{\delta}\!S_{u,\ub}}(\|\nabla \alpha_F\|_{L^4_{(sc)}({}^{2\delta}\!S_{u,\ub})}+ \|\alpha_F\|_{L^4_{(sc)}({}^{2\delta}\!S_{u,\ub})})\lesssim C + \sup_{{}^{\delta}\!S_{u,\ub} \subset S_{u,\ub}} \|\alpha_F\|_{L^4_{(sc)}({}^{2\delta}\!S_{u,\ub})}.
\end{equation*}
Thanks to \eqref{local alpha_F}, we have the desired estimates on $\alpha_F$. As conclusions,
\begin{proposition}\label{prop_L_infty_Psi_F}
 If $\delta$ is sufficiently small, we have
\begin{equation*}
\|\Upsilon\|_{L^{\infty}_{(sc)}} + \|\nabla \Upsilon\|_{L^{4}_{(sc)}(u,\ub)}  \lesssim C.
\end{equation*}
Moreover,
\begin{equation*}
  \delta^{\frac{1}{4}}\|\nabla_3 \alpha_F\|_{L^{4}_{(sc)}(u,\ub)}+ \|(\nabla_4\rho_F, \nabla_3 \rho_F, \nabla_4 \sigma_F, \nabla_3 \sigma_F, \nabla_4 \alphab_F)\|_{L^{4}_{(sc)}(u,\ub)}  \lesssim C.
\end{equation*}
\end{proposition}
This is an immediate consequence of null Maxwell equations. We remark that, compared to \eqref{bootstrap}, we have obtained better bound on $\|\Upsilon\|_{L^{\infty}_{(sc)}}$.

\subsection{Second Derivative Estimates}

\subsubsection{Second Angular Derivative Estimates on connection coefficients}
For $\Theta \in \{\nabla \tr\chi, \mu, \kappab\}$, we commute angular derivative with \eqref{L_transport_one_derivative} to derive
\begin{align*}
\nabla_4 \nabla \Theta &=  \nabla[\psi \cdot(\nabla \psi + \Psi + \Upsilon \cdot \Upsilon) + \psi \cdot\psi_g +\psi\cdot \psi\cdot \psi_g + \Upsilon \cdot \nabla\Upsilon ] + [\nabla_4, \nabla] \Theta\\
&= \psi \cdot (\nabla \psi + \nabla^2 \psi+ \nabla\Psi + \nabla \Theta + \Upsilon\cdot \nabla \Upsilon) + \nabla \psi \cdot( \nabla \psi + \Psi + \Upsilon\cdot \Upsilon )\\
&\quad + \psi\cdot(\psi \cdot\nabla \psi + \psi \cdot\Theta)  + (\Psi + \Upsilon \cdot\Upsilon)\cdot\Theta  + (\Upsilon\cdot \nabla^2 \Upsilon + \nabla \Upsilon \cdot\nabla \Upsilon) + \psi \cdot\nabla_4 \Theta\\
&= T_1 + T_2 + T_3 + T_4 + T_5 + T_6.
\end{align*}
Similarly, for $\Theta \in \{\nabla \trchibt, \mub, \kappa\}$, we have
\begin{align*}
\nabla_3 \nabla \Theta &= \psi\cdot(\nabla \psi + \nabla^2 \psi+ \nabla\Psi + \nabla \Theta + \Upsilon \cdot\nabla \Upsilon) + \nabla \psi\cdot( \nabla \psi + \Psi + \Upsilon\cdot \Upsilon )\\
&\qquad + \psi\cdot(\psi \cdot\nabla \psi + \psi \Theta)  + (\Psi + \Upsilon \cdot\Upsilon)\cdot\Theta  + (\Upsilon \cdot\nabla^2 \Upsilon + \nabla \Upsilon \cdot\nabla \Upsilon) + \psi \nabla_3 \Theta\\
&\qquad \tr\chib_0 \cdot(\nabla \psi + \nabla^2 \psi+ \nabla\Psi + \nabla \Theta + \Upsilon \cdot\nabla \Upsilon +\psi \cdot\nabla \psi + \psi \cdot \Theta )\\
&= S_1 + S_2 + S_3 + S_4 + S_5 + S_6 + S_7.
\end{align*}
where the additional term $S_7$ collects all terms involving $\tr\chib_0$.

We also apply $\D^*$ on Hodge system \eqref{Hodge_one_derivative} to derive
\begin{align*}
\triangle \psi &= K\cdot\psi + \D^* \D\psi =K \psi + \D^*( \Theta + \Psi_g + \tr\chib_0 \cdot \psi_g + \psi \cdot \psi + \Upsilon \cdot \Upsilon)\\
&=K\cdot \psi + \nabla \Theta + \nabla \Psi + \nabla \psi + \psi\cdot \nabla \psi + \Upsilon \cdot \nabla \Upsilon.
\end{align*}

Therefore, we obtain a transport-Hogde systems,
\begin{equation}\label{L_second}
\nabla_4 \nabla \Theta = T_1 + T_2 + T_3 + T_4 + T_5 + T_6, \qquad \Theta \in \{\nabla \tr\chi, \mu, \kappab\},
\end{equation}
\begin{equation}\label{Lb_second}
\nabla_3 \nabla \Theta = S_1 + S_2 + S_3 + S_4 + S_5 + S_6 + S_7, \qquad \Theta \in \{\nabla \trchibt, \mub, \kappa\},
\end{equation}
\begin{equation}\label{hodge_second}
\triangle \psi=K \cdot \psi + \nabla \Theta + \nabla \Psi + \nabla \psi + \psi\cdot \nabla \psi + \Upsilon\cdot \nabla \Upsilon.
\end{equation}

We start with \eqref{Lb_second} which are slightly more difficult than \eqref{L_second} because of the presence of $\tr\chib_0$. Estimates on \eqref{L_second} can be derived in a similar way. In \eqref{Lb_second}, we replace $\Theta$ by $\Theta = \nabla \psi +\Psi$ and replace $\nabla \Theta$ by $\nabla \Theta = \nabla^2\psi + \nabla \Psi$. By H\"older's inequality and the estimates derived so far, on $S_{u,\ub}$ we have
\begin{align*}
\|(S_1, S_3, S_7)\|_{{L^2_{(sc)}}} 
&\lesssim C + \|\nabla^2 \psi\|_{{L^2_{(sc)}}} + \|\nabla\Psi\|_{{L^2_{(sc)}}},
\end{align*}
\begin{equation*}
\|(S_2, S_4)\|_{{L^2_{(sc)}}}\lesssim \delta^{\frac12}( \| \nabla\psi\|_{{L^4_{(sc)}}}+\| \Psi\|_{{L^4_{(sc)}}})( \| \nabla\psi\|_{{L^4_{(sc)}}} +  \|\Psi\|_{{L^4_{(sc)}}} + \delta^\frac{1}{2}C),
\end{equation*}
\begin{equation*}
\|S_5\|_{{L^2_{(sc)}}} \lesssim \delta^{\frac12}(\|\Upsilon\|_{L^{\infty}_{(sc)}} \| \nabla^2 \Upsilon \|_{{L^2_{(sc)}}} + \|\nabla\Upsilon\|^2_{{L^4_{(sc)}}}) \lesssim C \delta^{\frac12}( \| \nabla^2 \Upsilon \|_{{L^2_{(sc)}}} + 1),
\end{equation*}
\begin{equation*}\label{s6}
\|S_6\|_{{L^2_{(sc)}(S_{u,\ub})}} \lesssim \delta^{\frac12}\Delta_0 \|\nabla_3 \Theta\|_{{L^2_{(sc)}(S_{u,\ub})}}.
\end{equation*}
In view of Lemma \ref{sobolev_trace_estimates}, on $S_{u,\ub}$ we have
\begin{align*}
\|\nabla \psi\|_{{L^4_{(sc)}}} & \lesssim \|\nabla \psi\|^{\frac{1}{2}}_{{L^2_{(sc)}}} \|\nabla^2 \psi\|^{\frac{1}{2}}_{{L^2_{(sc)}}}+\delta^{\frac{1}{4}}\|\nabla \psi\|_{{L^2_{(sc)}}}\lesssim \|\nabla^2 \psi\|^{\frac{1}{2}}_{{L^2_{(sc)}}} + C\delta^{\frac{1}{4}},\\
\|\Psi\|_{{L^4_{(sc)}}} & \lesssim \|\Psi\|^{\frac{1}{2}}_{{L^2_{(sc)}}} \|\nabla \Psi\|^{\frac{1}{2}}_{{L^2_{(sc)}}}+\delta^{\frac{1}{4}}\|\Psi\|_{{L^2_{(sc)}}}\lesssim C\delta^{-\frac{1}{4}}\|\nabla \Psi\|^{\frac{1}{2}}_{{L^2_{(sc)}}} + C\delta^{-\frac{1}{4}}.
\end{align*}
Thus,
\begin{equation}\label{si}
\|(S_1, S_2, S_3, S_4,S_7)\|_{{L^2_{(sc)}(S_{u,\ub})}} \lesssim C + \|\nabla^2 \psi\|_{{L^2_{(sc)}(S_{u,\ub})}} + \|\nabla\Psi\|_{{L^2_{(sc)}(S_{u,\ub})}}.
\end{equation}
We observe that $S_6$ can be estimated by the Corollary \ref{onederivativeonhhb}. Thus, for $\Theta \in \{\nabla \trchibt, \mub, \kappa\}$, we have
\begin{align}\label{eq1}
\|\nabla\Theta\|_{{L^2_{(sc)}(S_{u,\ub})}} &\lesssim \|\nabla\Theta\|_{L^2_{(sc)}(0, \ub)} + \int_0^u \|\nabla_3 \nabla \Theta\|_{L^2_{(sc)}(u', \ub)} \notag\\
&\lesssim C + \int_0^u \|\nabla^2 \psi\|_{L^2_{(sc)}(u', \ub)}+ \|\nabla \Psi\|_{L^2_{(sc)}(u', \ub)}+ \|\nabla^2 \Upsilon\|_{L^2_{(sc)}(u', \ub)}\notag\\
&\lesssim C + \int_0^u \|\nabla^2 \psi\|_{L^2_{(sc)}(u', \ub)}.
\end{align}

Similarly, for $\Theta \in \{\nabla \tr\chi, \mu, \kappab\}$, we can use \eqref{L_second} to derive
\begin{equation}\label{eq2}
\|\nabla\Theta\|_{{L^2_{(sc)}(S_{u,\ub})}} \lesssim C + \delta^{-1}\int_0^{\ub} \|\nabla^2 \psi\|_{L^2_{(sc)}(u, \ub')}.
\end{equation}

We make use of elliptic estimates on \eqref{hodge_second} on $S_{u,\ub}$ to derive,
\begin{align*}
\|\nabla^2 \psi\|_{{L^2_{(sc)}}} &\lesssim \delta^{\frac{1}{2}}\Delta_0 \|(K,\psi,\nabla \Upsilon)\|_{{L^2_{(sc)}}} + \|\nabla \Theta\|_{{L^2_{(sc)}}} +\|\nabla \Psi\|_{{L^2_{(sc)}}}+ \|\nabla \psi\|_{{L^2_{(sc)}}}\lesssim C + \|\nabla \Theta\|_{{L^2_{(sc)}}}.
\end{align*}
We plug it in \eqref{eq1} and \eqref{eq2} to derive
\begin{align*}
\|\nabla\Theta\|_{{L^2_{(sc)}(S_{u,\ub})}}&\lesssim C + \int_0^u \|\nabla \Theta \|_{L^2_{(sc)}(u', \ub)}du' \qquad \text{for}\quad \Theta \in \{\nabla \trchibt, \mub, \kappa\},\\
\|\nabla\Theta\|_{{L^2_{(sc)}(S_{u,\ub})}}&\lesssim C +\delta^{-1} \int_0^{\ub} \|\nabla \Theta\|_{L^2_{(sc)}(u, \ub')}d\ub' \qquad \text{for}\quad \Theta \in \{\nabla \tr\chi, \mu, \kappab\}.
\end{align*}
Thus, Gronwall inequality yields the following proposition,
\begin{proposition}\label{secondangularderivative}
If $\delta$ is sufficiently small, then for $\psi \in \{\tr\chi, \trchibt,\chih,\chibh,\eta, \etab, \omega, \omegab\}$ and $\Theta \in \{\nabla \tr \chi, \nabla \trchibt, \mu, \mub, \kappa, \kappab\}$, we have
\begin{equation*}
\|\nabla\Theta\|_{{L^2_{(sc)}(S_{u,\ub})}} + \|\nabla^2\psi\|_{L^2_{(sc)}(H_u)} +  \|\nabla^2\psi\|_{L^2_{(sc)}(\Hb_{\ub})}\lesssim C.
\end{equation*}
\end{proposition}
As a corollary of Proposition \ref{secondangularderivative} and Lemma \ref{sobolev_trace_estimates}, we have
\begin{corollary}[$\Oonefour$ Estimates]\label{Oonefour}
If $\delta$ is sufficiently small, then
\begin{equation*}
\Oonefour \lesssim C.
\end{equation*}
\end{corollary}

To predict formation of trapped surfaces, we need a slightly refined estimate on $\nabla^2 \eta$.
\begin{proposition}\label{refined estimates on nabla 2 eta} If $\delta$ is sufficiently small, then
\begin{align*}
\|\nabla^2 \eta\|_{L^2_{(sc)}(H_u)} &\lesssim \|\nabla \rho\|_{L^2_{(sc)}(H_u)} + \|\nabla \sigma\|_{L^2_{(sc)}(H_u)} + \delta^{\frac{1}{4}}C,\\
\|\nabla^2 \eta\|_{L^2_{(sc)}(\Hb_{\ub})} &\lesssim \|\nabla \rho\|_{L^2_{(sc)}(\Hb_{\ub})} + \|\nabla \sigma\|_{L^2_{(sc)}(\Hb_{\ub})} + \delta^{\frac{1}{4}}C.
\end{align*}
\end{proposition}
\begin{proof}
Recall the following Transport-Hodge system for $\eta$ and $\mu$,
\begin{align*}
 \curl \eta &= \sigma + \psi \cdot \psi, \qquad \divergence \eta = -\mu -\rho,\\
 \nabla_4 \mu &= \psi \cdot (\nabla \psi + \Psi +\Upsilon \cdot \Upsilon + \psi \cdot \psi) + \Upsilon \cdot \nabla \Upsilon.
\end{align*}
The first and second equations lead to
\begin{equation}\label{lap_eta}
 \triangle \eta = \nabla \sigma + \nabla \rho + \nabla \mu + \psi \cdot \nabla \psi + K \cdot \eta.
\end{equation}
The third equation leads to
\begin{align*}
 \nabla_4 \nabla \mu &= \psi \cdot(\nabla \psi + \psi \cdot\nabla \psi + \nabla \Psi + \psi\cdot \Psi + \Upsilon \cdot\Upsilon) + \nabla \psi \cdot(\nabla \psi + \Psi + \Upsilon\cdot \Upsilon) \\
&\qquad + (\nabla \Upsilon \cdot\nabla \Upsilon + \Upsilon \cdot\nabla^2 \Upsilon) + \psi_g \cdot\psi\cdot \psi \cdot\psi.
\end{align*}
When we derive estimates, because each use of H\"older's inequality gains $\delta^{\frac{1}{2}}$, without loss of generality, we ignore most of the nonlinear terms in previous expression. Thus, we have
\begin{equation}\label{transprot_eta}
 \nabla_4 \nabla \mu = \psi \cdot(\nabla \psi + \nabla \Psi) + \nabla \psi \cdot(\nabla \psi + \Psi) + (\nabla \Upsilon \cdot \nabla \Upsilon + \Upsilon \cdot \nabla^2 \Upsilon).
\end{equation}
Equipped with \eqref{lap_eta} and \eqref{transprot_eta}, we proceed exactly as we argued in Proposition \ref{secondangularderivative}. The most dangerous term from \eqref{transprot_eta} is $ \nabla \psi \cdot \Psi$ since $\Psi$ can be $\alpha$. In this case, we use Corollary \ref{Oonefour} and $L^4_{(sc)}$ norms to save a $\delta^{\frac{1}{4}}$. Since $\nabla \mu$ has trivial data on $\Hb_0$, we can conclude
\begin{equation}\label{nablamu}
 \|\nabla \mu\|_{{L^2_{(sc)}(S_{u,\ub})}} \lesssim \delta ^{\frac{1}{4}}C.
\end{equation}
Combined with elliptic estimates on \eqref{lap_eta}, we complete the proof.
\end{proof}

\subsubsection{First Derivative Estimates in Null Direction}
For Maxwell field, as an immediate consequence of \eqref{NM_Lb_alpha}-\eqref{NM_Lb_rho}, we have
\begin{proposition}\label{first_null_derivative_Maxwell}
 If $\delta$ is sufficiently small, then
\begin{equation}
 \delta^{\frac{1}{2}}\|\nabla_3 \alpha_F\|_{{L^2_{(sc)}(S_{u,\ub})}}+ \|( \nabla_3 \rho_F, \nabla_3 \sigma_F, \nabla_4 \rho_F, \nabla_4 \sigma_F, \nabla_4 \alphab_F)\|_{{L^2_{(sc)}(S_{u,\ub})}}\lesssim C.
\end{equation}
\end{proposition}

For connection coefficients, we have
\begin{proposition}\label{Remaininig_First_Derivative_Estimates}
 If $\delta$ is sufficiently small, on $S_{u,\ub}$ we have
\begin{equation}\label{est_NSE_4}
 \|(\nabla_4 \tr \chi, \nabla_4 \eta, \nabla_4 \omegab, \nabla_4 \tr \chib, \nabla_3 \trchibt, \nabla_3 \etab, \nabla_3 \omega, \nabla_3 \tr \chi)\|_{{L^2_{(sc)}}} \lesssim C,
\end{equation}
\begin{equation}\label{est_NSE_anomalous}
 \|(\nabla_4 \chih, \nabla_4 \chibh, \nabla_3 \chih, \nabla_3 \chibh)\|_{{L^2_{(sc)}}}\lesssim C \delta^{-\frac{1}{2}},
\end{equation}
\begin{equation}\label{est_noNSE}
 \|(\nabla_4 \etab, \nabla_4 \omega, \nabla_3 \eta, \nabla_3 \omegab)\|_{{L^2_{(sc)}}} \lesssim C.
\end{equation}
\end{proposition}
\begin{proof}
Those terms from \eqref{est_NSE_4} and \eqref{est_NSE_anomalous} appear also on the left hand side of \eqref{NSE_L_tr_chi}-\eqref{NSE_Lb_chih}. They can be estimated directly with the help of those equations. To illustrate the idea, we estimate $\nabla_4 \chih$ by virtue of \eqref{NSE_L_chih}, i.e. $\nabla_4 \chih = \psi_g \cdot \chih + \alpha$. Therefore,
\begin{equation*}
  \|\nabla_4 \chih\|_{{L^2_{(sc)}(S_{u,\ub})}} \lesssim \delta^{\frac{1}{2}}\|\psi_g\|_{{L^4_{(sc)}(S_{u,\ub})}}\|\chih\|_{{L^4_{(sc)}(S_{u,\ub})}} + C\delta^{-\frac{1}{2}} \lesssim \delta^{\frac{1}{4}}C+C\delta^{-\frac{1}{2}}.
\end{equation*}
The other terms in \eqref{est_NSE_4} and \eqref{est_NSE_anomalous} can be estimated in the same way. We omit the details.

\eqref{est_noNSE} is more delicate because $\nabla_4 \etab, \nabla_4 \omega, \nabla_3 \eta$ and $\nabla_3 \omegab$  do not appear in the null structure equations. In order to derive estimates, we have to commute derivatives. We only estimate $\nabla_4 \etab$. The remaining terms can be estimated exactly in the same way.

We commute \eqref{NSE_Lb_etab} with $\nabla_4$ to derive,
\begin{align*}
\nabla_3 \nabla_4 \etab &= [\nabla_3,\nabla_4] \etab + \nabla_4 (-\frac{1}{2}\tr\chib \cdot (\etab-\eta)-\chibh \cdot (\etab -\eta)+\betab-\frac{1}{2}{\sigma_F}\cdot \,^*\!{\alphab_F} + \frac{1}{2}{\alphab_F} \cdot {\rho_F})\\
&=\tr\chib_0 \cdot \nabla_4 \etab + \nabla_4 \betab + E.
\end{align*}
One checks easily that error term $E$ enjoys a better bound (notice that $\nabla_4 \Upsilon$ in $E$ is controlled by Proposition \ref{first_null_derivative_Maxwell}). In fact, we can show $ \int_0^{u} \|E\|_{L^2_{(sc)}(u', \ub)}  \lesssim C$. This bound is good enough for our purpose. We use \eqref{NBE_L_betab} to replace $\nabla_4 \betab$ by $\psi\cdot \nabla \Psi + \Upsilon \cdot \nabla \Upsilon$. It is easy to see that they can also be absorbed into error term $E$, thus we have $\nabla_3 \nabla_4 \etab = \tr\chib_0 \cdot \nabla_4 \etab  + E$ which implies
\begin{equation*}
\|\nabla_3 \nabla_4 \etab \|_{{L^2_{(sc)}(S_{u,\ub})}} \lesssim C + \|\nabla_4 \etab\|_{{L^2_{(sc)}(S_{u,\ub})}} + \|E\|_{{L^2_{(sc)}(S_{u,\ub})}}.
\end{equation*}
Therefore,
\begin{align*}
\|\nabla_4 \etab \|_{{L^2_{(sc)}(S_{u,\ub})}} &\lesssim \| \nabla_4 \etab \|_{L^2_{(sc)}(0, \ub)}  + \int_0^{u} \|\nabla_3 \nabla_4 \etab \|_{L^2_{(sc)}(u', \ub)} \\
&\lesssim C+ \int_0^{u} \|\nabla_4 \etab\|_{L^2_{(sc)}(u', \ub)} +\|E \|_{L^2_{(sc)}(u', \ub)} \lesssim C+\int_0^{u} \|\nabla_4 \etab\|_{L^2_{(sc)}(u', \ub)}.
\end{align*}
Finally, we use Gronwall's inequality to complete the proof.
\end{proof}
As a byproduct of the proof, we have
\begin{corollary}\label{nabla_3_nabla_4_eta_etab}
If $\delta$ is sufficiently small, then
\begin{equation}
 \|\nabla_3 \nabla_4 \etab\|_{L^{2}_{(sc)}({\Hb}_{\ub})} + \|\nabla_4 \nabla_3 \eta\|_{L^{2}_{(sc)}(H_u)} \lesssim C.
\end{equation}
\end{corollary}

\subsubsection{Remaining Second Derivative Estimates}
We start with a lemma on commutator estimates which allows us to freely switch the order of differentiations in second derivative estimates. For example, the order of differentiations in Proposition \ref{2 der est on max} does not effect the estimates.
\begin{lemma}\label{est_commutator}
If $\delta$ is sufficiently small, for all connection coefficients $\psi$ and null components of Maxwell field $\Upsilon$, we have
\begin{equation}\label{est_commutator_S}
\|([\nabla, \nabla_4] \psi, [\nabla, \nabla_3] \psi)\|_{{L^2_{(sc)}(S_{u,\ub})}} +\|([\nabla, \nabla_4] \Upsilon, [\nabla, \nabla_3] \Upsilon)\|_{{L^2_{(sc)}(S_{u,\ub})}} \lesssim C,
\end{equation}
\begin{equation}\label{est_commutator_H_Hb}
\|([\nabla, \nabla_4] \psi, [\nabla, \nabla_3]\psi) \|_{L^2_{(sc)}(H_u)} + \|([\nabla, \nabla_4] \psi, [\nabla, \nabla_3] \psi)\|_{L^2_{(sc)}(\Hb_{\ub})} \lesssim C.
\end{equation}
\begin{equation}\label{est_commutator_H_Hb_F}
\|([\nabla, \nabla_4] \Upsilon, [\nabla, \nabla_3]\Upsilon) \|_{L^2_{(sc)}(H_u)} + \|([\nabla, \nabla_4] \Upsilon, [\nabla, \nabla_3] \Upsilon)\|_{L^2_{(sc)}(\Hb_{\ub})} \lesssim C.
\end{equation}
\end{lemma}
\begin{proof} We prove \eqref{est_commutator_S}, while \eqref{est_commutator_H_Hb} and \eqref{est_commutator_H_Hb_F} are simple consequences of \eqref{est_commutator_S}. We use the following schematic formulas (it is of the same form once we replace $\psi$ by $\Upsilon$),
\begin{align*}
[\nabla_4, \nabla] \psi& = \psi \cdot (\nabla \psi + \Psi_g + \Upsilon \cdot\Upsilon) + \psi_g \cdot \nabla_4 \psi, \\
[\nabla_3, \nabla] \psi& = \tr\chib_0 \cdot \nabla \psi + \psi \cdot (\nabla \psi + \Psi_g + \Upsilon \cdot \Upsilon) + \psi_g \cdot \nabla_3 \psi,
\end{align*}
These formulas allows one to use previously obtained $\Oonetwo$ estimates to bound the commutators. We observe that possible anomalies are due to the presence of $\nabla_4 \psi$ and $\nabla_3 \psi$ at the right-hand side of the above formulas. They may cause a loss of $\delta^{-\frac{1}{2}}$. But H\"older's inequality gains $\delta^{\frac{1}{2}}$ which compensates this loss. This completes the proof.
\end{proof}
We turn to the estimates on second derivatives of Maxwell field.
\begin{proposition}\label{2 der est on max}
If $\delta$ is sufficiently small, we have
\begin{equation}
 \|(\nabla \nabla_4 \alphab_F, \nabla \nabla_3 \alpha_F, \nabla \nabla_4 \rho_F, \nabla \nabla_4 \sigma_F)\|_{L^2_{(sc)}(H_u)} \lesssim C,
\end{equation}
\begin{equation}
  \|(\nabla \nabla_4 \alphab_F, \nabla \nabla_3 \alpha_F, \nabla \nabla_3 \rho_F, \nabla \nabla_3 \sigma_F)\|_{L^2_{(sc)}(\Hb_{\ub})} \lesssim C,
\end{equation}
\end{proposition}
\begin{proof}
The proof is straightforward. We first apply $\nabla$ on \eqref{NM_L_alphab}-\eqref{NM_Lb_sigma}. The angular derivative cleans up all linear anomalies. Thus we can use the estimates derived so far to complete the estimates.  We omit the details.
\end{proof}

Next proposition collects remaining second derivative estimates on connection coefficients.
\begin{proposition}\label{twoderivativeRicciCoefficients}
 If $\delta$ is sufficiently small, then
\begin{equation}\label{est_4_H}
 \|(\nabla \nabla_4 \tr \chi, \nabla \nabla_4 \chih, \nabla  \nabla_4 \eta, \nabla \nabla_4 \etab, \nabla  \nabla_4 \omegab,\nabla \nabla_4 \trchibt,\nabla \nabla_4 \chibh)\|_{L^2_{(sc)}(H_u)} \lesssim C,
\end{equation}
\begin{equation}\label{est_4_Hb}
  \|(\nabla \nabla_4 \tr \chi, \nabla  \nabla_4 \eta, \nabla \nabla_4 \etab,\nabla  \nabla_4 \omegab,\nabla \nabla_4 \trchibt,\nabla \nabla_4 \chibh)\|_{L^2_{(sc)}(\Hb_{\ub})} \lesssim C,
\end{equation}
\begin{equation}\label{est_3_Hb}
  \|(\nabla \nabla_3 \tr \chi, \nabla \nabla_3 \chih, \nabla  \nabla_3 \etab, \nabla \nabla_3 \eta, \nabla  \nabla_3 \omega,\nabla \nabla_3 \trchibt,\nabla \nabla_3 \chibh)\|_{L^2_{(sc)}(\Hb_{\ub})} \lesssim C
\end{equation}
\begin{equation}\label{est_3_H}
  \|(\nabla \nabla_3 \tr \chi, \nabla \nabla_3 \chih, \nabla  \nabla_3 \etab, \nabla \nabla_3 \eta, \nabla  \nabla_3 \omega,\nabla \nabla_3 \trchibt)\|_{L^2_{(sc)}(H_u)} \lesssim C,
\end{equation}
\end{proposition}


\begin{proof} Except for $\nabla \nabla_4 \etab$ and $\nabla \nabla_3 \eta$, all the terms can be estimated directly from \eqref{NSE_L_tr_chi}-\eqref{NSE_Lb_chih}.
For $\psi \in \{\tr\chi, \chih,\omegab, \eta, \trchibt, \chibh\}$, the corresponding null structure equations in $\nabla_4$ direction read as $\nabla_4 \psi = \tr\chib_0 \cdot \psi + \psi \cdot \psi + \nabla \psi + \Psi_4 + \Upsilon \cdot \Upsilon$. Thus,
\begin{equation*}
\nabla \nabla_4 \psi = \tr\chib_0 \cdot \nabla \psi + \psi \cdot \nabla \psi + \nabla^2 \psi + \nabla \Psi_4 + \Upsilon \cdot \nabla \Upsilon.
\end{equation*}
Every term on the right-hand side is bounded along $H_u$. Thus,
\begin{equation*}
\| \nabla \nabla_4 \psi \|_{L^2_{(sc)}(H_u)} \lesssim C.
\end{equation*}
Similarly, along $\Hb_{\ub}$, we have
\begin{equation}\label{a4}
\| \nabla \nabla_4 \psi \|_{L^2_{(sc)}(\Hb_{\ub})} \lesssim C.
\end{equation}
Since the norm $\Roneb$ does not include $\nabla \alpha$, $\|\nabla \nabla_4 \chih\|_{L^2_{(sc)}(\Hb_{\ub})}$ is absent in \eqref{a4}. We can proceed in the same way to prove the those estimates in $\nabla_3$ directions and we omit the details.

It remains to control the most difficult terms $\nabla \nabla_4 \etab$ and $\nabla \nabla_3 \eta$. Since the arguments are similar for both, we only derive estimates on $\nabla \nabla_4 \etab$. Recall the transport-Hodge system for $(\mub,\etab)$,
\begin{equation}\label{a3}
 \nabla_3 \mub =\tr\chib_0  \cdot (\psi_g \cdot \psi + \nabla \psi) + \psi \cdot (\psi_g \cdot \psi + \nabla \psi + \Psi_3 + \Upsilon \cdot \Upsilon),
\end{equation}
\begin{equation}\label{ahodge}
 \D \etab = \mub + \rho + \sigma + \psi \cdot \psi.
\end{equation}
We observe that $\omega$ and $\omegab$ do not appear among $\nabla \psi$'s. This is extremely important since we do not have estimates on $\nabla \nabla_4 \omega$ and $\nabla \nabla_3 \omegab$. Another observation is equally important: there is no double anomalous terms in the form $\tr\chib_0 \cdot \Psi$, $\tr\chib_0 \cdot \chih$ or $\tr\chib_0 \cdot\chibh$.

We derive a transport-Hodge system for $(\nabla_4 \mub, \nabla_4 \etab)$. In view of \eqref{a3},
we derive
\begin{align*}
 &\nabla_3(\nabla_4 \mub) = (\psi\cdot \nabla \mub + \omegab \cdot \nabla_4 \mub + \omega \cdot \nabla_3 \mub)+ (\psi_g \cdot \psi + \nabla \psi) + (\psi \cdot\nabla_4 \psi + \nabla_4 \nabla \psi) \\
&+ \nabla_4 \psi \cdot (\psi_g \cdot \psi + \nabla \psi + \Psi_3 + \Upsilon \cdot \Upsilon) + \psi \cdot (\psi \cdot \nabla_4\psi + \nabla_4 \nabla \psi +\nabla_4 \Psi_3 + \Upsilon \cdot \nabla_4 \Upsilon).
\end{align*}
where $\Psi_3 \in \{ \beta, \rho, \sigma, \alphab\}$. We have replaced $\tr\chib_0$ and $\nabla_4 \tr\chib_0$ by the constant 1. Without loss of generality, we can drop some terms which enjoy easier and better estimates. Thus above equation is reduced to the following form,
\begin{align*}
 \nabla_3(\nabla_4 \mub) 
& =  \psi \cdot \nabla \mub + \omegab \cdot \nabla_4 \mub + \omega \cdot \nabla_3 \mub  + (1+\psi)\cdot\nabla_4 \nabla \psi + \nabla_4 \psi\cdot (\psi + \nabla \psi + \Psi_3 + \Upsilon \cdot \Upsilon) \\&+ \psi \cdot (\nabla_4 \Psi_3 + \Upsilon\cdot \nabla_4 \Upsilon) = T_1 + T_2 +T_3+ T_4 +T_5 + T_6.
\end{align*}
Commuting $\nabla_4$ with \eqref{ahodge}, we also have,
\begin{equation}\label{est_hodge_2}
 \D (\nabla_4 \etab) = \nabla_4 \mub + (\nabla_4 \rho + \nabla_4\sigma) + \psi \cdot (\nabla_4 \psi + \nabla \etab + \Psi_4 + \psi_g \cdot \psi).
\end{equation}

We control $\int_{0}^u \|T_i\|_{L^2_{(sc)}(u',\ub)}du'$. For $T_1$, in view of \eqref{mu}, we have
\begin{align*}
 \int_{0}^u \|T_1\|_{L^2_{(sc)}(u',\ub)} &\lesssim \int_{0}^u \|\psi  \cdot (\nabla^2 \psi + \nabla \Psi_g)\|_{L^2_{(sc)}(u',\ub)}\lesssim C\delta^{\frac{1}{2}}.
\end{align*}

For $T_2$, we have
\begin{align*}
 \int_{0}^u \|T_2\|_{L^2_{(sc)}(u',\ub)} &\lesssim \int_{0}^u \|\omegab \nabla_4 \mub\|_{L^2_{(sc)}(u',\ub)}  \lesssim \delta^{\frac{1}{2}}\Delta_0 \int_{0}^u \|\nabla_4 \mub\|_{L^2_{(sc)}(u',\ub)}.
\end{align*}

For $T_3$, in view of \eqref{a3}, we also have $\int_{0}^u \|T_3\|_{L^2_{(sc)}(u',\ub)}du' \lesssim C$.

For $T_5$ (we come back to $T_4$ later), it consists of four terms. The one involving Maxwell field is easy since it is a cubic nonlinearity. Each of the rest three terms can be estimated by virtue of Sobolev inequality. We only treat one of them to illustrate the idea. The others can be derived in the same manner.
First of all, we have
\begin{equation*}
 \int_{0}^u \|\nabla_4 \psi \cdot \Psi_3\|_{L^2_{(sc)}(u',\ub)} \lesssim \delta^{\frac{1}{2}} \int_{0}^u \|\nabla_4 \psi \|_{L^4_{(sc)}(u',\ub)} \|\Psi_3 \|_{L^4_{(sc)}(u',\ub)} .
\end{equation*}
In view of Proposition \ref{Remaininig_First_Derivative_Estimates} and Lemma \ref{sobolev_trace_estimates}, on $S_{u,\ub}$ we have
\begin{equation*}
 \|\nabla_4 \psi \|_{L^4_{(sc)}} \lesssim \|\nabla \nabla_4 \psi \|^{\frac{1}{2}}_{{L^2_{(sc)}}} \|\nabla_4 \psi \|^{\frac{1}{2}}_{{L^2_{(sc)}}} + \delta^{\frac{1}{4}}\|\nabla_4 \psi \|_{{L^2_{(sc)}}} \lesssim \delta^{-\frac{1}{4}}C ( \|\nabla \nabla_4 \psi \|^{\frac{1}{2}}_{{L^2_{(sc)}}} + 1),
\end{equation*}
\begin{equation*}
 \|\Psi_3 \|_{L^4_{(sc)}} \lesssim \|\nabla \Psi_3 \|^{\frac{1}{2}}_{{L^2_{(sc)}}} \|\Psi_3 \|^{\frac{1}{2}}_{{L^2_{(sc)}}} + \delta^{\frac{1}{4}}\|\Psi_3 \|_{{L^2_{(sc)}}} \lesssim C ( \|\nabla \Psi_3 \|^{\frac{1}{2}}_{{L^2_{(sc)}}} + \delta^{\frac{1}{4}}).
\end{equation*}
Thus,
\begin{align*}
 \int_{0}^u \|\nabla_4 \psi \cdot \Psi_3\|_{L^2_{(sc)}(u',\ub)} &\lesssim \delta^{\frac{1}{4}} C \int_{0}^u ( \|\nabla \nabla_4 \psi \|^{\frac{1}{2}}_{{L^2_{(sc)}(S_{u',\ub})}} + 1)( \|\nabla \Psi_3 \|^{\frac{1}{2}}_{{L^2_{(sc)}(S_{u',\ub})}} + 1)\\
&\lesssim \delta^{\frac{1}{4}} C \int_{0}^u ( \|\nabla \nabla_4 \psi \|_{{L^2_{(sc)}(S_{u',\ub})}} +  \|\nabla \Psi_3 \|_{{L^2_{(sc)}(S_{u',\ub})}} + 1) \\
&\lesssim C +  \|\nabla \nabla_4 \psi \|_{L^2_{(sc)}(\Hb_{\ub})}.
\end{align*}
Recall that in Proposition \ref{twoderivativeRicciCoefficients} we have just proved $\|\nabla \nabla_4 \psi \|_{L^2_{(sc)}(\Hb_{\ub})} \lesssim C$. Since $\omega$ and $\omegab$ are absent in \eqref{a3} and \eqref{ahodge}, we have $\int_{0}^u \|\nabla_4 \psi \cdot \Psi_3\|_{L^2_{(sc)}(u',\ub)}  \lesssim C$. Hence, $\int_{0}^u \| T_5 \|_{L^2_{(sc)}(u',\ub)} \lesssim C$.

For $T_4$, we have to treat different components in different ways,
\begin{align*}
 \int_{0}^u \|T_4\|_{L^2_{(sc)}(u',\ub)} & \lesssim \int_{0}^u \|\nabla_4 \nabla \psi \|_{L^2_{(sc)}(u',\ub)} \lesssim \int_{0}^u \|\nabla_4 \nabla \etab \|_{L^2_{(sc)}(u',\ub)} + \sum_{\psi \notin\{ \etab,\omegab,\omega\}} \|\nabla_4 \nabla \psi \|_{L^2_{(sc)}(u',\ub)}\\
&\lesssim C + \int_{0}^u \|\nabla \nabla_4 \etab \|_{L^2_{(sc)}(u',\ub)} +  \sum_{\psi \notin\{ \etab,\omegab,\omega\}} \int_{0}^u \|\nabla \nabla_4 \psi \|_{L^2_{(sc)}(u',\ub)}.
\end{align*}
Hence, $\int_{0}^u \|T_4\|_{L^2_{(sc)}(u',\ub)} \lesssim C + \int_{0}^u \|\nabla \nabla_4 \etab \|_{L^2_{(sc)}(u',\ub)}$.

For $T_6$, it has two terms. We estimate them one by one. According to null Bianchi equations, $\nabla_4 \Psi_3 = \nabla \Psi_g + \psi \cdot \Psi + \Upsilon \cdot D \Upsilon$,
we derive
\begin{equation*}
 \int_{0}^u \| \psi \cdot \nabla_4 \Psi_3\|_{L^2_{(sc)}(u',\ub)} \lesssim \delta^{\frac{1}{2}}\Delta_0\int_{0}^u \|\nabla \Psi \|_{L^2_{(sc)}(u',\ub)} + \delta^{\frac{1}{2}}\Delta_0 \|(\Psi_g,D\Upsilon)\|_{L^2_{(sc)}(u',\ub)} \lesssim C.
\end{equation*}
Thus, $ \int_{0}^u T_6 \lesssim C$. Putting all the $T_i$'s together, we have
\begin{align*}
 \sum_{i=1}^6 \int_{0}^u \|T_i\|_{L^2_{(sc)}(u',\ub)}du' &\lesssim C + \int_{0}^u \|\nabla_4 \mub\|_{L^2_{(sc)}(u',\ub)} + \int_{0}^u \|\nabla_4 \nabla \etab \|_{L^2_{(sc)}(u',\ub)}.
\end{align*}
Thus,
\begin{equation*}
\|\nabla_4 \mub\|_{{L^2_{(sc)}(S_{u,\ub})}} \lesssim  C + \int_{0}^u \|\nabla_4 \mub\|_{L^2_{(sc)}(u',\ub)} + \int_{0}^u \|\nabla \nabla_4 \etab \|_{L^2_{(sc)}(u',\ub)}.
\end{equation*}
Thanks to Gronwall's inequality, we have
\begin{equation}\label{est_nabla_mub}
\|\nabla_4 \mub\|_{{L^2_{(sc)}(S_{u,\ub})}} \lesssim C + \int_{0}^u \|\nabla \nabla_4 \etab \|_{L^2_{(sc)}(u',\ub)}.
\end{equation}
We use elliptic estimates on Hodge system \eqref{est_hodge_2} on $S_{u,\ub}$ to derive
\begin{align*}
\|\nabla \nabla_4 \etab \|_{{L^2_{(sc)}}} &\lesssim \|\nabla_4 \mub \|_{{L^2_{(sc)}}} + \|\nabla_4 \rho\|_{{L^2_{(sc)}}} +  \|\nabla_4 \sigma\|_{{L^2_{(sc)}}} + \|\psi \cdot (\nabla_4 \psi + \nabla \etab + \Psi_4 + \psi_g \cdot \psi) \|_{{L^2_{(sc)}}}\\
&\lesssim \|\nabla_4 \mub \|_{{L^2_{(sc)}}} +  \|\nabla_4 \rho\|_{{L^2_{(sc)}}} +  \|\nabla_4 \sigma\|_{{L^2_{(sc)}}} + C.
\end{align*}
Combined with \eqref{est_nabla_mub}, we have
\begin{align*}
\|\nabla_4 \mub\|_{{L^2_{(sc)}(S_{u,\ub})}} 
&\lesssim C + \int_{0}^u \| \nabla_4 \mub \|_{L^2_{(sc)}(u',\ub)}.
\end{align*}
Thanks to Gronwall's inequality, we have $\|\nabla_4 \mub\|_{{L^2_{(sc)}(S_{u,\ub})}} \lesssim C$. We now go back to \eqref{ahodge} again to derive
\begin{equation*}
\|\nabla \nabla_4 \etab\|_{L^2_{(sc)}(H_u)} + \|\nabla \nabla_4 \etab\|_{L^2_{(sc)}(\Hb_{\ub})} \lesssim C.
\end{equation*}
This completes the proof.
\end{proof}

\subsection{End of the Bootstrap Argument for Theorem A}
Combining all estimates derived so far, we close the bootstrap argument for {\bf{Theorem A}} by showing next proposition.
\begin{proposition}\label{close_bootstrap}
If $\delta$ is sufficiently small, we have
\begin{equation}
 \OSzeroinfinity \lesssim C.
\end{equation}
\end{proposition}
\begin{proof}
For $\psi \in \{\tr\chi, \chih, \eta, \omegab\}$, in view of Lemma \ref{sobolev_trace_estimates} and Proposition \ref{twoderivativeRicciCoefficients}, along $H=H_u$,
\begin{align*}
 \|\nabla \psi\|^2_{L^4_{(sc)}(u,\ub)}& \lesssim (\delta^{\frac{1}{2}}\|\nabla \psi\|_{L^2_{(sc)}(H)}+\|\nabla^2 \psi\|_{L^2_{(sc)}(H)}) (\delta^{\frac{1}{2}}\|\nabla \psi\|_{L^2_{(sc)}(H)}+\|\nabla_4 \nabla \psi\|_{L^2_{(sc)}(H)}) \lesssim C.
\end{align*}
 For $\psi \in \{\trchibt, \chibh, \etab, \omega\}$, similarly, we have $\|\nabla \psi\|_{L^4_{(sc)}(u,\ub)} \lesssim C$. Thus, for a non-anomalous $\psi_g$,
\begin{equation*}
 \|\psi_g\|_{L^{\infty}_{(sc)}(u,\ub)} \lesssim \|\psi_g\|^{\frac{1}{2}}_{L^4_{(sc)}(u,\ub)}\|\nabla \psi_g\|^{\frac{1}{2}}_{L^4_{(sc)}(u,\ub)}+ \delta^{\frac{1}{4}}\|\psi_g\|_{L^4_{(sc)}(u,\ub)} \lesssim C.
\end{equation*}
For an anomalous $\psi \in \{\chih, \chibh\}$, we use \eqref{localized_L_4_chih} and the localized Sobolev inequality \eqref{sobolev_local} to derive
\begin{equation*}
 \|\psi\|_{L^{\infty}_{(sc)}(u,\ub)} \lesssim \sup_{{}^{\delta}\!S_{u,\ub} \subset S_{u,\ub}}(\|\nabla \psi\|_{L^4_{(sc)}({}^{2\delta}\!S_{u,\ub})}+ \|\psi\|_{L^4_{(sc)}({}^{2\delta}\!S_{u,\ub})}) \lesssim C.
\end{equation*}
This completes the proof.
\end{proof}
In the rest of the section, we derive more $L^4_{(sc)}$ estimates for later use. For curvatures,
\begin{proposition} If $\delta$ is sufficiently small, we have
\begin{equation*}
 \delta^{\frac{1}{4}}\|\alpha\|_{{L^4_{(sc)}(S_{u,\ub})}} + \|(\beta, \rho, \sigma, \betab, \alphab, K)\|_{{L^4_{(sc)}(S_{u,\ub})}} \lesssim C
\end{equation*}
\end{proposition}
\begin{proof}
For curvature component $\Psi \in \{\sigma, \rho, \betab, \alphab\}$, in view of Lemma \ref{sobolev_trace_estimates}, we have
\begin{align*}
\|\Psi\|_{{L^4_{(sc)}(S_{u,\ub})}} &\lesssim (\delta^\frac12 \|\Psi\|_{L^2_{(sc)}(H_u)} + \| \nabla \Psi\|_{L^2_{(sc)}(H_u)})^\frac{1}{2} \cdot (\delta^\frac12 \|\Psi\|_{L^2_{(sc)}(H_u)} + \| \nabla_4 \Psi\|_{L^2_{(sc)}(H_u)})^\frac{1}{2}\\
&\lesssim C(\delta^\frac{1}{2} C + \| \nabla_4 \Psi\|_{L^2_{(sc)}(H_u)})^\frac{1}{2}.
\end{align*}
To bound $\|\nabla_4 \Psi\|_{L^2_{(sc)}(H_u)}$, we use \eqref{NBE_L_sigma}, \eqref{NBE_L_rho} or \eqref{NBE_L_betab} to replace $\nabla_4 \Psi$ by $\nabla_4 \Psi  = \nabla \Psi + \psi \cdot \Psi + \Upsilon \cdot \nabla \Upsilon$. Those terms can be easily bounded. This yields the desired bound.
Similarly, we can bound $\beta$. The estimates on $K$ is directly from \eqref{NSE_gauss}.

It remains to control $\alpha$.
\begin{align*}
\|\alpha\|_{{L^4_{(sc)}(S_{u,\ub})}} &\lesssim (\delta^\frac12 \|\alpha\|_{L^2_{(sc)}(\Hb_{\ub})} + \| \nabla \alpha\|_{L^2_{(sc)}(\Hb_{\ub})})^\frac{1}{2} \cdot (\delta^\frac12 \|\alpha\|_{L^2_{(sc)}(\Hb_{\ub})} + \| \nabla_3 \alpha\|_{L^2_{(sc)}(\Hb_{\ub})})^\frac{1}{2}\\
&\lesssim C(\delta^\frac{1}{2} C + \| \nabla_3 \alpha\|_{L^2_{(sc)}(\Hb_{\ub})})^\frac{1}{2}.
\end{align*}
In view of \eqref{NBE_Lb_alpha}, $\nabla_3 \alpha = \alpha + \nabla \Psi + \psi \cdot \Psi + \Upsilon \cdot \Upsilon$. In view of the anomalous estimates on $\alpha$, this gives the bound.
\end{proof}
For connection coefficients, we have
\begin{proposition}
 If $\delta$ is sufficiently small, we have
\begin{equation}\label{B_1}
 \|\nabla \psi\|_{{L^4_{(sc)}(S_{u,\ub})}} + \|(\nabla_4 \tr \chi, \nabla_4 \eta, \nabla_4 \omegab, \nabla_4 \tr \chib, \nabla_3 \trchibt, \nabla_3 \etab, \nabla_3 \omega, \nabla_3 \tr \chi)\|_{{L^4_{(sc)}(S_{u,\ub})}}\lesssim C,
\end{equation}
\begin{equation}\label{B_2}
 \|(\nabla_4 \chih, \nabla_4 \chibh, \nabla_3 \chih, \nabla_3 \chibh)\|_{{L^4_{(sc)}(S_{u,\ub})}}\lesssim C \delta^{-\frac{1}{4}}
\end{equation}
\begin{equation}
 \|(\nabla_4 \etab, \nabla_3 \eta)\|_{{L^4_{(sc)}(S_{u,\ub})}} \lesssim C,
\end{equation}
\end{proposition}
\begin{proof}
In view of the proof of Proposition \ref{close_bootstrap}, we have estimates on $\|\nabla \psi\|_{{L^4_{(sc)}(S_{u,\ub})}}$; for the remaining terms in \eqref{B_1} and \eqref{B_2}, they can be directly derived from null structure equations. Schematically, for $N=3$ or $4$, they satisfy
\begin{equation*}
 \nabla_N \psi = \tr\chib_0 \cdot \psi + \psi \cdot \psi + \nabla \psi + \Psi + \Upsilon \cdot \Upsilon.
\end{equation*}
Each term on the right-hand side can be bounded by the estimates derived so far. The potential danger comes from $\tr\chib_0 \cdot \psi$. When $\psi \in \{\chih, \chibh\}$,  it leads to anomalous behavior in \eqref{B_2}.

It remains to deal with $\|\nabla_4 \etab\|_{{L^4_{(sc)}(S_{u,\ub})}} $ and $\|\nabla_3 \eta\|_{{L^4_{(sc)}(S_{u,\ub})}}$. They can be estimated directly from Corollary \ref{nabla_3_nabla_4_eta_etab}, Proposition \ref{twoderivativeRicciCoefficients} and Lemma \ref{sobolev_trace_estimates}. This completes the proof.
\end{proof}
\begin{remark}
All the estimates derived in the bootstrap argument of {\bf{Theorem A}} will be valid throughout the paper. Notice that we do not have the ${L^4_{(sc)}(S_{u,\ub})}$ estimates on $\nabla_4 \omega$ and $\nabla_3 \omegab$.
\end{remark}


\section{Theorem B - Trace Estimates and Angular Momentums}\label{Section_Trace_Rotation}
\subsection{Trace Norms}
We recall the definitions of trace norms introduced in \cite{K-R-09}.
\begin{definition}
For a tensor field $\psi$ along $H = H_u^{(0,\ub)}$, relative to the transported coordinates $(\ub,\theta)$, its scale invariant trace norm is defined as
\begin{equation*}
\|\psi\|_{Tr_{(sc)}(H)} = \delta^{-sc(\psi)-\frac{1}{2}} (\sup_{\theta \in S(u,0)} \int_0^{\ub} |\psi(u,\ub',\theta)|^2d\ub')^{\frac{1}{2}};
\end{equation*}
for a tensor field $\psi$ along $\Hb = \Hb_{\ub}^{(0,u)}$, relative to the transported coordinates $(u,\thetab)$, its scale invariant trace norm is defined as
\begin{equation*}
\|\psi\|_{Tr_{(sc)}(\Hb)} = \delta^{-sc(\psi)} (\sup_{\thetab \in S(0,\ub)} \int_0^{\ub} |\psi(u',\ub,\thetab)|^2du')^{\frac{1}{2}}
\end{equation*}
\end{definition}
We recall the sharp trace theorem below. We refer to \cite{K-R-09} for a proof.
\begin{proposition}\label{trace theorem}
 If $\delta$ is sufficiently small, for any $\psi$ along $H = H_u^{(0,\ub)}$, we have
\begin{align*}
\|\nabla_4 \psi\|_{Tr_{(sc)}(H)} &\lesssim (\| \nabla^2_4 \psi\|_{L^2_{(sc)}(H)} + \| \psi\|_{L^2_{(sc)}(H)} + \delta^{\frac{1}{2}}C\sup_{\ub}(\| \psi\|_{L^{\infty}_{(sc)}(u,\ub)} + \| \nabla_4 \psi\|_{L^{4}_{(sc)}(u,\ub)}))^\frac{1}{2}\\
&\times (\| \nabla^2 \psi\|_{L^2_{(sc)}(H)} + \delta^{\frac{1}{2}}C\sup_{\ub}(\| \psi\|_{L^{\infty}_{(sc)}(u,\ub)} + \| \nabla \psi\|_{L^{4}_{(sc)}(u,\ub)}))^\frac{1}{2}\\
&+ \| \nabla_4 \nabla \psi\|_{L^2_{(sc)}(H)} +  \|\nabla \psi\|_{L^2_{(sc)}(H)}+\delta^{\frac{1}{2}}C\sup_{\ub}(\| \psi\|_{L^{\infty}_{(sc)}(u,\ub)} + \| \nabla \psi\|_{L^{4}_{(sc)}(u,\ub)})),
\end{align*}
and for any $\psi$ along $\Hb = \Hb_{\ub}^{(0,u)}$, we have
\begin{align*}
\|\nabla_3 \psi\|_{Tr_{(sc)}(\Hb)} &\lesssim (\| \nabla^2_3 \psi\|_{L^2_{(sc)}(\Hb)} + \| \psi\|_{L^2_{(sc)}(\Hb)} + \delta^{\frac{1}{2}}C\sup_{u}(\| \psi\|_{L^{\infty}_{(sc)}(u,\ub)} + \| \nabla_3 \psi\|_{L^{4}_{(sc)}(u,\ub)}))^\frac{1}{2}\\
&\times (\| \nabla^2 \psi\|_{L^2_{(sc)}(\Hb)} + \delta^{\frac{1}{2}}C\sup_{u}(\| \psi\|_{L^{\infty}_{(sc)}(u,\ub)} + \| \nabla \psi\|_{L^{4}_{(sc)}(u,\ub)}))^\frac{1}{2}\\
&+ \| \nabla_3\nabla \psi\|_{L^2_{(sc)}(\Hb)} +  \|\nabla \psi\|_{L^2_{(sc)}(\Hb)}+\delta^{\frac{1}{2}}C\sup_{\ub}(\| \psi\|_{L^{\infty}_{(sc)}(u,\ub)} + \| \nabla \psi\|_{L^{4}_{(sc)}(u,\ub)})).
\end{align*}
\end{proposition}



\subsection{Trace Estimates on Curvature and Maxwell Field}
\begin{proposition}
If $\delta$ is sufficiently small, for $\Psi_g \neq \alphab$, we have
\begin{equation}
\delta^{\frac{1}{4}}(\|\alpha\|_{Tr_{(sc)}(H_u)}+\|\alphab\|_{Tr_{(sc)}(\Hb_{\ub})}) +\|(\beta,\betab)\|_{Tr_{(sc)}(H_u)} + \|(\beta,\betab)\|_{Tr_{(sc)}(\Hb_{\ub})}
\lesssim C.
\end{equation}
\end{proposition}
\begin{proof}
We first derive estimates on $\alphab$ which relies on \eqref{NSE_Lb_chibh}. To avoid $\omegab$, we rewrite \eqref{NSE_Lb_chibh} in following form
\begin{equation}\label{nablathreeomegaprime}
 \alphab = -\Omega \cdot \nabla_3 \chibh' - \tr \chib \cdot \chibh,
\end{equation}
where the modified term $\chibh' = \frac{1}{\Omega}\chibh$ enjoys the same estimates as $\chibh$. We can ignore $\tr\chib \cdot \chih$ since it enjoys better trace estimates. In view of Proposition \ref{trace theorem}, we have
\begin{align*}
\|\alphab\|_{Tr_{(sc)}(\Hb_{\ub})} &\lesssim (\| \nabla^2_3 \chibh'\|_{L^2_{(sc)}(\Hb)} + C + \delta^{\frac{1}{2}}C\| \nabla_3 \chibh'\|_{L^{4}_{(sc)}(u,\ub)})^\frac{1}{2}(C + \delta^{\frac{1}{2}}C\| \nabla \chibh'\|_{L^{4}_{(sc)}(u,\ub)})^\frac{1}{2}\\
&+ \| \nabla_3\nabla\chibh'\|_{L^2_{(sc)}(\Hb)} + C\delta^{\frac{1}{2}}\| \nabla\chibh'\|_{L^{4}_{(sc)}(u,\ub)}+C \lesssim \| \nabla^2_3 \chibh'\|^{\frac{1}{2}}_{L^2_{(sc)}(\Hb)} + C.
\end{align*}
To estimate $\|\nabla^2_3 \chibh'\|_{L^2_{(sc)}(\Hb)}$, we differentiate \eqref{nablathreeomegaprime} to derive
\begin{equation*}
\nabla_3^2 \chibh' = \nabla_3 \alphab + \tr\chib_0 \cdot \nabla_3  \chibh + \tr\chib \cdot \nabla_3  \chibh+\chibh \cdot \nabla_3 \tr \chib.
\end{equation*}
Thus,
\begin{equation*}
\| \nabla^2_3 \chibh'\|_{L^2_{(sc)}(\Hb_{\ub})} \lesssim \| \nabla_3 \alphab\|_{L^2_{(sc)}(\Hb_{\ub})}  + \| \tr\chib_0 \cdot \nabla_3 \chibh\|_{L^2_{(sc)}(\Hb_{\ub})} + C \lesssim C \delta^{-\frac{1}{2}}.
\end{equation*}
This proves the desired estimates on  $\|\alphab\|_{Tr_{(sc)}(\Hb_{\ub})}$.

We now derive estimates on $\betab$. It relies on \eqref{NSE_Lb_etab}. The proof goes exactly as above. Similarly, we can derive estimates for other quantities. This completes the proof.
\end{proof}

We now turn to the trace estimates for Maxwell components.
\begin{proposition}
If $\delta$ is sufficiently small, we have
\begin{equation}
\delta^{\frac{1}{4}}(\|\nabla_4 \alpha_F\|_{Tr_{(sc)}(H_u)}+\|\nabla_3 \alphab_F\|_{Tr_{(sc)}(\Hb_{\ub})}) +\|\nabla \alpha_F\|_{Tr_{(sc)}(H_u)} + \|\nabla \alphab_F\|_{Tr_{(sc)}(\Hb_{\ub})}\lesssim C.
\end{equation}
\end{proposition}
\begin{proof}
We derive estimates for $\alphab_F$. The bound for $\alpha_F$ is similar so that we omit the details. We start with $\nabla_3 \alphab_F$. According to Proposition \ref{trace theorem}, we have
\begin{align*}
\|\nabla_3 \alphab_F\|_{Tr_{(sc)}(\Hb)} &\lesssim (\| \nabla^2_3 \alphab_F\|_{L^2_{(sc)}(\Hb)} + C + C\delta^{\frac{1}{2}}\sup_{u} \| \nabla_3 \alphab_F\|_{L^{4}_{(sc)}(u,\ub)})^\frac{1}{2} (\| \nabla^2 \alphab_F\|_{L^2_{(sc)}(\Hb)} + C)^\frac{1}{2}\\
&+ \| \nabla_3\nabla \alphab_F\|_{L^2_{(sc)}(\Hb)} +  \|\nabla \alphab_F\|_{L^2_{(sc)}(\Hb)}+C \delta^{\frac{1}{2}}\\
&\lesssim (C\delta^{-\frac{1}{2}} + C)^{\frac{1}{2}} + \| \nabla_3\nabla \alphab_F\|_{L^2_{(sc)}(\Hb)}+C \lesssim  C \delta^{-\frac{1}{4}} + \| \nabla_3\nabla \alphab_F\|_{L^2_{(sc)}(\Hb)}.
\end{align*}
It suffices to show that $\|\nabla_3\nabla \alphab_F\|_{L^2_{(sc)}(\Hb)} \lesssim C\delta^{-\frac{1}{4}}$. The idea is, by integration by parts, we can move bad derivatives to one component to save a $\delta^{\frac{1}{4}}$. In fact,
\begin{align*}
\|\nabla_3\nabla \alphab_F\|^2_{{L^2}(\Hb)} = \int_{\Hb} \nabla_3\nabla \alphab_F \cdot \nabla_3\nabla \alphab_F = \int_{\Hb} \nabla^2 \alphab_F \cdot \nabla_3\nabla_3 \alphab_F + E,
\end{align*}
where error term $E$ comes from boundary terms of integration by parts and commutator of $[\nabla_3,\nabla]$. It is easy to see that in this form, we have $\|\nabla_3\nabla \alphab_F\|^2_{L^2_{(sc)}(\Hb)} \lesssim C\delta^{-\frac{1}{2}}$. This yields the desired estimates.

For the bound of $\nabla \alphab_F$, we introduce an auxiliary tensor $\varphi$
\begin{equation}\label{auxillary_varphi}
 \nabla_3 \varphi = \nabla \alphab_F, \quad \varphi(0,\ub) = 0.
\end{equation}
By commuting derivatives, together with Gronwall's inequality and all the estimates derived so far, one can easily show that
\begin{equation*}
 \|\varphi\|_{{L^2_{(sc)}(S_{u,\ub})}} + \|\varphi\|_{{L^4_{(sc)}(S_{u,\ub})}} + \|\nabla \varphi\|_{{L^2_{(sc)}(S_{u,\ub})}} + \|\nabla_3 \nabla \varphi\|_{L^2_{(sc)}(\Hb_{\ub})} \lesssim C.
\end{equation*}
We claim a more serious bound on $\varphi$,
\begin{equation}
 \|\nabla^2 \varphi\|_{L^2_{(sc)}(\Hb_{\ub})} \lesssim  C.
\end{equation}
We use $l.o.t.$ to denote terms which either are much easier to estimate or enjoy better estimates. We commute $\triangle$ with \eqref{auxillary_varphi} to derive
\begin{equation*}
\nabla_3 \triangle \varphi = \nabla \triangle \alphab_F + l.o.t.
\end{equation*}
We apply $\Dtwostar$ on \eqref{NM_Lb_rho} and \eqref{NM_Lb_sigma} to replace $\triangle \alphab_F$ by $\nabla_3 \nabla (\rho, \sigma) + l.o.t$. This allows us to do a renormalization to rectify $\nabla \triangle \alphab_F$ as follows,
\begin{equation*}
\nabla_3 \triangle \varphi =\nabla \nabla_3 \nabla (\rho, \sigma)+ l.o.t. =\nabla_3 \nabla^2 (\rho, \sigma) + l.o.t.
\end{equation*}
Hence, we derive
\begin{equation*}
\nabla_3 (\triangle \varphi- \nabla^2 (\rho, \sigma))= l.o.t.
\end{equation*}
This equation allows one to derive estimates on $\|\triangle \varphi - \nabla^2 (\rho, \sigma)\|_{L^2_{(sc)}(\Hb_{\ub})}$ hence on $\|\nabla^2 \varphi \|_{L^2_{(sc)}(\Hb_{\ub})}$. Finally, we use Proposition \ref{trace theorem} and Lemma \ref{sobolev_trace_estimates} to control $\|\nabla_3 \varphi\|_{Tr_{(sc)}(\Hb)} $. In view of the definition of $\varphi$, this completes the proof.
\end{proof}
\begin{remark}
Similarly, we can easily derive that, for $D\Upsilon$ = $\nabla \rho_F$, $\nabla \sigma_F$, $\nabla_3 \alpha_F$, $\nabla_4\alphab_F$, $\nabla_4 \rho_F$, $\nabla_3 \rho_F$, $\nabla_4 \sigma_F$ or $\nabla_3 \sigma_F$, we have
\begin{equation*}
\|D\Upsilon\|_{Tr_{(sc)}(H_u)} + \|D\Upsilon\|_{Tr_{(sc)}(\Hb_{\ub})}\lesssim C.
\end{equation*}
\end{remark}

\subsection{Trace Estimates on connection coefficients}

\subsubsection{Estimates on $\| \nabla \chih\|_{Tr_{(sc)}(H_u)}$ and $\| \nabla \chibh\|_{Tr_{(sc)}(\Hb_{\ub})}$}
In order to control $\| \nabla \chih\|_{Tr_{(sc)}(H_u)}$ and $\| \nabla \chibh\|_{Tr_{(sc)}(\Hb_{\ub})}$, we introduce two auxiliary tensors $\phi$ and $\phib$,
\begin{align}\label{def_phi_phib}
 \nabla_4 \phi = \nabla \chih \quad &\text{on} \quad H_u,  \qquad \phi(u,0)=0, \\
\nabla_3 \phib = \nabla \chibh \quad &\text{on} \quad \Hb_{\ub},  \qquad \phib(0,\ub)=0.
\end{align}
\begin{proposition}\label{trace_estimates_chih_chibh} If $\delta$ is sufficiently small, we have
\begin{equation}\label{tr_phi_easy}
(1).\quad \|(\phi, \nabla \phi, \nabla_4 \phi)\|_{{L^2_{(sc)}(S_{u,\ub})}} + \|\phi\|_{{L^4_{(sc)}(S_{u,\ub})}} +  \|(\nabla \nabla_4 \phi, \nabla^2_4 \phi)\|_{L^2_{(sc)}(H_u)}  \lesssim C,
\end{equation}
\begin{equation}\label{tr_phi_hard}
 \|\nabla^2 \phi\|_{L^2_{(sc)}(H_u)} \lesssim C + \|\nabla^3 \tr \chi\|_{L^2_{(sc)}(H_u)},
\end{equation}
\begin{equation}\label{tr_phib_easy}
(2).\quad  \|(\phib, \nabla \phib, \nabla_3 \phib) \|_{{L^2_{(sc)}(S_{u,\ub})}} + \|\phib\|_{{L^4_{(sc)}(S_{u,\ub})}} +  \|(\nabla \nabla_3 \phib, \nabla^2_3 \phib)\|_{L^2_{(sc)}(\Hb_{\ub})} \lesssim C,
\end{equation}
\begin{equation}\label{tr_phib_hard}
 \|\nabla^2 \phib\|_{L^2_{(sc)}(\Hb_{\ub})} \lesssim C + \|\nabla^3 \tr \chib\|_{L^2_{(sc)}(\Hb_{\ub})},
\end{equation}
\end{proposition}
\begin{remark}\label{remark nabla chih chihb trace}
As consequences of Lemma \ref{sobolev_trace_estimates} and Proposition \ref{trace theorem}, we have
\begin{equation}
 \|\phi\|_{L^{\infty}_{(sc)}(H_u)} + \|(\nabla_4 \phi, \nabla \chih)\|_{Tr_{(sc)}(H_u)} \lesssim C + \|\nabla^3 \tr \chi\|^{\frac{1}{2}}_{L^2_{(sc)}(H_u)},
\end{equation}
\begin{equation}
 \|\phib\|_{L^{\infty}_{(sc)}(\Hb_{\ub})} + \|(\nabla_3 \phi,\nabla \chibh)\|_{Tr_{(sc)}(\Hb_{\ub})} \lesssim C + \|\nabla^3 \tr \chib\|^{\frac{1}{2}}_{L^2_{(sc)}(\Hb_{\ub})}.
\end{equation}
\end{remark}
\begin{proof}
Because of the presence of $\tr\chib_0$, we expect the estimates along incoming null hypersurfaces are more challenging. We will only prove the second part of the proposition. We observe that \eqref{tr_phib_easy} is a direct use of the estimates derived through the proof of {\bf{Theorem A}}. We turn to the proof of \eqref{tr_phib_hard}.

Let $\psi_3 \in \{\trchibt, \chibh, \eta, \etab\}$. In particular, $\omega$ and $\omegab$ are prohibited.
 We commute derivatives to derive
\begin{equation}
 [\nabla_3, \nabla] \phib = \chib \cdot \nabla\phib + \nabla \psi_3 \cdot \phib + \psi_3 \cdot \nabla_3\phib + \Upsilon \cdot \Upsilon \cdot \phib + \chib \cdot \psi_3 \cdot \phib,
\end{equation}
\begin{align*}
 [\nabla_3, \nabla^2] \phib &= \chib \cdot \nabla^2 \phib + \psi_3 \cdot (\nabla_3 \nabla \phib + \nabla \nabla_3 \phib) + \nabla \psi_3 \cdot (\nabla \phib + \nabla_3 \phib) + \nabla^2 \psi_3 \cdot \phib\\
&\quad + \Upsilon \cdot \Upsilon  \cdot \nabla \phib + \Upsilon \cdot \nabla \Upsilon  \cdot \phib + \chib \cdot \psi_3 \cdot \nabla \phib + \psi_3 \cdot \nabla \psi_3 \cdot \phib + \chib \cdot \nabla \psi_3 \cdot \phib.
\end{align*}
Hence, by taking trace, we have
\begin{align*}
 [\nabla_3, \triangle] \phib &= \tr \chib_0 \cdot \nabla^2 \phib + \chibh \cdot \nabla^2 \phib + E_1.
\end{align*}
We claim that error term $E_1$ satisfies $\|E_1\|_{L^2_{(sc)}(\Hb_{\ub})} \lesssim C + C\delta^{\frac{1}{2}} \|\nabla^2 \phib\|_{L^2_{(sc)}(\Hb_{\ub})}$. The idea is to use either $L^4_{(sc)}$ estimates on $\nabla \phib$ or $L^\infty_{(sc)}$ estimates on $\phib$. In view of Lemma \ref{sobolev_trace_estimates}, in either case the quantity is bounded by $\|\nabla^2 \phib\|_{L^2_{(sc)}(\Hb_{\ub})}$. We only deal with one term in $E_1$, say $\Upsilon \cdot \nabla \Upsilon  \cdot \phib$,  to illustrate the idea.
\begin{align*}
\|\Upsilon \cdot  \nabla \Upsilon &  \cdot \phib\|_{L^2_{(sc)}(\Hb_{\ub})} \lesssim \delta \|\Upsilon\|_{L^{\infty}_{(sc)}} \|\nabla \Upsilon\|_{L^2_{(sc)}(\Hb_{\ub})}\|\phib\|_{L^{\infty}_{(sc)}(\Hb_{\ub})}\lesssim C \delta \|\phib\|_{L^{\infty}_{(sc)}(\Hb_{\ub})}\\
&\lesssim C \delta (\|\nabla^2 \phib\|_{L^2_{(sc)}(\Hb_{\ub})}+\|\nabla \nabla_3 \phib\|_{L^2_{(sc)}(\Hb_{\ub})} + \| \phib\|_{L^2_{(sc)}(\Hb_{\ub})})\lesssim C + C\delta^{\frac{1}{2}} \|\nabla^2 \phib\|_{L^2_{(sc)}(\Hb_{\ub})}.
\end{align*}
Thus,
\begin{align*}
\nabla_3 \triangle \phib &= \triangle \nabla \chih + \tr \chib_0 \nabla^2 \phib + \chibh \cdot \nabla^2 \phib + E_1=\nabla  \triangle \chih + \tr \chib_0 \nabla^2 \phib + \chibh \cdot \nabla^2 \phib + [\triangle, \nabla] \chih + E_1
\end{align*}
Since, on $\Hb_{\ub}$,
\begin{equation*}
\|[\triangle, \nabla] \chih \|_{L^2_{(sc)}} = \|K \nabla \phib + \nabla K \cdot \phib \|_{L^2_{(sc)}}\lesssim \delta^{\frac{1}{2}}(\|K\|_{L^4_{(sc)}}\|\phib\|_{L^4_{(sc)}}+\|\nabla K\|_{L^2_{(sc)}}\|\phib\|_{L^{\infty}_{(sc)}})\lesssim C,
\end{equation*}
we derive
\begin{equation}\label{tteq_1}
\nabla_3 \triangle \phib = \nabla  \triangle \chih + \tr \chib_0 \nabla^2 \phib + \chibh \cdot \nabla^2 \phib + E_2,
\end{equation}
with error term $E_2$ satisfying $\|E_2\|_{L^2_{(sc)}(\Hb_{\ub})} \lesssim C + C\delta^{\frac{1}{2}} \|\nabla^2 \phib\|_{L^2_{(sc)}(\Hb_{\ub})}$. We then apply $\Dtwostar$ to \eqref{NSE_div_chibh} to derive
\begin{align}\label{laplacian_chibh}
 \triangle \chibh & = \Dtwostar \betab +\nabla^2 \tr\chib + \nabla (\tr\chib \cdot \psi_3 + \psi_3 \cdot \psi_3 + \Upsilon \cdot \Upsilon) + K \cdot \chibh.
\end{align}
We differentiate the equation once more to derive
\begin{align*}
 \nabla \triangle \chibh 
& = \nabla^2 \betab +\nabla^3 \tr\chib + K \cdot \nabla_3 \phib + E_3,
\end{align*}
with error term $\|E_3\|_{L^2_{(sc)}(\Hb_{\ub})} \lesssim C$. Combined with \eqref{tteq_1}, we have
\begin{align}\label{tteq_2}
\nabla_3 \triangle \phib =  \nabla^2 \betab + K \cdot \nabla_3 \phib + \nabla^3 \tr\chib  + \tr \chib_0 \cdot \nabla^2 \phib + \chibh \cdot \nabla^2 \phib + E_2 + E_3 .
\end{align}
For $\nabla^2 \betab$, in view of \eqref{NSE_Lb_etab}, we transform it into a term involving $\nabla_3$ derivative,
\begin{equation*}
 \nabla^2 \betab = \nabla^2 \nabla_3 \etab + \nabla^2 (\chib \cdot \psi_3) + \nabla^2(\Upsilon \cdot \Upsilon) = \nabla_3 \nabla^2 \etab + E_4,
\end{equation*}
with error term $\|E_4\|_{L^2_{(sc)}(\Hb_{\ub})} \lesssim C$. Combined with \eqref{tteq_2}, we have
\begin{equation*}
\nabla_3 \triangle \phib = \nabla_3 \nabla^2 \etab + \nabla_3 (K \cdot \phib) + \nabla^3 \tr\chib  + \tr \chib_0 \nabla^2 \phib + \chibh \cdot \nabla^2 \phib + E_5,
\end{equation*}
where $E_5 = E_2 + E_3 + E_4 - \nabla_3 K  \cdot \phib$. Since $\nabla_3 K \cdot \phib$ can be easily bounded according to \eqref{NSE_gauss}, we have
 \begin{equation*}
 \|E_5\|_{L^2_{(sc)}(\Hb_{\ub})} \lesssim C + C\delta^{\frac{1}{2}} \|\nabla^2 \phib\|_{L^2_{(sc)}(\Hb_{\ub})}.
\end{equation*}
Putting things together, we have derived the renormalized equation
\begin{equation}\label{renormalized_1}
\nabla_3 (\triangle \phib - \nabla^2 \etab - K \cdot \phib) =  \nabla^3 \tr\chib  + \tr \chib_0 \cdot \nabla^2 \phib + \chibh \cdot \nabla^2 \phib + E_5.
\end{equation}
This formula allows us to derive
\begin{equation*}
\|\triangle \phib - \nabla^2 \etab - K \cdot \phib \|_{{L^2_{(sc)}(S_{u,\ub})}} \lesssim C + \|\nabla^3\tr\chib + E_5\|_{L^2_{(sc)}(\Hb_{\ub})} +  \int_0^u \|\nabla^2 \phib\|_{{L^2_{(sc)}(S_{u,\ub})}},
\end{equation*}
thus, on $S=S_{u,\ub}$ and $\Hb=\Hb_{\ub}$, we have
\begin{align*}
\| \triangle \phib \|_{{L^2_{(sc)}(S)}} &\lesssim C + \| \nabla^2 \etab \|_{{L^2_{(sc)}(S)}} + C\delta^{\frac{1}{2}} \|\nabla^2 \phib\|_{L^2_{(sc)}(\Hb)}+ \|\nabla^3\tr\chib\|_{L^2_{(sc)}(\Hb)}+ \int_0^u \|\nabla^2 \phib\|_{{L^2_{(sc)}(S_{u,\ub})}}.
\end{align*}
Standard elliptic estimates imply
\begin{align*}
 \|\nabla^2 \phib\|_{{L^2_{(sc)}(S_{u,\ub})}} &\lesssim  \delta^{\frac{1}{2}}\|K\|_{{L^2_{(sc)}(S_{u,\ub})}}\|\phib\|_{L^{\infty}_{(sc)}} + \delta^{\frac{1}{4}}\|K\|^{\frac{1}{2}}_{{L^2_{(sc)}(S_{u,\ub})}}\|\nabla \phib\|_{L^{4}_{(sc)}(u,\ub)} +\|\triangle \phib\|_{{L^2_{(sc)}(S_{u,\ub})}}\\
& \lesssim \|\triangle \phib \|_{{L^2_{(sc)}(S_{u,\ub})}} + C + C \delta^{\frac{1}{2}}\|\nabla^2 \phib\|_{L^2_{(sc)}(\Hb_{\ub})}.
\end{align*}
Hence,
\begin{align*}
\| \nabla^2 \phib \|_{{L^2_{(sc)}(S)}} &\lesssim C + \| \nabla^2 \etab \|_{{L^2_{(sc)}(S)}}  + C\delta^{\frac{1}{2}} \|\nabla^2 \phib\|_{L^2_{(sc)}(\Hb)} + \|\nabla^3\tr\chib\|_{L^2_{(sc)}(\Hb)} + \int_0^u \|\nabla^2 \phib\|_{{L^2_{(sc)}(S_{u,\ub})}}
\end{align*}
We remove the last term by Gronwall's inequality to derive
\begin{align*}
\| \nabla^2 \phib \|_{{L^2_{(sc)}(S_{u,\ub})}} &\lesssim C + \| \nabla^2 \etab \|_{{L^2_{(sc)}(S_{u,\ub})}} + \|\nabla^3\tr\chib\|_{L^2_{(sc)}(\Hb_{\ub})}+ C\delta^{\frac{1}{2}} \|\nabla^2 \phib\|_{L^2_{(sc)}(\Hb_{\ub})}.
\end{align*}
We then integrate on $\Hb_{\ub}$ and use the smallness of $\delta$ to complete the proof.
\end{proof}


\subsubsection{Estimates on $\|\nabla^3 \tr \chi\|_{L^2_{(sc)}(H_u)}$ and $\| \nabla^3 \tr\chib\|_{L^2_{(sc)}(\Hb_{\ub})}$}\label{L_2_estimates_for_nabla^3_tr_chi}

\begin{proposition}
If $\delta$ is sufficiently small, we have
\begin{equation}\label{trace_trchib}
\|\nabla^3 \tr \chi\|_{L^2_{(sc)}(H_u)}+\|\nabla^3 \tr \chib\|_{L^2_{(sc)}(\Hb_{\ub})}  \lesssim C.
\end{equation}
\end{proposition}

\begin{remark}As a consequence, we have
\begin{equation}\label{remark nabla chibh nabla chih trace norm}
 \|(\nabla \tr\chi, \nabla \tr\chib)\|_{L^{\infty}_{(sc)}} + \|(\nabla \chih, \nabla \chibh) \|_{L^{4}_{(sc)}(u,\ub)} +  \|\nabla \chibh\|_{Tr_{(sc)}(\Hb_{\ub})}+\|\nabla \chibh\|_{Tr_{(sc)}(\Hb_{\ub})} \lesssim C.
\end{equation}
\end{remark}

\begin{proof}
We bound the second term of \eqref{trace_trchib}. The first one follows in the same manner.  To avoid unpleasant derivatives on $\omegab$, we define $\tr \chib' = \Omega^{-1}\tr\chib$ and use the following form of \eqref{NSE_Lb_tr_chib},
\begin{equation}\label{normalize_NSE_tr_chib}
 \nabla_3 \tr \chib' = -\frac{1}{2}\Omega(\tr \chib ')^2- \Omega^{-1}|\chibh|^2- \Omega^{-1}|{\alphab_F}|^2.
\end{equation}
By a direct but tedious commutation of derivative, we can show that
\begin{align*}
 [\nabla_3, \nabla^3] \tr\chib' &= \chib \cdot \nabla^3 \tr \chib' + \psi_3 \cdot (\nabla_3 \nabla^2 \tr \chib' + \nabla \nabla_3 \nabla \tr \chib' + \nabla^2 \nabla_3 \tr \chib') \\
& + [(\nabla_3 \psi_3 + \nabla \chib + \psi_3 \cdot \psi_3 + \Upsilon \cdot \Upsilon) \cdot \nabla^2 \tr \chib' + \nabla \psi_3 \cdot (\nabla_3 \nabla \tr \chib' + \nabla \nabla_3 \tr \chib') ]\\
&+ [(\nabla^2 \psi_3 + \psi_3 \cdot \nabla \psi_3 + \Upsilon \cdot \nabla \Upsilon) \cdot \nabla \tr \chib' + \nabla^2 \psi_3 \cdot \nabla_3 \tr \chib']\\
&=\chib \cdot \nabla^3 \tr \chib' + F_1.
\end{align*}
We claim that error term $F_1$ enjoys
\begin{equation}\label{eq error e_1}
 \|F_1\|_{L^2_{(sc)}(\Hb_{\ub})} \lesssim C + C \delta^{\frac{1}{2}} \|\nabla^3 \tr \chib\|^{\frac{1}{2}}_{L^2_{(sc)}(\Hb_{\ub})}.
\end{equation}
To prove \eqref{eq error e_1}, we first observe that, since $\nabla \Omega = \frac{1}{2}(\eta + \etab)$, we can simply ignore all $\Omega$'s in estimates. The second observation is the following Sobolev estimates,
\begin{equation}\label{kf_1}
 \|\nabla^2 \tr\chib\|_{{L^4_{(sc)}(S_{u,\ub})}} +  \|\nabla \tr\chib\|_{L^{\infty}_{(sc)}} \lesssim C \delta^{\frac{1}{4}} +  \|\nabla^3 \tr \chib\|^{\frac{1}{2}}_{L^2_{(sc)}(\Hb_{\ub})}
\end{equation}
The bound \eqref{eq error e_1}for $F_1$ is immediate by combining those two observations and the estimates derived so far.

Hence, we derive
\begin{equation*}
 \nabla_3 \nabla \triangle  \tr\chib' = \chibh \cdot  \nabla \triangle \chibh + \nabla \chibh \cdot \nabla^2 \chibh + \chib \cdot \nabla^3 \tr\chib + \nabla \chib \cdot \nabla^2 \tr\chib +\nabla \triangle  (|{\alphab_F}|^2) + F_2,
\end{equation*}
with error term $\|F_2\|_{L^2_{(sc)}(\Hb_{\ub})} \lesssim C + C \delta^{\frac{1}{2}} \|\nabla^3 \tr \chib\|_{L^2_{(sc)}(\Hb_{\ub})}$. We turn to the renormalization of terms with three angular derivatives. For $\nabla \triangle  (|{\alphab_F}|^2)$, in view of \eqref{NM_Lb_rho} and \eqref{NM_Lb_sigma}, we have $\nabla_3 ({\rho_F}, {\sigma_F}) = \D_1 {\alphab_F} + \psi_3 \cdot \Upsilon$. After using $\Donestar$ and commuting derivatives, this equation is reduced to
\begin{equation*}
 \nabla_3 \nabla ( {\rho_F}, {\sigma_F}) = \triangle {\alphab_F} + \psi_3 \cdot \nabla \Upsilon + \nabla \psi_3 \cdot \Upsilon + \psi_3 \cdot \Upsilon.
\end{equation*}
We differentiate the equation once more to derive
\begin{equation*}
 \nabla_3 \nabla^2 \Upsilon_g = \nabla \triangle {\alphab_F} + F_3,
\end{equation*}
with error term $\|F_3\|_{L^2_{(sc)}(\Hb_{\ub})} \lesssim C$. In view of \eqref{tteq_1} in the previous subsection,
we have
\begin{align*}
 \nabla_3 \nabla \triangle  \tr\chib' &= \chibh \cdot  \nabla_3 \triangle \phib  + \nabla \chibh \cdot \nabla^2 \chibh + \chib \cdot  \nabla^3 \tr\chib + \nabla \chib\cdot \nabla^2 \tr\chib + \nabla_3 \nabla^2 \Upsilon_g  \cdot {\alphab_F}  \\
&\qquad + \nabla^2 \alphab_F \cdot \nabla \alphab_F + \chih \cdot(\tr \chib_0 \nabla^2 \phib + \chibh \cdot \nabla^2 \phib + E_2) + F_2.
\end{align*}
Hence,
\begin{align}\label{renormalized_2}
 &\quad \nabla_3 (\nabla \triangle  \tr\chib' - \chih \triangle \phib - {\alphab_F} \nabla^2 \Upsilon_g) =T_1 + T_2 + \cdots + T_5 + F_4\\
&= \nabla_3 \chibh \cdot \triangle \phib  + \nabla \chibh \cdot \nabla^2 \chibh + \nabla \chib\cdot \nabla^2 \tr\chib + \nabla^2 \Upsilon_g \cdot \nabla_3 {\alphab_F} + \nabla^2 {\alphab_F} \cdot \nabla {\alphab_F} + F_4, \notag
\end{align}
with error term $ \|F_4\|_{L^2_{(sc)}(\Hb_{\ub})} \lesssim C + C \delta^{\frac{1}{2}} \|\nabla^3 \tr \chib\|^{\frac{1}{2}}_{L^2_{(sc)}(\Hb_{\ub})}$.
In the derivation of \eqref{renormalized_2}, we have ignored $\chib \cdot\nabla^3 \tr\chib$ since it will be eventually eliminated by the Gronwall's inequality.

We the estimate $T_i$'s along $\Hb_{\ub}$ one by one. For $T_1$, in view of Proposition \ref{trace_estimates_chih_chibh}, we have
\begin{equation*}
 \|T_1\|_{L^{2}_{(sc)}(\Hb_{\ub})} \lesssim \delta^{\frac{1}{2}}\|\nabla_3 \chibh \|_{Tr_{(sc)}(\Hb_{\ub})} \|\triangle \phib\|_{L^{2}_{(sc)}(\Hb_{\ub})}\lesssim \delta^{\frac{1}{2}}\|\nabla_3 \chibh \|_{Tr_{(sc)}(\Hb_{\ub})}(\|\nabla^3 \tr \chib\|_{L^2_{(sc)}(\Hb_{\ub})}+C)
\end{equation*}
In view of \eqref{NSE_Lb_chibh}, we have $\|\nabla_3 \chibh \|_{Tr_{(sc)}(\Hb_{\ub})} \lesssim C + \|\alphab\|_{Tr_{(sc)}(\Hb_{\ub})}$. Thus, we have
\begin{equation*}
 \|T_1\|_{L^{2}_{(sc)}(\Hb_{\ub})} \lesssim C \delta^{\frac{1}{4}}\|\nabla^3 \tr \chib\|_{L^2_{(sc)}(\Hb_{\ub})}+C.
\end{equation*}

For $T_2$, in view of Remark \ref{remark nabla chih chihb trace}, we have
\begin{equation*}
 \|T_2\|_{L^{2}_{(sc)}(\Hb_{\ub})} \lesssim \delta^{\frac{1}{2}}\|\nabla \chibh \|_{Tr_{(sc)}(\Hb_{\ub})} \|\nabla^2 \chibh\|_{L^{2}_{(sc)}(\Hb_{\ub})} \lesssim C \delta^{\frac{1}{2}} + C \delta^{\frac{1}{2}}\|\nabla^3 \tr \chib\|^{\frac{1}{2}}_{L^2_{(sc)}(\Hb_{\ub})}.
\end{equation*}

For $T_3$, we use two more derivative to bound $\|\nabla \tr\chib \|_{L^{\infty}_{(sc)}}$, thus,
\begin{align*}
 \|T_3\|_{L^{2}_{(sc)}(\Hb_{\ub})} &\lesssim \delta^{\frac{1}{2}}\|\nabla \tr\chib \|_{L^{\infty}_{(sc)}} \|\nabla^2 \tr\chib\|_{L^{2}_{(sc)}(\Hb_{\ub})} + \delta^{\frac{1}{2}}\|\nabla \chibh \|_{Tr_{(sc)}(\Hb_{\ub})} \|\nabla^2 \tr\chib\|_{L^{2}_{(sc)}(\Hb_{\ub})}\\
&\lesssim C\delta^{\frac{1}{2}} + C \delta^{\frac{1}{2}}\|\nabla^3 \tr \chib\|^{\frac{1}{2}}_{L^2_{(sc)}(\Hb_{\ub})}.
\end{align*}

For $T_4$, we have
\begin{equation*}
 \|T_4\|_{L^{2}_{(sc)}(\Hb_{\ub})} \lesssim \delta^{\frac{1}{2}}\|\nabla_3 {\alphab_F} \|_{Tr_{(sc)}(\Hb_{\ub})} \|\nabla^2 \Upsilon\|_{L^{2}_{(sc)}(\Hb_{\ub})} \lesssim C.
\end{equation*}

For $T_5$, we have
\begin{equation*}
 \|T_5\|_{L^{2}_{(sc)}(\Hb_{\ub})} \lesssim \delta^{\frac{1}{2}}\|\nabla {\alphab_F}\|_{Tr_{(sc)}(\Hb_{\ub})} \|\nabla^2 {\alphab_F}\|_{L^{2}_{(sc)}(\Hb_{\ub})} \lesssim C  \delta^{\frac{1}{2}}.
\end{equation*}
Thus \eqref{renormalized_2} implies
\begin{align*}
 \|\nabla \triangle  \tr\chib'\|_{L^{2}_{(sc)}(u,\ub)} & \lesssim C+\|(\chih \cdot \nabla \phib, \,\,\alphab_F \cdot \nabla^2 \Upsilon_g)\|_{L^{2}_{(sc)}(u,\ub)} + C\delta^{\frac{1}{4}}\|\nabla^3 \tr \chib\|_{L^2_{(sc)}(\Hb_{\ub})}\\
&\lesssim  C+ C \delta^{\frac{1}{4}} \|\nabla^3 \tr \chib\|_{L^2_{(sc)}(\Hb_{\ub})}.
\end{align*}
Thus, we have $\|\nabla ^3 \tr\chib\|_{L^{2}_{(sc)}(u,\ub)} \lesssim  C+ C \delta^{\frac{1}{4}} \|\nabla^3 \tr \chib\|_{L^2_{(sc)}(\Hb_{\ub})}$. We then integrate on $\Hb_{\ub}$ and use the smallness of $\delta$ to complete the proof.
\end{proof}


\subsubsection{Estimates on $\|(\nabla \eta, \nabla \etab)\|_{Tr_{(sc)}(H_u)}$ and $\| (\nabla \eta, \nabla \etab)\|_{Tr_{(sc)}(\Hb_{\ub})}$}\label{Tr_estimates_for_nabla_eta_etab}
In order to obtain trace estimates on $\nabla \eta$ and $\nabla \etab$, we introduce the following auxiliary tensors.
\begin{align}\label{def_phift}
 \nabla_4 (\phif, \phibf)  = (\nabla \eta, \nabla \etab) \quad &\text{on} \quad H_u,  \quad (\phif(u,0), \phibf(u,0))=0, \notag\\
 \nabla_3 (\phit,\phibt) = (\nabla \eta, \nabla \etab) \quad &\text{on} \quad \Hb_{\ub},  \quad (\phit(0,\ub),\phibt(0,\ub))=0.
\end{align}
We also define $\varphi = (\phif,\phibf)$ and $\varphib = (\phit,\phibt)$. They can be estimated as follows.
\begin{proposition} If $\delta$ is sufficiently small, we have
\begin{equation}\label{estimates_varphi_4_easy}
 \|(\varphi, \nabla \varphi, \nabla_4 \varphi)\|_{L^{2}_{(sc)}(u,\ub)} + \|\varphi\|_{L^{4}_{(sc)}(u,\ub)} + \|(\nabla \nabla_4 \varphi, \nabla_4^2  \varphi)\|_{L^2_{(sc)}(H_u)} \lesssim C,
\end{equation}
\begin{equation}\label{estimates_varphi_4_hard}
\|\nabla^2  \varphi\|_{L^2_{(sc)}(H_u)} \lesssim \|\nabla^2  \mu\|_{L^2_{(sc)}(H_u)} + C,
\end{equation}
\begin{equation}\label{estimates_varphi_3_easy}
 \|(\varphib, \nabla \varphib, \nabla_3 \varphib)\|_{L^{2}_{(sc)}(u,\ub)} + \|\varphib\|_{L^{4}_{(sc)}(u,\ub)} + \|(\nabla \nabla_3 \varphib, \nabla_3^2  \varphib)\|_{L^2_{(sc)}(\Hb_{\ub})} \lesssim C,
\end{equation}
\begin{equation}\label{estimates_varphi_3_hard}
\|\nabla^2  \varphib\|_{L^2_{(sc)}(\Hb_{\ub})} \lesssim \|\nabla^2  \mub\|_{L^2_{(sc)}(\Hb_{\ub})} + C.
\end{equation}
\end{proposition}
\begin{remark}
As consequences of Lemma \eqref{sobolev_trace_estimates} and Proposition \ref{trace theorem}, we have
\begin{equation*}
\|\varphi\|_{L^{\infty}_{(sc)}} +\|\nabla_4  \varphi\|_{Tr_{(sc)}(H_u)}+  \|(\nabla \eta, \nabla \etab)\|_{Tr_{(sc)}(H_u)} \lesssim \|\nabla^2  \mu\|_{L^2_{(sc)}(H_u)} + C,
\end{equation*}
\begin{equation*}
\|\varphib\|_{L^{\infty}_{(sc)}} +\|\nabla_3  \varphib\|_{Tr_{(sc)}(\Hb_{\ub})} + \|(\nabla \eta, \nabla \etab)\|_{Tr_{(sc)}(\Hb_{\ub})} \lesssim \|\nabla^2  \mub\|_{L^2_{(sc)}(\Hb_{\ub})} + C.
\end{equation*}
\end{remark}
\begin{proof}
\eqref{estimates_varphi_4_easy} and \eqref{estimates_varphi_3_easy} are extremely easy to prove by using the estimates derived so far. For the remaining estimates, we shall prove \eqref{estimates_varphi_3_hard}; similarly, one can derive \eqref{estimates_varphi_4_hard}.

First of all, we commute $\triangle$ with \eqref{def_phift} to derive
\begin{equation}\label{nabla_3_2derivative_varphib}
\nabla_3 \triangle \varphib = \nabla \triangle (\eta, \etab) + \tr\chib_0 \cdot \nabla^2 \varphib + \chibh \cdot \nabla^2 \varphib + G_1
\end{equation}
with error term $\|G_1\|_{L^2_{(sc)}(\Hb_{\ub})} \lesssim C+ C^{\frac{1}{2}}\|\nabla^2 \varphib\|_{L^2_{(sc)}(\Hb_{\ub})}$.

Secondly, we commute $\Donestar$ with the following Hodge system
\begin{equation*}
\divergence \etab = -\mub -\rho, \quad \curl \etab = -\sigma + \chih \wedge \chibh,
\end{equation*}
to derive
\begin{equation*}
\nabla \triangle \etab = \nabla^2 \mub + \nabla^2(\rho, \sigma) + G'_2,
\end{equation*}
with error term $\|G'_2\|_{L^2_{(sc)}(\Hb_{\ub})} \lesssim C$.
Similarly, we can derive
\begin{equation*}
\nabla \triangle \eta = \nabla_3 \nabla^2 \tr\chi + \nabla^2(\rho, \sigma) + G''_2,
\end{equation*}
with error term $\|G''_2\|_{L^2_{(sc)}(\Hb_{\ub})} \lesssim C$. Putting things together, we have
\begin{equation}\label{3derivatives_etas}
\nabla \triangle (\eta,\etab) = \nabla_3 \nabla^2 \tr\chi + \nabla^2 \mub + \nabla^2(\rho, \sigma) + G_2
\end{equation}
with error term $\|G_2\|_{L^2_{(sc)}(\Hb_{\ub})} \lesssim C$. Recall the following transport equation for $(\omega, \omegat)$ and we keep the exact coefficients of $\rho$ and $\sigma$.
\begin{equation*}
\nabla_3 (\omega,\omegat) =\frac{1}{2}(\rho,\sigma) + \psi \cdot \psi + \Upsilon \cdot \Upsilon
\end{equation*}
By commuting derivatives, we have
\begin{equation}\label{2derivatives_omegas}
\nabla^2(\rho,\sigma) =\nabla_3 \nabla^2 (\omega, \omegat) + G_3,
\end{equation}
with error term $\|G_3\|_{L^2_{(sc)}(\Hb_{\ub})} \lesssim C$. Putting \eqref{nabla_3_2derivative_varphib}, \eqref{3derivatives_etas} and \eqref{2derivatives_omegas} together, we have the renormalized equation
\begin{equation}\label{gr_1}
\nabla_3 (\triangle \varphib - \nabla^2 (\omega,\omegat) - \nabla^2 \tr\chi) = \nabla^2 \mub + \tr\chib_0 \cdot \nabla^2 \varphib + \chibh \cdot \nabla^2 \varphib + G_4
\end{equation}
with error term $\|G_4\|_{L^2_{(sc)}(\Hb_{\ub})} \lesssim C+ C^{\frac{1}{2}}\|\nabla^2 \varphib\|_{L^2_{(sc)}(\Hb_{\ub})}$. Now we can combine Gronwall's inequality, Lemma \ref{sobolev_trace_estimates} and the estimates derived so far to prove \eqref{estimates_varphi_3_hard}. Since this is standard, we omit the details.
\end{proof}

\subsubsection{Estimates on $\|\nabla^2 \mu\|_{L^2_{(sc)}(H_u)}$ and $\|\nabla^2 \mub\|_{L^2_{(sc)}(\Hb_{\ub})}$}
\begin{proposition}
If $\delta$ is sufficiently small, we have
\begin{equation}\label{twoderivative_mu}
\|\nabla^2 \mu\|_{L^2_{(sc)}(H_u)}  + \|\nabla^2 \mub\|_{L^2_{(sc)}(\Hb_{\ub})} \lesssim C,
\end{equation}
\begin{equation}\label{aaa}
 \|\nabla(\eta, \etab)\|_{Tr_{(sc)}(H_u)} + \|\nabla(\eta,\etab)\|_{Tr_{(sc)}(\Hb_{\ub})}  \lesssim C.
\end{equation}
\end{proposition}

\begin{proof} \eqref{aaa} is a direct consequence of \eqref{twoderivative_mu}. We only derive estimates on $\|\nabla^2 \mub\|_{L^2_{(sc)}(\Hb_{\ub})}$; the other one can be derived in the same manner.

We start to derive a transport equation on $\nabla_3 \mub$. We can not afford to any loss of information for highest derivative terms, since those terms are potentially dangerous. They can only be handled through a renormalization process. Commuting derivative with \eqref{NSE_Lb_etab}, we have
\begin{align*}
 \nabla_3 \,\divergence \etab & = -\frac{1}{2}\tr\chib \cdot \divergence \etab - \chibh \cdot \nabla \etab + \betab \cdot \etab -\eta \cdot \chibh \cdot \etab + \frac{1}{2} \tr\chib \cdot(\eta \cdot \etab) \\
&\quad + \frac{1}{2} (\eta + \etab)\cdot(-\chib \cdot (\etab -\eta) + \betab - \frac{1}{2}T_{b3}) +  \divergence (-\chib \cdot (\etab -\eta) + \betab + \frac{1}{2}T_{b3}).
\end{align*}
We can use \eqref{NSE_div_chibh} to eliminate $\divergence \chibh$ to derive
\begin{align*}
\nabla_3 \,\divergence \etab & = \divergence \betab +  \frac{1}{2} \divergence T_{b3} + \frac{1}{2} (\eta + 3\etab) \cdot \betab + (\eta + \etab)\cdot \nabla \tr\chib +(\nabla \eta + \nabla \etab)\cdot\chib + l.o.t.
\end{align*}
Added up with \eqref{NBE_Lb_rho} and ignoring some irrelevant coefficients, it is reduced to
\begin{align*}
 \nabla_3 \mub &= \tr\chib \cdot \rho + \chih \cdot \alphab + (\eta +\etab)\cdot \betab + (\eta + \etab)\cdot \nabla \tr\chib \\
&\qquad + \frac{1}{2} [\frac{1}{2}(D_3 R_{34}-D_4 R_{33})+\divergence T_{b3}]+(\nabla \eta + \nabla \etab)\cdot \chib + l.o.t.
\end{align*}
For the bracket, we use the key fact $\text{Div} T = 0$ of the energy-momentum tensor to derive
\begin{align*}
 \frac{1}{2}(D_3 R_{34}-D_4 R_{33})+\divergence T_{b3} &= \frac{1}{2} \nabla_3 (R_{34}) + (D^3 R_{33} + D^a R_{a3}) + l.o.t.\\
&= \frac{1}{2} \nabla_3 (R_{34}) -D^4 R_{43} + l.o.t. = \nabla_3 (R_{34}) + l.o.t.
\end{align*}
Thus, we derive
\begin{equation}\label{nab_3_mu_R}
 \nabla_3 (\mub + R_{34}) = \tr\chib \cdot \rho + \chih \cdot \alphab + (\eta +\etab)\cdot \betab  + (\eta + \etab)\cdot\nabla  \tr \chib +(\nabla \eta + \nabla \etab)\cdot \chib + l.o.t.
\end{equation}
Thanks to \eqref{NSE_Lb_chibh}, \eqref{NSE_Lb_etab} and \eqref{NSE_Lb_omega}, we have
\begin{equation}\label{nabla_3_curvature_components}
\nabla_3 (\chibh, \etab, \omega) = (-\alphab, \betab, \frac{1}{2}\rho) + l.o.t.
\end{equation}
We replace $\alphab$ and $\rho$ (but keep $\betab$) by the right-hand side of \eqref{nabla_3_curvature_components} in \eqref{nab_3_mu_R} to derive
\begin{align*}
 \nabla_3 (\mub + R_{34}) &= 2 \tr\chib \cdot \nabla_3 \omega - \chih \cdot \nabla_3 \chih  + (\eta +\etab)\cdot (\betab +\nabla \tr \chib) +(\nabla \eta + \nabla \etab)\cdot\chib + l.o.t.\\
&=  \nabla_3 (2\tr\chib \cdot \omega-\frac{1}{2}|\chibh|^2 )+ (\eta +\etab)\cdot (\betab + \nabla \tr \chib) +(\nabla \eta + \nabla \etab)\cdot\chib  + l.o.t.
\end{align*}
In view of \eqref{NSE_L_tr_chib}, we can regard $\omega \cdot \nabla_3\tr\chib$ as $l.o.t.$ Thus, we have
\begin{equation*}
 \nabla_3 \mub^1 = (\eta +\etab)\cdot \betab + (\eta + \etab) \cdot \nabla \tr \chib +(\nabla \eta + \nabla \etab)\cdot\chib + l.o.t.
\end{equation*}
where $\mub^1 = \mub + R_{34}-2\tr\chib \cdot \omega + \frac{1}{2}|\chibh|^2.$
Commuting again with $\triangle$, we derive
\begin{align*}
 \nabla_3 \triangle \mub^1 &= (\nabla^2\eta +\nabla^2 \etab)\cdot \betab + (\eta +\etab)\cdot \triangle \betab+ (\eta + \etab)\cdot \nabla^3 \tr \chib +(\nabla^2\eta +\nabla^2 \etab) \cdot \nabla \tr\chib\\
&\qquad +(\nabla \triangle \eta + \nabla \triangle \etab)\cdot\chib + (\nabla \eta +\nabla \etab) \cdot \triangle \chib + (\nabla^2\eta +\nabla^2 \etab) \cdot \nabla \chibh + H_1,
\end{align*}
with error term $\|H_1\|_{L^2_{(sc)}(\Hb_{\ub})} \lesssim C$. We now use \eqref{nabla_3_curvature_components} to replace $\triangle \betab$, thus
\begin{align*}
 \nabla_3 [\triangle \mub^1-(\eta +\etab)\cdot\triangle \etab] &= (\nabla^2\eta +\nabla^2 \etab)\cdot \betab + (\eta + \etab)\cdot \nabla^3 \tr \chib +(\nabla^2\eta +\nabla^2 \etab) \cdot (\nabla \tr\chib + \nabla \chibh)\\
&\quad +(\nabla \triangle \eta + \nabla \triangle \etab)\cdot \chib + (\nabla \eta +\nabla \etab) \cdot (\triangle \chibh + \nabla^2 \tr\chib) + H_2,
\end{align*}
with error term $\|H_2\|_{L^2_{(sc)}(\Hb_{\ub})} \lesssim C$. We still need to renormalize the terms involving $\nabla \triangle \eta$, $\nabla \triangle \etab\chib$ and $\triangle \chih$. In view of \eqref{nabla_3_2derivative_varphib}
and \eqref{laplacian_chibh},
if we define $\mub^2 = \triangle \mub^1-(\eta +\etab)\cdot\triangle \etab-\chib \cdot\triangle \phib$, we then have
\begin{align*}
 \nabla_3 \mub^2 &= (\nabla^2\eta +\nabla^2 \etab)\cdot \betab + (\eta + \etab)\cdot \nabla^3 \tr \chib +(\nabla^2\eta +\nabla^2 \etab)\cdot (\nabla \tr\chib + \nabla \chibh)\\
&\qquad +\chib \cdot \nabla^2 \varphib -\triangle \varphib \cdot \nabla_3 \chib + (\nabla \eta +\nabla \etab)\cdot(\nabla \betab + \nabla^2 \tr\chib) + H_3\\
&= T_1 +T_2 + \cdots + T_6 + H_3,
\end{align*}
with error term $\|H_3\|_{L^2_{(sc)}(\Hb_{\ub})} \lesssim C +  C \delta^{\frac{1}{2}}\|\nabla^2 \varphib\|_{L^2_{(sc)}(\Hb_{\ub})}$.

We estimate $T_i$'s one by one. For $T_1$, we have
\begin{equation*}
 \|T_1\|_{L^2_{(sc)}(\Hb_{\ub})} \lesssim \delta^{\frac{1}{2}}\|(\nabla^2 \eta, \nabla^2 \etab)\|_{L^2_{(sc)}(\Hb_{\ub})}\|\betab\|_{Tr_{(sc)}(\Hb_{\ub})}\lesssim C\delta^{\frac{1}{2}};
\end{equation*}

For $T_2$, according to Section \ref{L_2_estimates_for_nabla^3_tr_chi}, we have
\begin{equation*}
 \|T_2\|_{L^2_{(sc)}(\Hb_{\ub})} \lesssim C \delta^{\frac{1}{2}}\|\nabla^3\tr\chib\|_{L^2_{(sc)}(\Hb_{\ub})} \lesssim C \delta^{\frac{1}{2}};
\end{equation*}

For $T_3$, thanks again to Section \ref{L_2_estimates_for_nabla^3_tr_chi}, we have
\begin{equation*}
 \|T_3\|_{L^2_{(sc)}(\Hb_{\ub})} \lesssim \delta^{\frac{1}{2}}\|(\nabla^2 \eta, \nabla^2 \etab)\|_{L^2_{(sc)}(\Hb_{\ub})}(\|\nabla \tr\chib\|_{L^{\infty}_{(sc)}} + \|\nabla \chibh\|_{Tr_{(sc)}(\Hb_{\ub})})\lesssim C\delta^{\frac{1}{2}};
\end{equation*}

For $T_4$, in view of \eqref{estimates_varphi_3_hard} or \eqref{gr_1}, we have
\begin{equation*}
 \int_{0}^{u} \|T_4\|_{L^2_{(sc)}(S_{u,\ub})} \lesssim (1+C\delta^{\frac{1}{2}}) \int_{0}^{u}\|\nabla \varphib\|_{L^2_{(sc)}(S_{u,\ub})} \lesssim (1+C\delta^{\frac{1}{2}}) \int_{0}^{u}\| \nabla^2 \mub\|_{L^2_{(sc)}(S_{u,\ub})} + C;
\end{equation*}

For $T_5$, in view of \eqref{NSE_L_tr_chib}, we replace $\nabla_3 \chib$ by $\alphab$ to derive
\begin{equation*}
 \|T_5\|_{L^2_{(sc)}(\Hb_{\ub})} \lesssim \delta^{\frac{1}{2}}\|\nabla^2 \varphib\|_{L^2_{(sc)}(\Hb_{\ub})}\|\alphab\|_{Tr_{(sc)}(\Hb_{\ub})}+ C \delta^{\frac{1}{2}} \lesssim C \delta^{\frac{1}{2}}\|\nabla^2 \mub\|_{L^2_{(sc)}(\Hb_{\ub})}+ C \delta^{\frac{1}{2}};
\end{equation*}

For $T_6$, in view of the estimates Section \ref{Tr_estimates_for_nabla_eta_etab}, we have
 \begin{align*}
 \|T_6\|_{L^2_{(sc)}(\Hb_{\ub})} &\lesssim \delta^{\frac{1}{2}}\|(\nabla \eta, \nabla \etab)\|_{Tr_{(sc)}(\Hb_{\ub})}(\|\nabla\betab\|_{L^2_{(sc)}(\Hb_{\ub})}+\|\nabla^2 \tr\chib\|_{L^2_{(sc)}(\Hb_{\ub})})\\
&\lesssim C\delta^{\frac{1}{2}}\|\nabla^2 \mub\|_{L^2_{(sc)}(\Hb_{\ub})}+C.
\end{align*}
Putting things together, we have
\begin{align*}
\|\nabla^2 \mub\|_{{L^2_{(sc)}(S_{u,\ub})}} &\lesssim C + C\delta^{\frac{1}{2}}\|\nabla^2 \mub\|_{L^2_{(sc)}(\Hb_{\ub})} + (1+C\delta^{\frac{1}{2}}) \int_{0}^{u}\| \nabla^2 \mub\|_{L^2_{(sc)}(S_{u,\ub})}.
\end{align*}
Last term can be removed by Gronwall's inequality. We then perform an integration on the incoming null hypersurfaces to complete the proof.
\end{proof}

We now show the trace estimates on $\rho$ and $\sigma$.
\begin{proposition}
If $\delta$ is sufficiently small, we have
\begin{equation}
\|(\rho,\sigma)\|_{Tr_{(sc)}(H_u)} + \|(\rho,\sigma)\|_{Tr_{(sc)}(\Hb_{\ub})} \lesssim C.
\end{equation}
\end{proposition}
\begin{proof}
We prove estimates on incoming null cones $\Hb$ and the rest follows in the same way.

For $\|\sigma\|_{Tr_{(sc)}(\Hb_{\ub})}$, in view of \eqref{NSE_curl_etab}, we have
\begin{equation*}
 \|\sigma\|_{Tr_{(sc)}(\Hb_{\ub})} = \|\curl \eta\|_{Tr_{(sc)}(\Hb_{\ub})} + \|\chih \wedge \chibh\|_{Tr_{(sc)}(\Hb_{\ub})} \lesssim C.
\end{equation*}

For $\|\rho\|_{Tr_{(sc)}(\Hb_{\ub})}$, we first define $\omega' = \Omega \cdot \omega$ and rewrite \eqref{NSE_Lb_omega} as
\begin{equation}\label{nabla_3_omega_prime}
\rho = \frac{1}{\Omega}\nabla_3(\omega') + (\eta + \etab) \cdot (\eta+\etab) + ({\rho_F}^2 + {\sigma_F}^2).
\end{equation}
We only keep the principal term $\frac{1}{\Omega}\nabla_3(\omega')$ since nonlinear terms enjoy better estimates. Thus,
\begin{align*}
\|\rho\|_{Tr_{(sc)}(\Hb)} &\lesssim (\| \nabla^2_3 \omega'\|_{L^2_{(sc)}(\Hb)} + C)^\frac{1}{2}(\| \nabla^2 \omega'\|_{L^2_{(sc)}(\Hb)} + C)^\frac{1}{2}+C\lesssim C (\| \nabla^2_3 \omega'\|_{L^2_{(sc)}(\Hb)} + C)^\frac{1}{2}+C.
\end{align*}
To estimate $\| \nabla^2_3 \omega'\|_{L^2_{(sc)}(\Hb)}$, we differentiate \eqref{nabla_3_omega_prime} to derive
\begin{equation*}
\nabla^2_3 \omega' = \nabla_3 \rho + (\eta + \etab) \cdot (\nabla_3 \eta+ \nabla_3 \etab) + {\rho_F}\cdot \nabla_3 {\rho_F} +{\sigma_F} \cdot \nabla_3 {\sigma_F}.
\end{equation*}
Thanks to \eqref{NBE_Lb_rho}, we have estimates on $\nabla_3 \rho$; for those nonlinear terms, they can be easily estimated by the estimates derived so far. Hence, we complete the proof.
\end{proof}


\subsection{Estimates on Angular Momentum}\label{rotational_symmetry}
For $i=1,2,3$,  $^{(i)}\!O$ on $S_{u,0} \subset \Hb_0$ denote the generators of $\mathfrak{so}(3)$ satisfying the relation
\begin{equation}
[^{(i)}\!O, ^{(j)}\!O] = \epsilon_{ijk}\,^{(k)} \!O.
\end{equation}
In what follows, we shall suppress the index $^{(i)}$. We extend $\O$ to the entire domain $\D$ by
\begin{equation}
 \nabla_4 O_b = \chi_{b}{}^a O_a \,.
\end{equation}
The deformation tensor $\pi_{\alpha\beta}$ of $O$ is defined by $\pi_{\alpha\beta} = \L_{O} g = \frac{1}{2}(D_\alpha O_\beta + D_\beta O_\alpha)$. We also define $H_{ab} = \nabla_a O_b$ and $Z_a = \nabla_3 O_a -\chib_{ab}O_b$ and we can associate signatures to them,
\begin{equation}
 sgn(H) = \frac{1}{2}, \quad sgn(Z) = 0.
\end{equation}
We list the null components of $\pi$,
\begin{equation*}
 \pi_{33}= \pi_{44} =\pi_{4a}=0, \quad \pi_{34} = -(\eta + \etab)\cdot O,  \pi_{ab} =  \frac{1}{2}(H_{ab} + H_{ba}), \pi_{3a}= \frac{1}{2}Z_a.
\end{equation*}

We observe that both $\pi_{ab}$ and $Z$ (but not $H$!) vanish on $\Hb_0$. The goal of the current section is to derive estimates on $\pi$.

By virtue of the definitions of $H$ and $Z$, as well as \eqref{NSE_L_tr_chi}, \eqref{NSE_L_chih}, \eqref{NSE_L_tr_chib} and \eqref{NSE_L_chibh}, we derive transport equations for $H$ and $Z$
\begin{equation}\label{nabla_4_H}
 \nabla_4 H = \chi \cdot H + \nabla \chi \cdot O + \beta \cdot O + R_{b4} \cdot O + (\eta + \etab) \cdot \chi \cdot O,
\end{equation}
\begin{equation}\label{nabla_4_Z}
 \nabla_4 Z = (2\omega+\chi)\cdot Z +(\eta + \etab)\cdot H + \sigma \cdot O +\nabla(\eta + \etab) \cdot O + (\eta + \etab)\cdot O \cdot (\eta + \etab).
\end{equation}

We observe that $\chib$ and $\omegab$ do not appear in \eqref{nabla_4_H} and \eqref{nabla_4_Z}.

\begin{proposition}
 If $\delta$ is sufficiently small, we have
\begin{equation*}
 \|H\|_{L^\infty_{(sc)}} +\|Z\|_{L^\infty_{(sc)}} \lesssim C, \quad  \|\pi_{ab}\|_{L^2_{(sc)}(S_{u,\ub})} +\|Z\|_{L^2_{(sc)}(S_{u,\ub})} \lesssim C.
\end{equation*}
\end{proposition}
\begin{proof}
We integrate \eqref{nabla_4_H} to derive
\begin{align*}
 \|H\|_{L^\infty_{(sc)}} &\lesssim \|H(u,0)\|_{L^\infty_{(sc)}} + \|(\nabla \chih, \beta)\|_{Tr_{(sc)}(H_u)} + C + C \delta^{\frac{1}{2}}\|H\|_{L^\infty_{(sc)}}.
\end{align*}
Since $\|H(u,0)\|_{L^\infty} \lesssim 1$ and $sc(H)=0$, we have $\|H(u,0)\|_{L^\infty_{(sc)}} \lesssim 1$. This yields the estimate on $H$. The $L^\infty$ estimate for $Z$ is similar in view of the key fact that $Z$ vanishes on $\Hb_0$. The $L^2_{(sc)}$ estimates follows again directly from \eqref{nabla_4_H} and \eqref{nabla_4_Z} and the triviality of the initial data of $\pi_{ab}$ and $Z$.
\end{proof}
We remark that $L^2_{(sc)}$ estimate for $H$ is anomalous because $\|H(u,0)\|_{L^2_{(sc)}} \sim \delta^{-\frac{1}{2}}$.

We turn to one derivative of $H$ and $Z$. Commuting $\nabla$ with \eqref{nabla_4_H} and \eqref{nabla_4_Z}, we derive
\begin{equation}\label{nabla_4_nabla_H}
 \nabla_4 \nabla H = \chi \cdot \nabla H + \nabla^2 \chi \cdot O + \nabla \beta \cdot O + E_1
\end{equation}
\begin{equation}\label{nabla_4_nabla_Z}
 \nabla_4 \nabla Z = (\omega + \chi)\cdot \nabla Z + \nabla^2(\eta + \etab) \cdot O + \nabla \sigma \cdot O + E_2
\end{equation}
We can easily show that the error terms can be bounded as follows,
\begin{equation*}
\|(E_1,E_2)\|_{{L^2_{(sc)}(S_{u,\ub})}} + \|(E_1,E_2)\|_{{L^4_{(sc)}(S_{u,\ub})}} \lesssim C.
\end{equation*}
\begin{proposition}
If $\delta$ is sufficiently small, we have
\begin{equation}\label{L_2_nabla_H_Z}
 \|(\nabla H, \nabla Z)\|_{{L^2_{(sc)}(S_{u,\ub})}} + \|(\nabla_4 \nabla H, \nabla_4 \nabla Z)\|_{L^2_{(sc)}(H_u)} \lesssim C,
\end{equation}
\begin{equation}\label{L_4_nabla_H_Z}
 \|(\nabla H, \nabla Z)\|_{{L^4_{(sc)}(S_{u,\ub})}} \lesssim C.
\end{equation}
\end{proposition}
\begin{proof}
Besides $E_1$ and $E_2$, there are six terms appearing in \eqref{nabla_4_nabla_H} and \eqref{nabla_4_nabla_Z}. For $\chi \cdot \nabla H$ and $(\omega + \chi)\cdot \nabla Z$, they can be absorbed eventually by Gronwall's inequality. We now indicate how to renormalize the rest four terms. The details of the proof are routine.

For $\nabla^2 \chi \cdot O$, we split it as $\nabla^2 \chi \cdot O = \nabla^2 \chih \cdot O + \nabla^2 \tr\chi \cdot O$. Thanks to $L^2_{(sc)}(H_u)$ estimates on $\nabla^3 \tr\chi$, we have
\begin{align*}
 \delta^{-1}\int_0^{\ub} \|\nabla^2 \tr\chi \cdot O\|_{L^4_{(sc)}(u,\ub')}  
&\lesssim \|\nabla^3 \tr\chi\|^{\frac{1}{2}}_{L^2_{(sc)}(H_u)}\|\nabla^2 \tr\chi\|^{\frac{1}{2}}_{L^2_{(sc)}(H_u)} + \delta^{\frac{1}{4}}\|\nabla^2 \tr\chi\|_{L^2_{(sc)}(H_u)}\lesssim C.
\end{align*}
To control $\nabla^2 \chih \cdot O$, we use \eqref{laplacian_chibh} to derive $\nabla^2 \chih = \nabla^2 \triangle^{-1} \triangle \chih=\nabla^2 \triangle^{-1} \nabla \beta + E_3$ where $E_3$ is much easier to control. Up to easier error terms, we use \eqref{NBE_L_rho} and \eqref{NBE_L_sigma} to replace $\nabla \beta$ by $\nabla_4 (\rho, \sigma)$. In this way, we can easily renormalize the system by moving $\nabla_4 \nabla^2 \triangle^{-1}(\rho, \sigma)$ to the left hand side of the equation. Hence, we can derive estimates easily. We also observe that $\nabla \beta \cdot O$ can also be treated in this way.

For  $\nabla \sigma \cdot O$, we renormalize it via transport equation $\nabla_4 \omegat = \frac{1}{2}\sigma$.

For $\nabla^2(\eta + \etab) \cdot O$, recall we defined $ \nabla_4 \varphi = (\nabla \eta, \nabla \etab)$ in Section \ref{Tr_estimates_for_nabla_eta_etab}, thus
\begin{align*}
 \nabla^2(\eta + \etab) \cdot O &= \nabla_4 (\nabla \varphi \cdot O)-\nabla \varphi \cdot \nabla_4 O + \chi \cdot \nabla \varphi + (\chi \cdot \etab+\beta+\Upsilon \cdot\Upsilon)\cdot\varphi + (\eta + \etab)\cdot\nabla(\eta, \etab).
\end{align*}
We then eliminate $\nabla_4 (\nabla \varphi \cdot O)$ via the standard renormalization. The other terms are easy to control thanks to the estimates on $\varphi$ in Section \ref{Tr_estimates_for_nabla_eta_etab}. This completes the proof.
\end{proof}

We commute $\nabla_3$ with \eqref{nabla_4_H} and \eqref{nabla_4_Z}  to derive
\begin{align*}
 \nabla_4 \nabla_3 H &= (\chi + \omega) \cdot \nabla_3 H + (\nabla_3 \tr\chi + \eta + \etab  + \sigma)\cdot H +\nabla \chi \cdot Z + (\nabla_3\nabla \chi + \chib \cdot \nabla \chi) \cdot O\\
& \quad + \omegab \cdot [\chi\cdot H + (\nabla \chi +(\eta + \etab)\cdot \chi+ \beta + \Upsilon\cdot \Upsilon)\cdot O] +(\beta + \Upsilon\cdot \Upsilon)\cdot (Z + \chib \cdot O)\\
&\quad+(\nabla_3 \beta + \nabla_3 R_{b4})\cdot O + \nabla_3 \chih \cdot Z,\\
 \nabla_4 \nabla_3 Z &= (\omega + \chi) \cdot Z + \omegab \cdot \nabla_4 Z + (\eta + \etab)\cdot \nabla_3 H + \nabla_3 \sigma \cdot O + \nabla_3 \chih \cdot Z\\
&\quad + \nabla_3 (\eta + \etab) \cdot H + [\sigma +(\eta + \etab)\cdot(1 + \eta + \etab) + \nabla(\eta + \etab) + \nabla_3 (\omega + \tr\chi)]\cdot Z\\
&\quad + [\nabla_3 \nabla (\eta + \etab) + \chib \cdot \nabla (\eta + \etab) + (\eta + \etab)\cdot (\nabla_3 (\eta + \etab)+ (\eta + \etab)\cdot \chib)+\chib \cdot\sigma ]\cdot O.
\end{align*}
\begin{proposition}
If $\delta$ is sufficiently small, we have
\begin{equation}\label{L_2_nabla_nabla_3_HZ}
 \|(\nabla_3 H, \nabla_3 Z)\|_{{L^2_{(sc)}(S_{u,\ub})}} + \|(\nabla_4 \nabla_3 H, \nabla_4 \nabla_3 Z)\|_{L^2_{(sc)}(H_u)} \lesssim C.
\end{equation}
\end{proposition}
\begin{proof} We only derive estimates on $\|\nabla_3 H\|_{{L^2_{(sc)}(S_{u,\ub})}}$ and $\|\nabla_3 Z\|_{{L^2_{(sc)}(S_{u,\ub})}}$ can be bounded exactly in the same manner.

We observe that the possible dangerous terms are $\nabla_3 \beta \cdot O$, $\nabla_3 R_{b4}\cdot O$ and $\nabla_3 \chih \cdot Z$. The rest are much easier to control. For $\nabla_3 \chih \cdot Z$, $Z$ is \emph{not} anomalous, thus H\"older's inequality gains an extra $\delta^{\frac{1}{2}}$ to compensate the anomaly of $\nabla_3 \chih$; for $\nabla_3 \beta \cdot O$, we use \eqref{NBE_Lb_beta} to convert the $\nabla_3 \beta$ to angular derivative of curvature components; for $\nabla_3 R_{b4}\cdot O$, we have
\begin{align*}
 \nabla_3 R_{b4} &= \nabla_3 (\rho_F + \sigma_F) \cdot \alpha_F + l.o.t. = \Upsilon \cdot (\nabla_3 \rho + \nabla_3 \sigma ) +l.o.t.
\end{align*}
which is still easy to estimate. This completes the proof.
\end{proof}
\begin{remark}
We have a slightly refined estimates on $\nabla_3 Z$,
\begin{equation}\label{bbb}
  \|\nabla_3 Z\|_{L^2_{(sc)}(S_{u,\ub})L^{\infty}[0,\ub]} = \| \sup_{\ub' \in [0,\ub]}|\nabla_3 Z (u, \ub')|\|_{L^{2}_{(sc)}(u,\ub)}\lesssim C.
\end{equation}
\end{remark}
To prove this estimate, we observe that
\begin{equation*}
 \nabla_4 [\sup_{\ub' \in [0,\ub]}|\nabla_3 Z (u, \ub')|] \leq |\nabla_4 \nabla_3 Z (u, \ub)|.
\end{equation*}
Thus, we can still use the transport equation for $\nabla_3 Z$ to derive \eqref{bbb}.
Now we state the main proposition of this section.
\begin{proposition} If $\delta$ is sufficiently small, we have
\begin{equation*}
 \|D \pi\|_{{L^2_{(sc)}(S_{u,\ub})}} \lesssim C.
\end{equation*}
\end{proposition}
\begin{proof} We only sketch the idea. The details can be easily carried out by the estimates derived so far. If one computes the null component of $D \pi$ one by one, we see immediately that $H$ comes either as $\psi \cdot H$ or as $\tr\chi_0 \cdot (H + \,^t\! H)$. We pay attention to $H$ since it does not enjoy non-anomalous $L^2_{(sc)}$ estimates. Nevertheless, we still can proceed as follows,
\begin{equation*}
  \| \psi \cdot H\|_{{L^2_{(sc)}(S_{u,\ub})}} \lesssim \delta^{\frac{1}{2}} \|\psi\|_{L^\infty_{(sc)}(u,\ub)}  \|H\|_{{L^2_{(sc)}(S_{u,\ub})}} \lesssim C.
\end{equation*}
For $\tr\chi_0 \cdot (H_{ab} + H_{ab}) = \pi_{ab}$, since $\pi_{ab}$ has trivial data on the initial surfaces, this allows us to derive non-anomalous estimates on $\pi_{ab}$.
\end{proof}

We summarize the estimates in this section in the following,
\begin{theoremB}\label{theorem B}If $\delta$ is sufficiently small, we have
\begin{equation*}
  \| \pi\|_{{L^4_{(sc)}(S_{u,\ub})}}+\| \pi\|_{L^{\infty}_{(sc)}(u,\ub)} +  \|D \pi\|_{{L^2_{(sc)}(S_{u,\ub})}} \lesssim C,
\end{equation*}
and for all the components of $D\pi$ except $D_3 \pi_{3a}$  we have
\begin{equation*}
\|D \pi\|_{L^{4}_{(sc)}(u,\ub)} \lesssim C,
\end{equation*}
Moreover,
\begin{equation*}
\| D_3 \pi_{3a}- \frac{1}{2}\nabla_3 Z\|_{L^{4}_{(sc)}(u,\ub)} +\|\nabla_3 Z\|_{L^2_{(sc)}(S_{u,\ub})L^{\infty}[0,\ub]} \lesssim C.
\end{equation*}
\end{theoremB}


\section{Theorem C - Energy Estimates on Curvature and Maxwell Field}


\subsection{Energy Estimates on Maxwell Field}
Before we start doing estimates, we first list the components of the deformation tensors $^{(L)}\!\pi$ and $^{(\Lb)}\!\pi$,
\begin{center}
  \begin{tabular}{ | c | c | c | c | c | c | c|}
    \hline
    & $\pi_{44}$ & $\pi_{34}$ & $\pi_{33}$ & $\pi_{4a}$ & $\pi_{3a}$ & $\pi_{ab}$\\
    \hline
    $L$ & $0$ & $0$ & $-8\Omega^{-1} \omegab$ & $0$ & $2\Omega^{-1}\eta_a$ & $\Omega^{-1} \chi_{ab}$\\
    $\Lb$ & $-8\Omega^{-1} \omega$ & $0$ & $0$ & $2 \Omega^{-1}\etab_a$ & $0$ & $\Omega^{-1} \chib_{ab}$\\
    \hline
  \end{tabular}
\end{center}
We use $L$ and $\Lb$ as multipliers rather than $e_4$ and $e_3$. The main reason is that we can avoid the appearance of $\omega$ and $\omegab$ in $\pi_{34}$.

We first take $X = L$ in \eqref{stokes_energy_momentum} to derive
\begin{align*}
 \int_{H_u^{(0,\ub)}} |{\alpha_F}|^2+ \int_{\Hb_{\ub}^{(0,u)}}|{\rho_F}|^2 + |{\sigma_F}|^2 &=\int_{H_0^{(0,\ub)}} |\alpha|^2 + \doubleint_{\D(u,\ub)} \, ^{(L)}\!\pih \cdot T[F].
\end{align*}
In the scale invariant norms, we have
\begin{align*}
\|{\alpha_F}\|^2_{L^2_{(sc)}(H_u^{(0,\ub)})}+ \|({\rho_F}, {\sigma_F})\|^2_{L^2_{(sc)}(\Hb_{\ub}^{(0,u)})} &=\|{\alpha_F}\|^2_{L^2_{(sc)}(H_0^{(0,\ub)})} + \delta^{-1} \doubleint_{\D(u,\ub)} {}^{(L)}\pih \cdot T[F].
\end{align*}
The integrand of the last term can be written schematically as $\psi^{s_1} \cdot \Upsilon^{s_2} \cdot \Upsilon^{s_3}$ with $s_1 + s_2 + s_3 = 2$, where $s_i$ indicates the signature for the corresponding component. Since
\begin{align*}
|\doubleint_{\D(u,\ub)}\psi^{s_1} \cdot \psi^{s_2} \cdot \psi^{s_3}| \lesssim \delta^{-s_1-s_2 -s_3 + \frac{7}{2}} \|\psi^{s_1}\|_{L^{\infty}_{(sc)}}(\|\psi^{s_2}\|^2_{L^2_{(sc)}(H_u^{(0,\ub)})} +\|\psi^{s_3}\|^2_{L^2_{(sc)}(\Hb_{\ub}^{(0,u)})}).
\end{align*}
Together with the $L^\infty_{(sc)}$ estimates of connection coefficients, we have
\begin{equation*}
| \delta^{-1} \doubleint_{\D(u,\ub)} \pih \cdot T[F] | \lesssim C \delta^{\frac{1}{2}}\int_{0}^u \|{\alpha_F}\|^2_{L^2_{(sc)}(H_{u'}^{(0,\ub)})} du' + C \delta^{\frac{1}{2}} (\Fzero^2(u,\ub)+\Fzerob^2(u,\ub)).
\end{equation*}
After a standard use of Gronwall's inequality, we have
\begin{equation*}
\|{\alpha_F}\|^2_{L^2_{(sc)}(H_u^{(0,\ub)})}+ \|({\rho_F}, {\sigma_F})\|^2_{L^2_{(sc)}(\Hb_{\ub}^{(0,u)})} \lesssim \|{\alpha_F}\|^2_{L^2_{(sc)}(H_0)} + C \delta^{\frac{1}{2}} (\Fzero^2(u,\ub)+\Fzerob^2(u,\ub)).
\end{equation*}
which can be written in $\Fzero$ and $\Fzerob$ norms as
\begin{equation}\label{anomalous_Fzero}
\Fzero[\alpha_F] + \Fzerob[{\rho_F}, {\sigma_F}] \lesssim \Izero +C \delta^{\frac{3}{4}} (\Fzero+\Fzerob).
\end{equation}

We take $X = \Lb$ in \eqref{stokes_energy_momentum} to derive
\begin{align*}
 \int_{H_u^{(0,\ub)}} |{\rho_F}|^2 + |{\sigma_F}|^2+ \int_{\Hb_{\ub}^{(0,u)}} |{\alphab_F}|^2 &=\int_{H_0^{(0,\ub)}} |{\rho_F}|^2 + |{\sigma_F}|^2  + \doubleint_{\D(u,\ub)} \, ^{(\Lb)}\!\pih \cdot T[F].
\end{align*}
In the scale invariant norms, we have
\begin{align*}
\|({\rho_F}, {\sigma_F})\|^2_{L^2_{(sc)}(H_u^{(0,\ub)})}+ \|\alphab\|^2_{L^2_{(sc)}(\Hb_{\ub}^{(0,u)})} &\lesssim \Ozero\Izero + \delta^{-2} \doubleint_{\D(u,\ub)} \phi^{s_1} \cdot \Upsilon^{s_2} \cdot \Upsilon^{s_3}.
\end{align*}
with $s_1 + s_2 + s_3 = 1$. To control the bulk integral $E$, we make the following a couple of observations to avoid double anomalies.
\begin{remark}\label{observation_1}
For the integrand $\phi^{s_1} \cdot \Upsilon^{s_2} \cdot \Upsilon^{s_3}$, at least one of the $\Upsilon^{s_i}$'s is \emph{not} $\alpha_F$.
\end{remark}
\begin{remark}\label{observation_2}
For the integrand $\phi^{s_1} \cdot \Upsilon^{s_2} \cdot \Upsilon^{s_3}$, if $\phi^{s_1} = \tr\chib$, then \emph{none} of $\Upsilon^{s_i}$'s is $\alpha_F$.
\end{remark}
The first one is due to the signature consideration. For the second, we observe that only way for $\chib$ to appear is through the terms $\chib^{ab} \cdot T_{ab}$. While this term is equal to $\tr\chib \cdot (|\rho_F|^2 + |\sigma_F|^2)$. These two types of remarks will be repeatedly used in rest of the paper.

Thanks to these two remarks, we split $E$ into two terms
\begin{equation*}
 E= \delta^{-2} \doubleint_{\D(u,\ub)} \tr\chib_0 \cdot (|\rho_F|^2 + |\sigma_F|^2) + \delta^{-2} \doubleint_{\D(u,\ub)} \psi \cdot \Upsilon \cdot {\Upsilon}_g = E_1 + E_2.
\end{equation*}

For $E_1$, we have $|E_1| \lesssim \int_0^u \Fzero^2(u',\ub)$.

For $E_2$, since the appearance of $\alpha_F$ may cause a loss of $\delta^{\frac{1}{2}}$, we use $L^{4}_{(sc)}$ estimates on $\Upsilon$ to save an extra $\delta^{\frac{1}{4}}$. This is still not good enough because $\alpha_F$ can still be coupled to the anomalous term $\chibh$ (notice that $\chih$ does not appear). Thus, we further split $E_2$ into two parts
\begin{equation*}
 E_2= \delta^{-2} \doubleint_{\D(u,\ub)} \psi_g \cdot \Upsilon \cdot {\Upsilon}_g + \delta^{-2} \doubleint_{\D(u,\ub)} \chibh \cdot \alpha_F \cdot {\Upsilon}_g= E_{21} + E_{22}
\end{equation*}
By signature considerations, more accurately, we have $E_{22}=\delta^{-2} \doubleint_{\D(u,\ub)} \chibh \cdot \alpha_F \cdot \alphab_F$.

For $E_{21}$, in view of the fact that $\|\psi\|_{L^4_{(sc)}(H_{u}^{(0,\ub)})} \lesssim C$, we have
\begin{align*}
 |E_{21}| &\lesssim \delta^{\frac{1}{2}} \int_{0}^{u} \|\psi_g\|_{L^4_{(sc)}(H_{u'}^{(0,\ub)})} \|\Upsilon\|_{L^4_{(sc)}(H_{u'}^{(0,\ub)})}\|{\Upsilon}_g\|_{L^2_{(sc)}(H_{u'}^{(0,\ub)})} du' \lesssim C \delta^{\frac{1}{4}}.
\end{align*}

For $E_{22}$, we have
\begin{align*}
 |E_{22}| &\lesssim \delta^{-\frac{1}{2}} \int_{0}^{u}\int_{0}^{\ub} \|\chibh\|_{L^4_{(sc)}(u',\ub')} \|\alpha_F\|_{L^4_{(sc)}(u',\ub')}\|\alphab_F\|_{L^2_{(sc)}(u',\ub')}\\
& \lesssim \delta^{-\frac{3}{4}} \Ozerofour[\chibh] \int_{0}^{u}\int_{0}^{\ub}  \|\alpha_F\|_{L^4_{(sc)}(u',\ub')}\|\alphab_F\|_{L^2_{(sc)}(u',\ub')}\\
&\lesssim \delta^{-\frac{3}{4}} \Ozerofour[\chibh] \doubleint_{\D(u,\ub)} (\|\alpha_F\|^{\frac{1}{2}}_{{L^2_{(sc)}(S_{u,\ub})}}\|\nabla\alpha_F\|^{\frac{1}{2}}_{{L^2_{(sc)}(S_{u,\ub})}} + \delta^{\frac{1}{4}}\|\alpha_F\|_{{L^2_{(sc)}(S_{u,\ub})}} )\|\alphab_F\|_{{L^2_{(sc)}(S_{u,\ub})}}
\end{align*}
Thus,
\begin{align*}
 |E_{22}| &\lesssim \delta^{-\frac{3}{4}} \Ozerofour[\chibh] (\doubleint_{\D(u,\ub)} \|\alpha_F\|_{{L^2_{(sc)}(S_{u,\ub})}}\|\nabla\alpha_F\|_{{L^2_{(sc)}(S_{u,\ub})}} )^{\frac{1}{2}}(\doubleint_{\D(u,\ub)} \|\alphab_F\|^2_{{L^2_{(sc)}(S_{u,\ub})}} )^{\frac{1}{2}}\\
&\quad + \delta^{-\frac{1}{2}} \Ozerofour[\chibh] (\doubleint_{\D(u,\ub)} \|\alpha_F\|^2_{{L^2_{(sc)}(S_{u,\ub})}})^{\frac{1}{2}}(\doubleint_{\D(u,\ub)} \|\alphab_F\|^2_{{L^2_{(sc)}(S_{u,\ub})}} )^{\frac{1}{2}}\\
&\lesssim \delta^{\frac{1}{4}} \Ozerofour[\chibh] (\int_0^{u}\|\alpha_F\|_{L^2_{(sc)}(H_{u'}^{(0,\ub)})}\|\nabla\alpha_F\|_{L^2_{(sc)}(H_{u'}^{(0,\ub)})} d u' )^{\frac{1}{2}}(\int_0^{\ub}\delta^{-1} \Fzerob(u,\ub')^2 d\ub' )^{\frac{1}{2}}\\
&\quad + \delta^{\frac{1}{2}}\Ozerofour[\chibh] (\int_0^{u}\|\alpha_F\|^2_{L^2_{(sc)}(H_{u'}^{(0,\ub)})} d u' )^{\frac{1}{2}}(\int_0^{\ub}\delta^{-1} \Fzerob(u,\ub')^2 d\ub' )^{\frac{1}{2}}
\end{align*}
We retrieve the loss of $\delta$'s back to the initial data. More precisely, according to \eqref{anomalous_Fzero} and Proposition \ref{OZERO_FOUR_ESTIMATES}, we have
\begin{equation*}
\|\alpha_F\|_{L^2_{(sc)}(H_u^{(0,\ub)})} \lesssim \Izero \delta^{-\frac{1}{2}} + C \delta^{\frac{1}{4}}(\Fzero+\Fzerob),\,\,\Ozerofour[\chibh] \lesssim \Izero + C \delta^{\frac{1}{4}}.
\end{equation*}
Thus, we have
\begin{align*}
 |E_{22}| &\lesssim (\Ozero\Izero + C \delta^{\frac{1}{4}})\times[\delta^{\frac{1}{4}}  (\int_0^{u}(\Ozero\Izero \delta^{-\frac{1}{2}} + C \delta^{\frac{1}{4}}(\Fzero+\Fzerob))\Fone(u',\ub) d u' )^{\frac{1}{2}}\\
&\quad +\delta^{\frac{1}{2}}(\int_0^{u}(\Izero \delta^{-\frac{1}{2}} + C \delta^{\frac{1}{4}}(\Fzero+\Fzerob))^2 d u' )^{\frac{1}{2}}]\times(\int_0^{\ub}\delta^{-1} \Fzerob(u,\ub')^2 d\ub' )^{\frac{1}{2}}\\
&\lesssim \Izero + C \delta^{\frac{1}{4}} + C(\Izero) \Fone.
\end{align*}
Putting the estimates on $E_1$ and $E_2$, we have
\begin{align*}
&\|({\rho_F}, {\sigma_F})\|^2_{L^2_{(sc)}(H_u^{(0,\ub)})}+ \|\alphab_F\|^2_{L^2_{(sc)}(\Hb_{\ub}^{(0,u)})}\\
 &\lesssim \Izero+C\delta^{\frac{1}{4}} + \int_0 ^{\ub} \delta^{-1} \Fzerob^2(u,\ub') + \int_0^u \Fzero^2(u',\ub)+C(\Izero) \Fone.
\end{align*}
Together with \eqref{anomalous_Fzero}, we have
\begin{align*}
 \Fzero^2(u,\ub)+ \Fzerob^2(u,\ub)&\lesssim \Izero + C \delta^{\frac{1}{4}}+\int_0 ^{\ub} \delta^{-1} \Fzerob^2(u,\ub') + \int_0^u \Fzero^2(u',\ub) + C(\Izero)\Fone.
\end{align*}
We then use the Gronwall's inequality to remove the integral terms. Thus, we derive the energy estimates on Maxwell field as follows,
\begin{equation}\label{Fzero_Fzerob_Estimates}
\Fzero + \Fzerob \lesssim \Izero + C\delta^{\frac{1}{8}} + C(\Izero)\Fone^{\frac{1}{2}}.
\end{equation}


\subsection{Energy Estimates on One derivative of Maxwell Field}

\subsubsection{Preliminaries}
\begin{lemma}\label{comparison_maxwell_1}
If $\delta$ is sufficiently small, for $N = e_3$ or $e_4$ and for all components $\Upsilon$, we have
\begin{equation}
 \|\Upsilon(D_N F) - \nabla_N \Upsilon\|_{L^2_{(sc)}({H_u^{(0,\ub)}})} +  \|\Upsilon(D_N F) - \nabla_N \Upsilon\|_{L^2_{(sc)}({\Hb_{\ub}^{(0,u)}})} \lesssim C\delta^{\frac{1}{2}}.
\end{equation}
\end{lemma}
\begin{proof}
We observe that $\Upsilon(D_X F) - \nabla_X \Upsilon$ is a linear combination of the terms of form $\psi_g \cdot \Upsilon$. We use H\"{o}lder's inequality to bound $\Upsilon$ in $L^{\infty}_{(sc)}$ and $\psi_g$ in $L^{2}_{(sc)}$.
\end{proof}
\begin{remark} Similar argument also gives
\begin{equation}\label{comparison_alpha_D_a_F}
 \|\nabla_a \alpha_F-\alpha(D_a F)\|_{L^2_{(sc)}({H_u^{(0,\ub)}})} \lesssim C\delta^{\frac{1}{2}}.
\end{equation}
\end{remark}
\begin{lemma}\label{comparison_maxwell_Lie}
If $\delta$ is sufficiently small, for $\Upsilon_g \in \{\rho_F, \sigma_F, \alphab_F\}$, we have
\begin{equation*}
 \|\Upsilon_g(\Lh_O F) - \nabla_O \Upsilon_g\|_{L^2_{(sc)}({H_u^{(0,\ub)}})} +  \|\Upsilon(\Lh_O F) - \nabla_O \Upsilon_g\|_{L^2_{(sc)}({\Hb_{ub}^{(0,u)}})} \lesssim C\delta^{\frac{1}{2}}.
\end{equation*}
\end{lemma}
\begin{proof}
We prove the theorem by direct computations. We only derive estimates for $\alphab_F$, the others can be obtained in the same manner. We compute
\begin{equation*}
\alphab(\Lh_O F) = \nabla_O \alphab_F + \frac{1}{2}(Z\cdot \sigma_F - Z \cdot \rho_F + (H-{}^t \!H) + 2 (\eta+\etab)\cdot O + \tr H ) \cdot \alphab_F.
\end{equation*}
There are two possible anomalies: $\alpha_F$ (this one appears in $\rho(\Lh_O F)$ and $\sigma(\Lh_O F)$ but not here) and $H$. We shall use $L^\infty_{(sc)}$ estimates on them to avoid the loss of $\delta$. Since these two anomalies are never coupled together, so we can easily prove the lemma.
\end{proof}
\begin{remark}\label{difference_Lie_O_alpha}
We can compute $\alpha(\Lh_O F)$ directly to see that $H$ and $\alpha_F$ are coupled together to form a quadratic term, thus we have
\begin{equation*}
 \|\alpha(\Lh_O F) - \nabla_O \alpha_F\|_{L^2_{(sc)}({H_u^{(0,\ub)}})} \lesssim C \delta^{\frac{1}{2}} + \delta^{\frac{1}{2}} \|H\|_{L^4_{(sc)}({H_u^{(0,\ub)}})} \|\alpha_F\|_{L^4_{(sc)}({H_u^{(0,\ub)}})}.
\end{equation*}
\end{remark}
Thanks to \eqref{Fzero_Fzerob_Estimates}, next lemma is immediate from \eqref{NM_L_alphab}-\eqref{NM_Lb_sigma}.
\begin{lemma}\label{comparison_maxwell_null_derivatives}
If $\delta$ is sufficiently small, we have
\begin{equation}\label{comparison_nabla_4_Psi_F}
 \|(\nabla_4 \rho_F, \nabla_4 \sigma_F, \nabla_4 \alphab_F)\|_{L^2_{(sc)}({H_u^{(0,\ub)}})} \lesssim \Fone(u,\ub) + C\delta^{\frac{1}{4}},
\end{equation}
\begin{equation}\label{comparison_nabla_3_Psi_F}
 \|(\nabla_3 \rho_F, \nabla_3 \sigma_F)\|_{L^2_{(sc)}({\Hb_{\ub}^{(0,u)}})} \lesssim \Foneb(u,\ub) + C\delta^{\frac{1}{4}}.
\end{equation}
\end{lemma}
\begin{remark}
The absence of $\nabla_3 \alpha_F$ in above lemma is due to the appearance of $\tr\chib \cdot \alpha_F$ in \eqref{NM_Lb_alpha}. In fact, $\nabla_3 \alpha_F$ is anomalous and its anomaly can be retrieved to the initial data.
\end{remark}

\subsubsection{Estimates on Anomalies}

For a given vector field $X$, the Maxwell equations for $D_X F_{\alpha\beta}$ are given as follows,
\begin{equation}\label{D_X of F}
 J_\nu = D^{\mu}D_X F_{\mu\nu} = D^{\gamma}X^{\delta} D_{\gamma}F_{\delta \nu}- R_{X}{}^{\gamma}F_{\gamma\nu}+R_{X}{}^{\gamma}{}_{\nu}{}^{\delta}F_{\gamma\delta},
\end{equation}
\begin{equation}\label{D_X of Fstar}
  ^*\!J_\nu = D^{\mu}D_X \,^*\!F_{\mu\nu} = D^{\gamma}X^{\delta} D_{\gamma}\,^*\!F_{\delta \nu}- R_{X}{}^{\gamma}\,^*\!F_{\gamma\nu}+R_{X}{}^{\gamma}{}_{\nu}{}^{\delta}\,^*\!F_{\gamma\delta}.
\end{equation}

We take $X=L$ in \eqref{D_X of F} and $X = L$ in \eqref{stokes_energy_momentum} to derive
\begin{align*}
 \|\alpha(D_4 F)\|^2_{L^2_{(sc)}(H_u^{(0,\ub)})} \lesssim \|\alpha(D_4 F)\|^2_{L^2_{(sc)}(H_0^{(0,\ub)})} + \delta  \doubleint_{\D(u,\ub)} \,^{(L)}\!\pih \cdot T[D_4 F] + \Divergence T[D_4 F](L).
\end{align*}
In view of Lemma \ref{comparison_maxwell_1}, we rewrite the above as
\begin{align*}
\|\nabla_4 \alpha_F\|^2_{L^2_{(sc)}(H_u^{(0,\ub)})} &\lesssim C\delta^{\frac{1}{2}} + \|\nabla_4 \alpha_F\|^2_{L^2_{(sc)}(H_0^{(0,\ub)})} + \delta  \doubleint_{\D(u,\ub)} \,^{(L)}\!\pih \cdot T[D_4 F] + \Divergence T[D_4 F](L)\\
&= C\delta^{\frac{1}{2}} + \|\nabla_4 \alpha_F\|^2_{L^2_{(sc)}(H_0^{(0,\ub)})} + T_1 + T_2.
\end{align*}

For $T_1$, its integrand $\,^{(L)}\!\pih \cdot T[D_4 F]$ can be written as $\psi^{s_1} \cdot \Psi(D_4 F)^{s_2} \cdot \Psi(D_4 F)^{s_3}$ with $s_1 + s_2 + s_3 = 4$. Due to signature considerations, at most one of the $\Psi(D_4 F)^{s_i}$'s can be $\alpha(D_4 F)$. Since $\chib$ does not appear in ${}^{(L)}\pih$, we can bound $\psi^{s_1}$ in $L^{\infty}_{(sc)}$. In view of \eqref{comparison_nabla_4_Psi_F}, we have
\begin{align*}
\delta \doubleint_{\D(u,\ub)} \,^{(L)}\!\pih \cdot T[D_X F] &\lesssim \delta^{\frac{1}{2}}\|\psi^{s_1}\|_{L^{\infty}_{(sc)}}\int_{0}^u \| \Psi(D_4 F)^{s_2}\|_{L^2_{(sc)}(H_{u'})}\| \Psi(D_4 F)^{s_3}\|_{L^2_{(sc)}(H_{u'})}d u'\\
&\lesssim C\delta^{\frac{1}{2}} \Fone\int_{0}^u \| \nabla_4 \alpha_F\|_{L^2_{(sc)}(H_{u'})} d u' + C\delta^{\frac{1}{2}}\Fone^2 + C\delta^{\frac{1}{4}}\\
&\lesssim C\delta^{\frac{1}{2}} \int_{0}^u \| \nabla_4 \alpha_F\|_{L^2_{(sc)}(H_{u'})} d u' +C\delta^{\frac{1}{4}}.
\end{align*}

For $T_2$, we first compute it integrand $\Divergence T[D_4 F](L)$. In view of \eqref{divergence_of_T[F]}, we can ignore the part involving the Hodge dual of $D_4 F$ since it can be treated exactly in the same manner. Thus, we rewrite $\Divergence T[D_4 F](L)$ as
\begin{align*}
 \Divergence T[D_4 F](e_4) = -\frac{1}{2}D_4 F_{34} \cdot J_4 + D_4 F_{4a} \cdot J_a = T_{21}+T_{22}.
\end{align*}

For $T_{21}$, we observe that $-\frac{1}{2}D_4 F_{34}$ is not anomalous thanks to Lemma \ref{comparison_maxwell_1} and
\begin{align*}
 J_4 &= D^{\gamma}L^{\delta} \cdot D_{\gamma}F_{\delta 4} - R_{4}{}^{\gamma}F_{\gamma 4} + R_{4}{}^{\gamma}{}_{4}{}^{\delta}F_{\gamma\delta} =  D^{\gamma}L^{\delta} \cdot D_{\gamma}F_{\delta 4} - R_{4}{}^{\gamma}F_{\gamma 4}\\
&=D^{a}L^b \cdot D_a F_{b4} -\Omega^{-1}\eta_a D_4 F_{ab}- R_{4}{}^{\gamma}F_{\gamma 4} .
\end{align*}
In view of \eqref{comparison_alpha_D_a_F}, we can replace $D_a F_{b4}$ by $\nabla \alpha_F$ which is not anomalous; $D_4 F_{ab}$ is also not anomalous; $R_{4}{}^{\gamma} \cdot F_{\gamma 4}$ is harmless since we can pose $L^{\infty}_{(sc)}$ estimates on $\Upsilon$'s. Thus, we can easily show that $T_{21}$ enjoys the same estimates as $T_1$.

For $T_{22}$, recall that $J_a = D^{\gamma}L^{\delta} \cdot D_{\gamma}F_{\delta a} - R_{4}{}^{\gamma}F_{\gamma a} + R_{4}{}^{\gamma}{}_{a}{}^{\delta}F_{\gamma\delta}$. As we did for $T_{21}$, the second term $R_{4}{}^\gamma \cdot F_{\gamma a}$ is harmless so that we ignore it. We turn to the third term. If the curvature term $R_{4}{}^{\gamma}{}_{a}{}^{\delta}$ is not anomalous, we can still use the same estimates on $T_1$ to bound it. Finally, we bound the only possible dangerous term $R_{4}{}^{b}{}_{a}{}_{4}\cdot F_{b3}= \alpha \cdot \alphab_F$ as follows,
\begin{align*}
|\delta \doubleint_{\D(u,\ub)}\nabla_4 \alpha_F \cdot \alpha \cdot \alphab_F |&\lesssim \delta^{\frac{1}{2}}\int_0^{u} \|\nabla_4 \alpha_F\|_{L^2_{(sc)}(H_{u'})}\|\alpha\|_{L^4_{(sc)}(H_{u')}}\| \alphab_F\|_{L^4_{(sc)}(H_{u'})}\\
&\lesssim C \delta^{\frac{1}{4}}\int_0^{u} \|\nabla_4 \alpha_F\|_{L^2_{(sc)}(H_{u'})}.
\end{align*}

It remains to estimate the first term $D^{\gamma}L^{\delta} \cdot D_{\gamma}F_{\delta a}$ in $J_a$ which reads as
\begin{equation*}
 D^{\gamma}L^{\delta} \cdot D_{\gamma}F_{\delta a} =\Omega^{-1}(\chi_{bc} \cdot D_b F_{ca} -2\omegab \cdot D_4 F_{4a} -\eta_b\cdot D_4F_{ba} + \eta_b \cdot D_b F_{4a} ).
\end{equation*}
We can easily show that the last three terms have the same estimates as $T_1$. For the first one,
\begin{align*}
&|\doubleint_{\D(u,\ub)}\Omega^{-1} D_4 F_{4a} \cdot \chi_{bc} \cdot D_b F_{ca}| \lesssim \doubleint_{\D(u,\ub)}|\nabla_4 \alpha_F + \psi\cdot \alpha_F| |\chi| |\nabla \sigma_F + \tr\chib \alpha_F + \psi \cdot \Upsilon|\\
&\lesssim \doubleint_{\D(u,\ub)}|\nabla_4 \alpha_F | |\chi| |\nabla \sigma_F| + \doubleint_{\D(u,\ub)}|\nabla_4 \alpha_F | |\chi| |\sigma_F| + C \delta^{\frac{1}{4}}.
\end{align*}
Thus, it still has the same estimates as $T_1$. Putting things together, we derive
\begin{equation*}
 \|\nabla_4 \alpha_F\|^2_{L^2_{(sc)}(H_u^{(0,\ub)})} \lesssim C\delta^{\frac{1}{4}} + \|\nabla_4 \alpha_F\|^2_{L^2_{(sc)}(H_0^{(0,\ub)})} + \int_{0}^u \| \nabla_4 \alpha_F\|_{L^2_{(sc)}(H_{u'})}.
\end{equation*}
Thus, Gronwall's inequality yields
\begin{equation}\label{energy_estimates_for_nabla_4_alpha_F}
 \|\nabla_4 \alpha_F\|_{L^2_{(sc)}(H_u^{(0,\ub)})} \lesssim C\delta^{\frac{1}{8}} + \|\nabla_4 \alpha_F\|^2_{L^2_{(sc)}(H_0^{(0,\ub)})} \lesssim C\delta^{\frac{1}{8}} + \Izero \delta^{-\frac{1}{2}}.
\end{equation}

We take $X=\Lb$ in \eqref{D_X of F} and $X = \Lb$ in \eqref{stokes_energy_momentum} to derive
\begin{align*}
 \|\nabla_3 \alphab_F\|^2_{L^2_{(sc)}(\Hb_{\ub}^{(0,u)})} &\lesssim C\delta^{\frac{1}{4}} +\Izero \delta^{-1}+ \delta  \doubleint_{\D(u,\ub)} {}^{(\Lb)}\!\pih \cdot T[D_3 F] + \Divergence T[D_3 F](e_3)\\
&=C\delta^{\frac{1}{4}} +\Izero \delta^{-1} + T_1 + T_2.
\end{align*}

For $T_1$, we write its integrand $\,^{(\Lb)}\!\pih \cdot T[D_3 F]$ as $\psi^{s_1} \cdot \Psi(D_4 F)^{s_2} \cdot \Psi(D_4 F)^{s_3}$ with $s_1 + s_2 + s_3 = 1$. Compared to $T_1$ term we encountered for $\nabla_4 \alpha_F$, the appearance of $\tr\chib_0$ as $\psi^{s_1}$ may cause a loss of $\delta^{\frac{1}{2}}$. An additional complication comes from $\alpha(D_3 \Upsilon)$: although it can be replaced by $\nabla_3 \alpha_F$, but according to \eqref{NM_Lb_alpha}, $\nabla_3 \alpha_F$ is anomalous. Fortunately, this anomaly can be traced back to initial data. \eqref{NM_Lb_alpha} shows the anomaly is from $\alpha_F$, and the energy estimates in previous section shows the anomaly of $\alpha_F$ is from the initial data. Hence,
\begin{equation}\label{anomaly_of_nabla_3_alpha_F}
 \|\nabla_3 \alpha_F\|_{L^2_{(sc)}(H_u^{(0,\ub)})} 
 \lesssim \delta^{-\frac{1}{2}} \Izero +  \Fone + C\delta^{\frac{1}{8}}
\end{equation}
In view of Remark \ref{observation_2}, we observe that the double anomalies $|\nabla_3 \alpha_F|^2$ is absent and $\tr\chib_0$ is neither coupled with $\nabla_3 \alpha_F$ nor $\nabla_3 \alphab_F$. Thus, $T_1$ can be bounded as follows
\begin{align*}
T_1 &\lesssim (1+\delta^{\frac{1}{2}}\|\psi\|_{L^{\infty}_{(sc)}})\int_{0}^{\ub} \delta^{-1}\| \Psi(D_3 F)\|_{L^2_{(sc)}(\Hb_{\ub'})}\| \Psi(D_3 F)\|_{L^2_{(sc)}(\Hb_{\ub'})}d \ub' \\
&\lesssim C\delta^{\frac{1}{2}}\int_{0}^{\ub}\delta^{-1} (\| \nabla_3 \alphab_F\|^2_{L^2_{(sc)}(\Hb_{\ub'})} + \Foneb \cdot\| \nabla_3 \alphab_F\|_{L^2_{(sc)}(\Hb_{\ub'})}) d \ub'+ C\delta^{\frac{1}{2}}\Foneb^2 \\
&\quad + \doubleint_{\D(u,\ub)} |\chibh|\cdot |\nabla_3 \alpha_F| \cdot |\nabla_3 \alphab_F|\\
&\lesssim C\delta^{\frac{1}{4}} + C\delta^{\frac{1}{2}}\int_{0}^{\ub}\delta^{-1}\| \nabla_3 \alphab_F\|^2_{L^2_{(sc)}(\Hb_{\ub'})}  d \ub'+ \delta\doubleint_{\D(u,\ub)} |\chibh|\cdot |\nabla_3 \alpha_F| \cdot |\nabla_3 \alphab_F|.
\end{align*}
We bound $\chibh$ in ${L^{\infty}_{(sc)}}$ norm to estimate the bulk integral as follows,
\begin{align*}
 &\delta\doubleint_{\D(u,\ub)} |\chibh|\cdot |\nabla_3 \alpha_F| \cdot |\nabla_3 \alphab_F| \lesssim C \delta^{\frac{1}{2}}\int_{0}^{\ub} \delta^{-1}\| \nabla_3 \alpha_F\|_{L^2_{(sc)}(\Hb_{\ub'})}\|\nabla_3 \alphab_F\|_{L^2_{(sc)}(\Hb_{\ub'})}\\
&\lesssim C^2 \int_{0}^{\ub} \delta^{-1}\| \nabla_3 \alpha_F\|^2_{L^2_{(sc)}(\Hb_{\ub'})} +\int_{0}^{\ub} \delta^{-1}\| \nabla_3 \alphab_F\|^2_{L^2_{(sc)}(\Hb_{\ub'})}.
\end{align*}
Together with \eqref{anomaly_of_nabla_3_alpha_F}, we have
\begin{align*}
T_1 &\lesssim C\delta^{\frac{1}{4}} + (1+C\delta^{\frac{1}{2}})\int_{0}^{\ub}\delta^{-1}\| \nabla_3 \alphab_F\|^2_{L^2_{(sc)}(\Hb_{\ub'})}  d \ub'+ C^2 \Izero.
\end{align*}

For $T_2$, by ignoring the part involving the Hodge dual of $D_3 F$, we rewrite $\Divergence T[D_3 F](\Lb)$ as
\begin{align*}
 \Divergence T[D_3 F](e_3) = -\frac{1}{2}D_3 F_{34} \cdot J_3 + D_3 F_{3a} \cdot J_a = T_{21} + T_{22}.
\end{align*}

For $T_{21}$, we observe that $-\frac{1}{2}D_3 F_{34}$ is not anomalous and
\begin{equation*}
 J_3 =D^{a}{\Lb}^b \cdot D_a F_{b3} -\Omega^{-1}\etab_a D_3 F_{ab}- R_{3}{}^{\gamma}F_{\gamma 3} .
\end{equation*}
The principle term is $D^{a}{\Lb}^b \cdot D_a F_{b3}$ since $D^{a}{\Lb}^b$ contributes a constant $\tr\chib_0$. We can easily show that the remaining terms enjoy the same estimates as $T_1$. For the principle term,
\begin{align*}
\delta\doubleint_{\D(u,\ub)} |\nabla_3 \rho| \cdot |D_a F_{b3}| = \delta \doubleint_{\D(u,\ub)} |\nabla_3 \rho| \cdot |\nabla \alphab + \chib\cdot (\rho_F+\sigma_F) + (\eta + \etab)\cdot \alphab_F|.
\end{align*}
It is clear that the estimates on $T_1$ are sufficient to dominate this bulk integral.

For $T_{22}$, recall that
\begin{align*}
 J_a &= D^{\gamma}{\Lb}^{\delta} \cdot D_{\gamma}F_{\delta a} - R_{3}{}^{\gamma}F_{\gamma a} + R_{3}{}^{\gamma}{}_{a}{}^{\delta}F_{\gamma\delta}.
\end{align*}
Since $R_{3}{}^{\gamma}{}_{a}{}^{\delta}$ is never anomalous, we can easily show that last two terms in $J_a$ can be bounded by the estimates on $T_1$. We now rewrite the principle term $D^{\gamma}{\Lb}^{\delta} \cdot D_{\gamma}F_{\delta a}$ as
\begin{equation*}
 D^{\gamma}L^{\delta} \cdot D_{\gamma}F_{\delta a} =\Omega^{-1}(\chib_{bc} \cdot D_b F_{ca} -2\omega \cdot D_3 F_{3a} -\etab_b\cdot D_3 F_{ba} + \etab_b \cdot D_b F_{3a} ).
\end{equation*}
Once again the last three terms can be bounded by the estimates of $T_1$. The only remaining term can be bounded as follows (we shall treat $\chib$ as a constant)
\begin{align*}
&\quad |\delta \doubleint_{\D(u,\ub)}\Omega^{-1} D_3 F_{3a} \cdot \chib_{bc} \cdot D_b F_{ca}| \lesssim \delta \doubleint_{\D(u,\ub)}|\nabla_3 \alphab_F + \psi \cdot \alphab_F| |\nabla \sigma_F + \tr\chib \alpha_F + \psi \cdot \Upsilon|\\
&\lesssim \delta \doubleint_{\D(u,\ub)}|\nabla_3 \alphab_F | |\nabla \sigma_F| + \delta \doubleint_{\D(u,\ub)}|\nabla_3 \alphab_F | |\alpha_F| + C.
\end{align*}
Thus, it still enjoys the same estimate as $T_1$. Putting things together, we derive
\begin{align*}
 \|\nabla_3 \alphab_F\|^2_{L^2_{(sc)}(\Hb_{\ub}^{(0,u)})} &\lesssim C\delta^{\frac{1}{4}} +\Izero + (1+C\delta^{\frac{1}{2}})\int_{0}^{\ub}\delta^{-1}\| \nabla_3 \alphab_F\|^2_{L^2_{(sc)}(\Hb_{\ub'})}  d \ub'+ C^2 \Izero
\end{align*}
Thus, Gronwall's inequality yields
\begin{equation}\label{energy_estimates_for_nabla_3_alphab_F}
 \delta^{\frac{1}{2}}\|\nabla_3 \alphab_F\|_{L^2_{(sc)}(\Hb_{\ub}^{(0,u)})} \lesssim C\delta^{\frac{1}{2}}.
\end{equation}

\subsubsection{Energy Estimates on Non-anomalies.}
We use $\Lh_O F$ instead of $F$ in \eqref{stokes_energy_momentum} to derive estimates. 
We first take $X=L$ to derive
\begin{align*}
 &\quad \|\alpha(\Lh_O F)\|^2_{L^2_{(sc)}(H_u^{(0,\ub)})} +\|(\rho(\Lh_O F), \sigma(\Lh_O F))\|^2_{L^2_{(sc)}(\Hb_{\ub}^{(0,u)})} \lesssim \Izero + T_1 + T_2\\
&=\|\alpha(\Lh_O F)\|^2_{L^2_{(sc)}(H_0^{(0,\ub)})} + \doubleint_{\D(u,\ub)} \,^{(L)}\!\pih \cdot T[\Lh_O F] + \Divergence T[\Lh_O F](L);
\end{align*}
We then take $X=\Lb$ to derive
\begin{align*}
 &\quad \|(\rho(\Lh_O F), \sigma(\Lh_O F))\|^2_{L^2_{(sc)}(H_u^{(0,\ub)})} +\|\alphab(\Lh_O F)\|^2_{L^2_{(sc)}(\Hb_{\ub}^{(0,u)})} \lesssim \Izero + S_1 + S_2\\
&\lesssim \|(\rho(\Lh_O F), \sigma(\Lh_O F))\|^2_{L^2_{(sc)}(H_0^{(0,\ub)})} + \delta^{-1}\doubleint_{\D(u,\ub)} \,^{(\Lb)}\!\pih \cdot T[\Lh_O F] + \Divergence T[\Lh_O F](\Lb).
\end{align*}
Adding things together, we have
\begin{equation}\label{Lie_O_Divergence_estimates_first_step}
 \|(\alpha, \rho, \sigma)(\Lh_O F)\|^2_{L^2_{(sc)}(H_u^{(0,\ub)})} +\|(\rho, \sigma,\alphab)(\Lh_O F)\|^2_{L^2_{(sc)}(\Hb_{\ub}^{(0,u)})} \lesssim \Izero + (T_1 + S_1) + (T_2 +S_2).
\end{equation}

For $T_1$ and $S_1$ terms,  regardless of the appearance of $\tr\chib_0$, we have
\begin{equation*}
 T_1 + S_1 \lesssim \int_{0}^{\ub} \delta^{-1} \|\Upsilon(\Lh_O F)\|^2_{L^2_{(sc)}(\Hb_{\ub'})} +\int_{0}^u \|\Upsilon(\Lh_O F)\|^2_{L^2_{(sc)}(H_{u'})}.
\end{equation*}
Thus, back to \eqref{Lie_O_Divergence_estimates_first_step}, these two integrals can be absorbed by Gronwall's inequality.

For $T_2$ and $S_2$ terms, their integrands have following two types schematic expressions,
\begin{equation*}
 (I) \quad \Upsilon (\Lh_O F) \cdot {}^{(O)}\!\pih \cdot \Upsilon(D_\gamma F), \qquad (II) \quad \Upsilon (\Lh_O F) \cdot D{}^{(O)}\!\pih \cdot \Upsilon.
\end{equation*}

For type $(II)$ terms, we bound them as follows
\begin{align*}
 |(II)| &\lesssim \delta^{\frac{1}{2}} \|\Upsilon\|_{L^{\infty}_{(sc)}} \|D{}^{(O)}\!\pi\|_{L^2_{(sc)}}(\int_{0}^{\ub} \delta^{-1} \|\Upsilon(\Lh_O F)\|_{L^2_{(sc)}(\Hb_{\ub'})} +\int_{0}^u \|\Upsilon(\Lh_O F)\|_{L^2_{(sc)}(H_{u'})})\\
&\lesssim C \delta^{\frac{1}{2}}(\int_{0}^{\ub} \delta^{-1} \|\Upsilon(\Lh_O F)\|_{L^2_{(sc)}(\Hb_{\ub'})}  +\int_{0}^u \|\Upsilon(\Lh_O F)\|_{L^2_{(sc)}(H_{u'})} ).
\end{align*}
Thus, back to \eqref{Lie_O_Divergence_estimates}, we can handle these two integrals by Gronwall's inequality.

For type $(I)$ terms, the situation is more complicated since $\Psi(D_\gamma F)$ may involve anomalies. We divide it into two cases $(I)=(I)_1+(I)_2$.

For $(I)_1$, we require $\Upsilon(D_\gamma F) \notin \{\alpha(D_4 F), \alphab(D_3 F)\}$. According to Maxwell equations, we have $\Upsilon(D_\gamma F) = \nabla \Upsilon + (1+ \psi)\cdot \Upsilon$. Hence, $|(I)_1| $ can be bounded by following terms,
\begin{align*}
 &\quad \doubleint_{\D(u,\ub)} |\Upsilon(\Lh_O F)||{}^{(O)}\pi|(|\nabla \Upsilon|+ |(1+\psi)\cdot \Upsilon|)\\
&\lesssim \delta^{\frac{1}{2}} \|{}^{(O)}\pi\|_{L^{\infty}_{(sc)}} (\int_{0}^{\ub} \delta^{-1} \|\Upsilon(\Lh_O F)\|_{L^2_{(sc)}(\Hb)} \|\nabla \Upsilon\|_{L^2_{(sc)}(\Hb)}+ \int_{0}^u \|\Upsilon(\Lh_O F)\|_{L^2_{(sc)}(H)}\|\nabla \Upsilon\|_{L^2_{(sc)}(H)} )\\
&\quad  + \delta^{\frac{1}{2}}\|\Upsilon\|_{L^{\infty}_{(sc)}} (\int_{0}^{\ub} \delta^{-1} \|\Upsilon(\Lh_O F)\|_{L^2_{(sc)}(\Hb)}\|{}^{(O)}\pi\|_{L^2_{(sc)}(\Hb)} +\int_{0}^u \|\Upsilon(\Lh_O F)\|_{L^2_{(sc)}(H)} \|{}^{(O)}\pi\|_{L^2_{(sc)}(H)}).
\end{align*}
Thus, $(I)_1$ enjoys the same estimates as $(II)$'s.

For $(I)_2$, we require $\Psi(D_\gamma F) \in \{\alpha(D_4 F), \alphab(D_3 F)\}$. Recall that those terms appear through the following expressions $\Lh_O F_{X}{}^\mu \cdot {}^{(O)}\pi^{\alpha\beta} \cdot D_{\alpha}F_{\beta\mu}$. We deduce that $\Psi(D_\gamma F) \neq \alphab(D_3 F)$ since ${}^{(O)}\pi^{3 a}=0$. This forces the integrand of $(I)_2$ to be
\begin{equation*}
 \Lh_O F_{43} \cdot {}^{(O)}\pi^{4a} \cdot D_{4}F_{4a} = \rho(\Lh_O F) \cdot Z \cdot \nabla_4 \alpha_F + l.o.t.
\end{equation*}
We shall use one more derivative on $\nabla_4 \alpha_F$ to save $\delta^{\frac{1}{4}}$.
\begin{align*}
 \doubleint_{\D(u,\ub)} |\rho(\Lh_O F)||Z||\nabla_4 \alpha_F| &\lesssim \delta^{\frac{1}{2}}\doubleint_{\D(u,\ub)} \|\rho(\Lh_O F)\|_{{L^2_{(sc)}(S_{u,\ub})}}\|Z\|_{{L^4_{(sc)}(S_{u,\ub})}}\|\nabla_4 \alpha_F\|_{{L^4_{(sc)}(S_{u,\ub})}}\\
&\lesssim C \delta^{\frac{1}{4}}\int_{0}^u \|\Psi(\Lh_O F)\|_{L^2_{(sc)}(H_{u'})}.
\end{align*}
Putting the estimates on $(I)_1$, $(I)_2$ and $(II)$ back into \eqref{Lie_O_Divergence_estimates_first_step}, after using Gronwall's inequality, we have
\begin{equation}\label{Lie_O_Divergence_estimates}
 \|(\alpha, \rho, \sigma)(\Lh_O F)\|_{L^2_{(sc)}(H_u^{(0,\ub)})} +\|(\rho, \sigma,\alphab)(\Lh_O F)\|_{L^2_{(sc)}(\Hb_{\ub}^{(0,u)})} \lesssim \Izero.
\end{equation}
As a consequence of Lemma \ref{comparison_maxwell_Lie}, we have
\begin{equation}
 \|(\nabla \rho_F, \nabla \sigma_F)\|_{L^2_{(sc)}(H_u^{(0,\ub)})} +\|(\nabla \rho_F, \nabla \sigma_F,\nabla\alphab_F)\|_{L^2_{(sc)}(\Hb_{\ub}^{(0,u)})} \lesssim \Izero + C\delta^{\frac{1}{8}}.
\end{equation}
It remains to control $\|\nabla \alpha_F\|_{L^2_{(sc)}(H_u^{(0,\ub)})}$. Let $H= H_u^{(0,\ub)}$, according to Remark \ref{difference_Lie_O_alpha}, we have
\begin{align*}
\|\nabla \alpha_F\|_{L^2_{(sc)}(H)} &\leq   \|\alpha(\Lh_O F)\|_{L^2_{(sc)}(H)} +\|\nabla \alpha_F-\alpha(\Lh_O F)\|_{L^2_{(sc)}(H)} \\
&\lesssim \Izero + C\delta^{\frac{1}{8}} + \delta^{\frac{1}{2}} \|H\|_{L^4_{(sc)}(H)} \|\alpha_F\|_{L^4_{(sc)}(H)}.
\end{align*}
It is easy to see that the anomaly of $H$ comes from the initial data, thus
\begin{equation*}
 \|H\|_{L^4_{(sc)}({H_u^{(0,\ub)}})} \lesssim C_0 \delta^{-\frac{1}{4}} + C.
\end{equation*}
For $\|\alpha_F\|_{L^4_{(sc)}({H_u^{(0,\ub)}})}$, thanks to the estimates on $\alpha_F$ and $\nabla_4 \alpha_F$ in previous sections, we have
\begin{align*}
 \|\alpha_F\|_{L^4_{(sc)}(H)} &\lesssim (\delta^{\frac{1}{2}} \| \alpha_F\|_{L^2_{(sc)}(H)} + \|\nabla \alpha_F\|_{L^2_{(sc)}(H)} )^{\frac{1}{2}} (\delta^{\frac{1}{2}} \| \alpha_F\|_{L^2_{(sc)}(H)} + \|\nabla_4 \alpha_F\|_{L^2_{(sc)}(H)} )^{\frac{1}{2}}\\
&\lesssim (\Izero  + \Izero \delta^{-\frac{1}{2}} + C )^{\frac{1}{2}}(\Izero +C\delta^{\frac{1}{8}} + \|\nabla\alpha_F\|_{L^2_{(sc)}(H)} )^{\frac{1}{2}}.
\end{align*}
We combine those estimates with \eqref{Lie_O_Divergence_estimates} to derive
\begin{align*}
\|\nabla \alpha_F\|_{L^2_{(sc)}(H_u^{(0,\ub)})} \lesssim \Izero + C\delta^{\frac{1}{8}} + C(\Izero)(\Izero + C\delta^{\frac{1}{8}} + \|\nabla \alpha_F\|_{L^2_{(sc)}(H_u^{(0,\ub)})} )^{\frac{1}{2}}.
\end{align*}
This implies the desired estimates $\|\nabla \alpha_F\|_{L^2_{(sc)}(H_u^{(0,\ub)})} \lesssim \Izero + C\delta^{\frac{1}{8}}$.

In view of \eqref{Fzero_Fzerob_Estimates}, we summarize the energy estimates derived so far as follows
\begin{equation}\label{Fone_Foneb_Estimates}
 \Fzero + \Fzerob+ \Fone + \Foneb \lesssim \Izero + C \delta^{\frac{1}{8}}.
\end{equation}

\subsection{Energy Estimates on Curvature}

We first recall the divergence equations for Weyl curvature,
\begin{equation*}
 D^{\alpha}W_{\alpha\beta\gamma\delta} = J_{\beta\gamma\delta} = \frac{1}{2}(D_{\gamma}R_{\beta\delta}-D_{\delta}R_{\beta\gamma}),\quad D^{\alpha}{}\Wstar_{\alpha\beta\gamma\delta} =J^{*}_{\beta\gamma\delta}=\frac{1}{2}(D_{\mu}R_{\beta\nu}-D_{\nu}R_{\beta\mu})\epsilon^{\mu\nu}{}_{\gamma\delta}.
\end{equation*}

Due to signature considerations, $J$ and $J^*$ excludes the double anomalies $\alpha_F \cdot \nabla_4 \alpha_F$ since $J_{444} = 0$. By direct computations, we can see that other possible anomalies in $D_X R_{\mu\nu}$ are of the form $\chib_0 \cdot \alpha_F \cdot \alpha_F$, $\alpha_F \cdot \nabla_3 \alpha_F$, $\Upsilon_g \cdot \nabla_4 \alpha_F$ or $\Upsilon_g \cdot \nabla_3 \alphab_F$. Furthermore, $\alpha_F \cdot \nabla_3\alphab_F$ only appears in $D_3 R_{ab}$. The remaining terms are either trilinear (which will extremely easy to estimate) or without anomaly in derivatives.

We also have a divergence equation for Bel-Robinson tensor $Q[W]$,
\begin{equation*}
D^\alpha  Q[W]_{\alpha\beta\gamma\delta}=
W_{\beta}{}^{\mu}{}_{\delta}{}^{\nu}J_{\mu\gamma\nu} +
W_{\beta}{}^{\mu}{}_{\gamma}{}^{\nu}J_{\mu\delta\nu} +
\Wstar_{\beta}{}^{\mu}{}_{\delta}{}^{\nu}J^*{}_{\mu\gamma\nu} +
\Wstar_{\beta}{}^{\mu}{}_{\gamma}{}^{\nu}J^*{}_{\mu\delta\nu}.
\end{equation*}

\subsubsection{Anomalous Estimates}
We apply \eqref{stokes_Bel_Robinson} with $(X,Y,Z)=(L,L,L)$ to derive
\begin{align*}
 \|\alpha\|^2_{L^2_{(sc)}(H_u^{(0,\ub)})}  + \|\beta\|^2_{L^2_{(sc)}(\Hb_{\ub}^{(0,u)})}  & = \|\alpha\|^2_{L^2_{(sc)}(H_0^{(0,\ub)})} + \delta\doubleint_{\D(u,\ub)} (\Divergence Q[W]+  \pih \cdot
 Q[W])(L,L,L)\\
 &=\|\alpha\|^2_{L^2_{(sc)}(H_0^{(0,\ub)})} + T_1 +T_2.
\end{align*}

We first bound $T_2$ whose integrand can be written as $\psi^{s_1} \cdot \Psi^{s_2} \cdot \Psi^{s_3}$ with $s_1 + s_2 + s_3 = 4$. We observe that $\psi^{s_1}$ cannot be $\chib$. Hence,
\begin{align*}
 |T_2| & \lesssim C \delta^{\frac{1}{2}}\|\alpha\|^2_{L^2_{(sc)}(H_u^{(0,\ub)})} + C \delta^{\frac{1}{2}} \Rzero \|\alpha\|_{L^2_{(sc)}(H_u^{(0,\ub)})} + C \delta^{\frac{1}{2}} \Rzero^2.
\end{align*}

We turn to $T_1$ whose integrand can be written as $\Psi^{s_1} \cdot \Upsilon^{s_2} \cdot (D\Upsilon)^{s_3}$ with $s_1 + s_2 + s_3 = 4$.
They all come through $ W_{4}{}^{\mu}{}_{4}{}^{\nu} (D_4 R_{\mu\nu} - D_{\nu}R_{\mu 4})$. We ignored similar terms for the Hodge dual part since they can treated exactly in the same manner. Due to signature considerations, $\nabla_3 \alphab_F$ does not appear. Thus, we can divided the integrands of $T_1$ into following cases,
\begin{equation*}
(I) \quad \Psi \cdot \alpha_F \cdot \alpha_F, \quad (II) \quad
\Psi \cdot \Upsilon \cdot \nabla \Upsilon, \quad (III) \quad \Psi \cdot
\Upsilon_g \cdot \nabla_4 \alpha_F.
\end{equation*}

For type $(I)$ terms, we retrieve the anomalies back to the initial data
\begin{align*}
 |(I)| & \lesssim \delta^{\frac{1}{2}}\int_{0}^{u}\int_{0}^{\ub} \|\alpha_F\|^2_{L^4_{(sc)}(u',\ub')} \|\Psi\|_{L^2_{(sc)}(u',\ub')}
\lesssim \Izero \int_{0}^{u} \|\alpha\|_{L^2_{(sc)}(\Hb_{\ub}^{(0,u)})}+
C\delta^{\frac{1}{4}}.
\end{align*}

For type $(II)$ terms, we have
\begin{align*}
 |(II)| &\lesssim \delta^{\frac{1}{2}}\int_{0}^{u}\int_{0}^{\ub} \|\Upsilon\|_{L^4_{(sc)}(u',\ub')}\|\nabla \Upsilon\|_{L^4_{(sc)}(u',\ub')} \|\Psi\|_{L^2_{(sc)}(u',\ub')} \lesssim C \delta^{\frac{1}{4}} \int_{0}^{u} \|\alpha\|_{L^2_{(sc)}(\Hb_{\ub}^{(0,u)})}+
C\delta^{\frac{1}{4}}.
\end{align*}

For type $(III)$ terms, we bound them as follows,
\begin{align*}
 |(III)| &\lesssim \int_{0}^{u}\int_{0}^{\ub} \|\Upsilon\|_{L^4_{(sc)}(u',\ub')}\|\nabla_4 \alpha_F\|_{L^4_{(sc)}(u',\ub')} \|\Psi\|_{L^2_{(sc)}(u',\ub')} du' d\ub'\\
 &\lesssim C\doubleint_{\D(u,\ub)}
(\|\nabla_4 \alpha_F\|^{\frac{1}{2}}_{{L^2_{(sc)}(S_{u,\ub})}}\|\nabla\nabla_4\alpha_F\|^{\frac{1}{2}}_{{L^2_{(sc)}(S_{u,\ub})}}
+ \delta^{\frac{1}{4}}\|\nabla_4\alpha_F\|_{{L^2_{(sc)}(S_{u,\ub})}}
)\|\Psi\|_{{L^2_{(sc)}(S_{u,\ub})}}\\
&\lesssim C \delta^{\frac{1}{4}} \int_{0}^{u} \|\alpha\|_{L^2_{(sc)}(\Hb_{\ub}^{(0,u)})}+
C\delta^{\frac{1}{4}}.
\end{align*}
Combining all those estimates, we derive
\begin{equation*}
 \|\alpha\|^2_{L^2_{(sc)}(H_u^{(0,\ub)})}  + \|\beta\|^2_{L^2_{(sc)}(\Hb_{\ub}^{(0,u)})}  \lesssim  \|\alpha\|^2_{L^2_{(sc)}(H_0^{(0,\ub)})}+ C \delta^{\frac{1}{4}} + C(\Izero)\int_{0}^{u} \|\alpha\|_{L^2_{(sc)}(\Hb_{\ub}^{(0,u)})}.
\end{equation*}
After using Gronwall's inequality, we have
\begin{equation}\label{estimatesforalpha}
 \Rzero[\alpha] + \Rzerob[\beta] \lesssim \Izero +  C \delta^{\frac{1}{4}}.
\end{equation}

\subsubsection{Non-anomalous Estimates}
We apply \eqref{stokes_Bel_Robinson} with $(X,Y,Z)=(L,L,\Lb)$ to derive
\begin{align*}
 \|\beta\|^2_{L^2_{(sc)}(H_u^{(0,\ub)})}  + \|(\rho,\sigma)\|^2_{L^2_{(sc)}(\Hb_{\ub}^{(0,u)})} & = \|\beta\|^2_{L^2_{(sc)}(H_0^{(0,\ub)})}   +\doubleint_{\D(u,\ub)} (\Divergence Q[W] + \pih \cdot Q[W])(L,L,\Lb)\\
&=\|\beta\|^2_{L^2_{(sc)}(H_0^{(0,\ub)})} + T_1 +T_2.
\end{align*}

For $T_2$, its integrands can be written as $\psi^{s_1} \cdot \Psi^{s_2} \cdot \Psi^{s_3}$ with $s_1 + s_2 + s_3 = 3$.
By signature considerations, there is no term quadratic in $\alpha$. We still have to pay attention to those dangerous terms which contain  $\tr\chib$ and $\alpha$. Since $\tr\chib$ can only appear through $\,^{(\Lb)}\pi$, it must come from the following expression (we use \textit{l.o.t.} to denote those terms which do not contain $\tr\chib$, they can be easily bounded)
\begin{align*}
Q[W]_{ab44} \cdot \,^{(\Lb)}\pi^{ab}& = \tr\chib \cdot Q[W]_{aa44}+ l.o.t.=\tr\chib \cdot |\beta|^2 + l.o.t.
\end{align*}
Thus, $\tr\chib$ will not be coupled with $\alpha$ in the integrand of $T_2$. Among the terms of the integrand of $T_2$, for those terms of the form $\tr\chib_0 \cdot \Psi_g \cdot \Psi_g$, one can use Gronwall's inequality to eliminate them; for those terms of the form $\psi \cdot \Psi_g \cdot \Psi_g$, they can be estimated by $C \delta^{\frac{1}{2}}$.

It remains to bound the terms of the form $\psi \cdot \Psi \cdot \alpha$. If $\psi \notin \{\chih, \chibh\}$, we can gain $\delta^{\frac{1}{4}}$ by using the $L^4_{(sc)}$ estimates on $\alpha$ and $\psi$.
We still have to estimate the following two types terms: $\chih \cdot \Psi \cdot \alpha$ and $\chibh \cdot \Psi \cdot \alpha$. Notice that $\chih \cdot \Psi \cdot \alpha$ can not appear. The reason is as follows: assume this was the case, then by the signature considerations, $\Psi$ must be $\alphab$. Since the original expression of all the terms are $\pi \cdot Q[W]$. Notice that $\alpha \cdot \alphab$ never appears in a null component of $Q[W]$ which yields a contradiction.
To bound $\chibh \cdot \Psi \cdot \alpha$, we retrieve the anomaly of $\|\chibh\|_{{L^4_{(sc)}(S_{u,\ub})}}$ to initial data. According to {\bf{Theorem A}}, $\|\chibh\|_{{L^4_{(sc)}(S_{u,\ub})}} \lesssim \delta^{-\frac{1}{4}} \Izero + C$. Thus,
Thus,
\begin{align*}
|\doubleint_{\D(u,\ub)} \chibh \cdot \Psi_g \cdot \alpha| &\lesssim \OSzerofour[\chibh]\cdot \Rzero \cdot\Rzero[\alpha]^{\frac{1}{2}}\Rone[\alpha]^{\frac{1}{2}} \lesssim \Izero \cdot \Rzero \cdot\Rzero[\alpha]^{\frac{1}{2}}\Rone[\alpha]^{\frac{1}{2}} + \delta^{\frac{1}{4}}C.
\end{align*}

We turn to $T_1$ whose integrands can be written as $\Psi^{s_1} \cdot \Upsilon^{s_2} \cdot (D\Upsilon)^{s_3}$ with $s_1 + s_2 + s_3 = 3$. We have to analyze the trilinear structure carefully to control triple anomalies. The other terms are much easier to control so we will not give details. Possible triple anomalies are $\alpha \cdot \alpha_F \cdot \alpha_F$, $\alpha \cdot \alpha_F \cdot \nabla_4 \alpha_F$ and $\alpha \cdot \alpha_F \cdot \nabla_3 \alphab_F$. The first and second cases are excluded by signature considerations. As we remarked in the beginning of the section, the appearance of $ \alpha \cdot \alpha_F \cdot \nabla_3 \alphab_F$ is through the term $\alpha_{ab}\cdot D_3 R_{ab}$. To address the control for this term, we need an extra integration by parts to move bad derivative $\nabla_{3}$ to $\alpha$.
\begin{align*}
 |\doubleint_{D(u,\ub)}\alpha \cdot \nabla_3 R_{ab}| = |\int_{H_u} \alpha \cdot R_{ab} -\int_{H_0}\alpha \cdot R_{ab}  -\doubleint_{D(u,\ub)}\nabla_3 \alpha \cdot R_{ab}|.
\end{align*}
For those three terms at the right-hand side of the above equation, the last is easy to control thanks to \eqref{NBE_Lb_beta}. For the second or the first term, the worst possible quadratic terms in $R_{ab}$ is $\alpha_F \cdot \alphab_F$, we ignore the others and bound this one by
\begin{equation*} |\int_{H_u} \alpha \cdot \alpha_F \cdot \alphab_F| \lesssim \delta^{\frac{1}{2}} \|\alpha_F\|_{L^{4}_{(sc)}(H_u)}\|\alpha\|_{L^4_{(sc)}(H_u)}\|\alphab_F\|_{L^2_{(sc)}(H_u)}.
\end{equation*}
Recall that $\|\alphab_F\|_{L^2_{(sc)}(H_u)} \lesssim \Izero + C\delta^{\frac{1}{8}}$. By losing one derivative, we can trace back the loss of $\delta^{\frac{1}{4}}$ in $\|\alpha\|_{{L^4_{(sc)}(S_{u,\ub})}}$ to initial data, thus we have
\begin{equation*} |\int_{H_u} \alpha \cdot \alpha_F \cdot \alphab_F| \lesssim C\delta^{\frac{1}{4}} + \Izero \cdot \Rzero[\alpha]^{\frac{1}{2}} \cdot \Rone[\alpha]^{\frac{1}{2}}.
\end{equation*}
Adding up the estimates on $T_1$ and $T_2$, we derive
\begin{equation*}
 \|\beta\|^2_{L^2_{(sc)}({H_u^{(0,\ub)}})}  + \|(\rho,\sigma)\|^2_{L^2_{(sc)}(\Hb_{\ub}^{(0,u)})}  \lesssim \Izero + C \delta^{\frac{1}{4}}+\Izero \cdot \Rzero \cdot\Rzero[\alpha]^{\frac{1}{2}}\cdot \Rone[\alpha]^{\frac{1}{2}}.
\end{equation*}

We then take $(X,Y,Z)=(L,\Lb,\Lb)$ or $(X,Y,Z)=(\Lb,\Lb,\Lb)$ in \eqref{stokes_Bel_Robinson}, similar analysis gives
\begin{align*}
 &\quad \|(\rho,\sigma)\|^2_{L^2_{(sc)}({H_u^{(0,\ub)}})}  +  \|\betab\|^2_{L^2_{(sc)}({H_u^{(0,\ub)}})}+ \|\betab\|^2_{L^2_{(sc)}(\Hb_{\ub}^{(0,u)})}+\|\alphab\|^2_{L^2_{(sc)}(\Hb_{\ub}^{(0,u)})} \\
 & \lesssim \Izero + C \delta^{\frac{1}{4}}+\Izero \cdot \Rzero \cdot \Rzero[\alpha]^{\frac{1}{2}} \cdot \Rone[\alpha]^{\frac{1}{2}}.
\end{align*}
Putting all the estimates together, in particular, we use $\Rzero[\alpha] \lesssim \Izero + C \delta^{\frac{1}{4}}$, we can derive
\begin{equation*}
 \Rzero^2 + \Rzerob^2 \lesssim \Izero + C\delta^{\frac{1}{4}}+\Izero \cdot (\Rzero+\Rzerob) \cdot \Rone^{\frac{1}{2}}.
\end{equation*}
Finally, we have energy estimates on curvatures
\begin{equation}\label{energy_on_curvature_no_derivatives}
 \Rzero + \Rzerob \lesssim \Izero + C\delta^{\frac{1}{8}}+C(\Izero) \cdot \R^{\frac{1}{2}}.
\end{equation}


\subsection{Energy Estimates on Second Derivative of Maxwell Field}
\subsubsection{Preliminaries}
We can first take two covariant derivatives of Maxwell field and then take its null component. We can also first take a null component of Maxwell field and then take two horizontal covariant derivatives. The following lemmas record the deviation of these two operations.
\begin{lemma}\label{comparison_maxwell_2}
If $\delta$ is sufficiently small, for $N_1$, $N_2 = e_3$ or $e_4$ and for all the possible null components $\Psi$, we have
\begin{equation*}
 \|\Upsilon(D_{N_1} D_{N_2} F) - \nabla_{N_1} \nabla_{N_2} \Upsilon\|_{L^2_{(sc)}({H_u^{(0,\ub)}})} +  \|\Upsilon(D_{N_1} D_{N_2} F) - \nabla_{N_1} \nabla_{N_2} \Upsilon\|_{L^2_{(sc)}({\Hb_{\ub}^{(0,u)}})} \lesssim C\delta^{\frac{1}{4}}.
\end{equation*}
\end{lemma}
\begin{proof}
The proof is routine. We first compute the commutator to observe that there is no double anomaly. Then, we bound $\Upsilon$'s in $L^{\infty}_{(sc)}$ norms and we bound $\nabla_3 \Upsilon$, $\nabla_4 \Upsilon$ in $L^{4}_{(sc)}$ norms to avoid the loss of $\delta^{\frac{1}{2}}$ due to anomalies.
\end{proof}
\begin{lemma}\label{comparison_maxwell_2_D_a_D_null}
Let $H={H_u^{(0,\ub)}}$ and $\Hb = {\Hb_{\ub}^{(0,\ub)}}$. If $\delta$ is sufficiently small, we can bound the following quantities
\begin{equation*}
\|\alpha(D_a D_4 F) - \nabla_a \nabla_4 \alpha_F\|_{L^2_{(sc)}(H)}, \|\alphab(D_a D_4 F) - \nabla_a \nabla_4 \alphab_F\|_{L^2_{(sc)}(H)}, \|\alpha(D_a D_3 F) - \nabla_a \nabla_3 \alpha_F)\|_{L^2_{(sc)}(H)},
\end{equation*}
\begin{equation*}
\|(\rho(D_a D_4 F) - \nabla_a \nabla_4 \rho_F - \nabla_4 \alpha_F\|_{L^2_{(sc)}(H)}, \|\sigma(D_a D_4 F) - \nabla_a \nabla_4 \sigma_F - \nabla_4 \alpha_F)\|_{L^2_{(sc)}(H)};
\end{equation*}
\begin{equation*}
\|(\alphab(D_a D_4 F) - \nabla_a \nabla_4 \alphab_F\|_{L^2_{(sc)}(\Hb)},\|\alphab(D_a D_3 F) - \nabla_a \nabla_3 \alphab_F\|_{L^2_{(sc)}(\Hb)}, \|\alpha(D_a D_3 F) - \nabla_a \nabla_3 \alpha_F)\|_{L^2_{(sc)}(\Hb)},
\end{equation*}
\begin{equation*}
\|(\rho(D_a D_3 F) - \nabla_a \nabla_3 \rho_F+ \alpha_F \|_{L^2_{(sc)}(\Hb)}, \|\sigma(D_a D_3 F) - \nabla_a \nabla_3 \sigma_F + \alpha_F)\|_{L^2_{(sc)}({\Hb_{\ub}^{(0,u)}})}.
\end{equation*}
by $\Izero + C\delta^{\frac{1}{4}}$.
\end{lemma}
\begin{proof}
 The proof is routine. It is based on direct computation and the energy estimates derived so far. We only treat one term to illustrate the idea. We can compute
\begin{align*}
&\sigma(D_a D_3 F) - \nabla_a \nabla_3 \sigma_F = \frac{1}{2}\chib \wedge \nabla_3 \alpha_F -\chib \cdot \nabla \sigma_F - \zeta_a \cdot \nabla_3 \sigma_F + \eta \wedge \nabla_a \alphab_F + \frac{1}{2}\chi \wedge \nabla_3 \alphab_F\\
&\quad -\frac{1}{2} \chib \cdot (\chib \wedge \alpha_F)-\chib \wedge (\eta\cdot (\rho_F + \sigma_F)) +\nabla \eta \wedge \alphab_F + \chi \wedge (\omegab \cdot \alphab_F) -\frac{1}{2}\chib \cdot (\chi \wedge \alphab_F)- \zeta \cdot (\eta \wedge \alphab_F).
\end{align*}
Thus, the anomaly the right-hand side is $\frac{1}{2}\chib \cdot (\chib \wedge \alpha_F)$. Hence,
\begin{equation*}
\|(\sigma(D_a D_3 F) - \nabla_a \nabla_3 \sigma_F +\frac{1}{2} \chib \wedge (\chib\wedge \alpha_F)\|_{L^2_{(sc)}({\Hb_{\ub}^{(0,u)}})} \lesssim \Izero + C \delta^{\frac{1}{4}}.
\end{equation*}
\end{proof}

Similarly, we have
\begin{lemma}\label{comparison_maxwell_2_D_a_D_b}
If $\delta$ is sufficiently small, we have the following estimates
\begin{equation*}
 \|\alpha(D_a D_b F) - \nabla_a \nabla_b \alpha_F -\nabla_4 \alpha_F\|_{L^2_{(sc)}({H_u^{(0,\ub)}})} \lesssim C \delta^{\frac{1}{4}},
\end{equation*}
\begin{equation*}
 \|\alphab(D_a D_b F) - \nabla_a \nabla_b \alphab_F-\alpha_F\|_{L^2_{(sc)}({H_u^{(0,\ub)}})} \lesssim \Izero+ C \delta^{\frac{1}{4}},
\end{equation*}
\begin{equation*}
 \|\rho(D_a D_b F) - \nabla_a \nabla_b \rho_F\|_{L^2_{(sc)}({H_u^{(0,\ub)}})} +  \|\sigma(D_a D_b F) - \nabla_a \nabla_b \sigma_F\|_{L^2_{(sc)}({H_u^{(0,\ub)}})} \lesssim \Izero+ C \delta^{\frac{1}{4}}.
\end{equation*}
\end{lemma}
We also need the following formula to compute current.
\begin{align*}
 D^{\mu}D_X D_Y F_{\mu\alpha}= J_\alpha(X,Y) &= (D^{\mu}Y^{\nu} D_{\mu}D_X F_{\nu \alpha} + D^{\mu}X^{\nu} D_Y D_{\mu} F_{\nu \alpha}) + D_Y D^{\mu}X^{\nu}D_{\mu}F_{\nu\alpha} \\
& +(R_{Y}{}^{\mu}{}_{\alpha}{}^{\nu}D_X F_{\mu\nu}+ R_{X}{}^{\mu}{}_{\alpha}{}^{\nu}D_Y F_{\mu\nu}) + D_Y R_{X}{}^{\mu}{}_{\alpha}{}^{\nu} F_{\mu\nu} \\
&-(R_Y{}^{\mu}D_X F_{\mu\alpha}+R_X{}^{\mu}D_Y F_{\mu\alpha}+D_Y R_X{}^{\mu}F_{\mu\alpha}),
\end{align*}
\begin{remark}
The the last three terms have schematic expression $\Upsilon \cdot \Upsilon \cdot D \Upsilon$. Since we can use $L^{\infty}_{(sc)}$ estimates on $\Upsilon$, the trilinear structure of these terms gains an extra $\delta$. Thus, in what follows, we shall ignore those (trilinear) terms.
\end{remark}
Similarly, since the difference between the Weyl tensor $W_{\alpha\beta\gamma\delta}$ and $R_{\alpha\beta\gamma\delta}$ is quadratic in $\Upsilon$, we now can replace $R_{\alpha\beta\gamma\delta}$ by $W_{\alpha\beta\gamma\delta}$ in $J_\alpha(X,Y)$. By ignoring trilinear terms, this allows us to write $J_\alpha(X,Y)$ as
\begin{align}\label{D_X D_Y of F}
 J_\alpha(X,Y) = D^{\mu}D_X D_Y F_{\mu\alpha}&= (D^{\mu}Y^{\nu} D_{\mu}D_X F_{\nu \alpha} + D^{\mu}X^{\nu} D_Y D_{\mu} F_{\nu \alpha}) + D_Y D^{\mu}X^{\nu}D_{\mu}F_{\nu\alpha}\\
& +(W_{Y}{}^{\mu}{}_{\alpha}{}^{\nu}D_X F_{\mu\nu}+ W_{X}{}^{\mu}{}_{\alpha}{}^{\nu}D_Y F_{\mu\nu}) + D_Y W_{X}{}^{\mu}{}_{\alpha}{}^{\nu} F_{\mu\nu}.\notag
\end{align}

We need two more lemmas.\footnote{\, $\Lh_O D_\mu F$ should be understood as $\Lh_O (D_\mu F)$ where $D_\mu F$ is a give two form.}
\begin{lemma}\label{two more 1}
Let $H= {H_u^{(0,\ub)}}$ and $\Hb = \Hb_{\ub}^{(0,u)}$. If $\delta$ is sufficiently small, for $\Upsilon \in \{\rho_F, \sigma_F, \alphab_F\}$, we have
\begin{equation*}
 \|\Upsilon(\Lh_O D_4 F) - \nabla_O \nabla_4 \Upsilon\|_{L^2_{(sc)}(H)} + \|\Upsilon(\Lh_O D_3 F) - \nabla_O \nabla_3 \Upsilon\|_{L^2_{(sc)}(\Hb)} \lesssim C\delta^{\frac{1}{4}}.
\end{equation*}
Moreover, the following quantities
\begin{align*}
 \|\alpha(\Lh_O D_4 F) - \nabla_O \nabla_4 \alpha_F-H\cdot \nabla_4 \alpha_F\|_{L^2_{(sc)}(H)}, &\|\alpha(\Lh_O D_3 F) - \nabla_O \nabla_3 \alpha_F-H\cdot \nabla_3 \alpha_F\|_{L^2_{(sc)}(H)}\\
\|\alpha(\Lh_O D_3 F) - \nabla_O \nabla_3& \alpha_F-H\cdot \nabla_3 \alpha_F\|_{L^2_{(sc)}(\Hb)},
\end{align*}
are bounded by $C\delta^{\frac{1}{4}}$.
\end{lemma}
\begin{proof}
Those estimates are again based on direct computation. We only show two of them and the rest can be derived in the same manner. We first consider $\|\rho(\Lh_O D_4 F) - \nabla_O \nabla_4 \rho_F\|_{L^2_{(sc)}(H)}$.
\begin{align*}
 \rho(\Lh_O D_4 F) - \nabla_O \nabla_4 \rho_F &= -\nabla_O \etab \cdot \alpha_F - \etab \cdot \nabla_O \alpha_F + \frac{1}{4}Z \cdot \nabla_4 \alpha_F + \frac{1}{2}\omega \cdot (Z \cdot \alpha_F)\\
&\quad +(2\zeta \cdot O + \tr H)\cdot \nabla_4 \rho_F -(2\zeta \cdot O + \tr H)\cdot(\etab\cdot \alpha_F).
\end{align*}
Notice that $\frac{1}{4}Z \cdot \nabla_4 \alpha_F$ is the most difficult term. We can bound it by using $L^4_{(sc)}$ estimates to save $\delta^{\frac{1}{4}}$ since $Z$ is not anomalous. This yields the desired estimates.

For $\|\alpha(\Lh_O D_4 F) - \nabla_O \nabla_4 \alpha_F-H\cdot \nabla_4 \alpha_F\|_{L^2_{(sc)}(H)}$, we compute
\begin{align*}
 \alpha(\Lh_O (D_4 F))-\nabla_O \nabla_4 \alpha_F &= 2\omega \cdot \nabla_O \alpha_F + 2\nabla_O \omega \cdot \alpha_F + (2\eta\cdot O + \frac{1}{2}(H-{}^t\!H) + \tr H)\nabla_4 \alpha_F \\
&\quad 	+ 2(2\eta\cdot O + \frac{1}{2}(H-{}^t\!H) + \tr H)\cdot(\omega \cdot \alpha_F).
\end{align*}
The estimates are immediate from this formula.
\end{proof}
\begin{remark}
Since we can trace back the anomaly of $H$ to the initial data, we can bound
\begin{equation*}
 \|(\alpha(\Lh_O D_4 F) - \nabla_O \nabla_4 \alpha_F, \alpha(\Lh_O D_3 F) - \nabla_O \nabla_3 \alpha_F)\|_{L^2_{(sc)}(H)},  \|\alpha(\Lh_O D_3 F) - \nabla_O \nabla_3 \alpha_F\|_{L^2_{(sc)}(\Hb)},
\end{equation*}
by $\Izero + C\delta^{\frac{1}{4}}$.
\end{remark}
Similarly, we derive
\begin{lemma}\label{two more 2}
Let $H= {H_u^{(0,\ub)}}$ and $\Hb = \Hb_{\ub}^{(0,u)}$. If $\delta$ is sufficiently small, for $\Upsilon_4 \in \{\alpha_F, \rho_F, \sigma_F\}$ or $\Upsilon_3 \in \{\rho_F, \sigma_F, \alphab_F\}$, we have
\begin{equation*}
 \|\Upsilon_4(\Lh_O D_a F) - \nabla_O \nabla_a \Upsilon_4\|_{L^2_{(sc)}(H)} + \|\Upsilon_3(\Lh_O D_a F) - \nabla_O \nabla_a \Upsilon_3\|_{L^2_{(sc)}(\Hb)} \lesssim \Izero + C \delta^{\frac{1}{4}}.
\end{equation*}
\end{lemma}

\subsubsection{Energy Estimates on Anomalies} Let $H = H_{u}^{(0,\ub)}$. We take $X=Y=L$ in \eqref{D_X D_Y of F} and we also take $X = L$ in \eqref{stokes_energy_momentum} to derive
\begin{align*}
 \|\alpha(D_4 D_4 F)\|^2_{L^2_{(sc)}(H)} &\lesssim \|\alpha(D_4 D_4 F)\|^2_{L^2_{(sc)}(H)} + \delta ^{3} \doubleint_{\D(u,\ub)} \!\!\!{^{(L)}\!\pih \cdot T[D_4 D_4 F] + \Divergence T[D_4 D_4 F](L)}.
\end{align*}
In view of Lemma \ref{comparison_maxwell_2}, we rewrite the above as
\begin{align*}
\|\nabla_4 \nabla_4 \alpha_F\|^2_{L^2_{(sc)}(H)} &\lesssim C\delta^{\frac{1}{4}} + \|\nabla_4 \nabla_4 \alpha_F\|^2_{L^2_{(sc)}(H)} + \delta^{3} \doubleint_{\D(u,\ub)} \!\!\!\!\!{^{(L)}\!\pih \cdot T[D_4 D_4 F] + \Divergence T[D_4 D_4 F](L)}\\
&= C\delta^{\frac{1}{4}} + \|\nabla_4 \nabla_4 \alpha_F\|^2_{L^2_{(sc)}(H_0^{(0,\ub)})} + T_1 + T_2.
\end{align*}

For $T_1$, its integrands can be written as $\psi^{s_1} \cdot \Upsilon(D_4 D_4 F)^{s_2} \cdot \Upsilon(D_4 D_4 F)^{s_3}$ with $s_1 + s_2 + s_3 = 6$.  Since $\chib$ does not appear in ${}^{(L)}\pih$, we can bound $\psi^{s_1}$ in $L^{\infty}_{(sc)}$. Thus,
\begin{align*}
T_1 &\lesssim \delta^{\frac{1}{2}}\|\psi^{s_1}\|_{L^{\infty}_{(sc)}}\int_{0}^u \| \Upsilon(D_4 D_4 F)^{s_2}\|_{L^2_{(sc)}(H_{u'})}\| \Upsilon(D_4 D_4 F)^{s_3}\|_{L^2_{(sc)}(H_{u'})}\\
&\lesssim C \delta^{\frac{1}{2}} \int_{0}^u C\delta^{-\frac{1}{2}} \cdot C\delta^{-\frac{1}{2}}\lesssim C\delta^{-\frac{1}{2}}.
\end{align*}

For $T_2$, we first compute $\Divergence T[D_4 D_4 F](L)$ as follows\footnote{\, We always ignore the Hodge dual part since it can be treated exactly in the same manner.}
\begin{align*}
 \Divergence T[D_4 D_4 F](e_4) = -\frac{1}{2}D_4 D_4 F_{34} \cdot J_4(L,L) + D_4 D_4 F_{4a} \cdot J_a(L,L)= S_1 +S_2,
\end{align*}
where $J_\alpha(L,L)$ is given by
\begin{align*}
 J_\alpha(L,L)  &= (D^{\mu}L^{\nu} D_{\mu}D_4 F_{\nu \alpha} + D^{\mu}L^{\nu} D_4 D_{\mu} F_{\nu \alpha}) + D_4 D^{\mu}L^{\nu}D_{\mu}F_{\nu\alpha} \notag\\
&\quad +W_{4}{}^{\mu}{}_{\alpha}{}^{\nu}D_4 F_{\mu\nu} + D_4 W_{4}{}^{\mu}{}_{\alpha}{}^{\nu} F_{\mu\nu}.\notag
\end{align*}
The first two terms in the parentheses and the last term can be bounded in the same way as $T_1$ by bounding $\Upsilon$ in $L^\infty_{(sc)}$. Thus, by ignoring those terms, we have
\begin{equation*}
 J_\alpha(L,L)  = D_4 D^{\mu}L^{\nu}D_{\mu}F_{\nu\alpha} \notag +W_{4}{}^{\mu}{}_{\alpha}{}^{\nu}D_4 F_{\mu\nu}.
\end{equation*}

For $S_1$, since $ W_{4}{}^{\mu}{}_{4}{}^{\nu}$ is symmetric in $\mu$ and $\nu$, thus $W_{4}{}^{\mu}{}_{4}{}^{\nu}D_4 F_{\mu\nu}$= 0. Thus, we derive
\begin{align*}
 |S_{1}| &\lesssim  C\delta^{-\frac{1}{2}}+\delta^{3} \doubleint_{\D(u,\ub)}|D_4 D_4 F_{34} | | D_4 D^{\mu}L^{a}||D_{\mu}F_{a 4}|\\
&\lesssim C\delta^{-\frac{1}{2}} + \delta^{\frac{1}{2}} \int_{0}^u \|D_4 D_4 F_{34}\|_{L^2_{(sc)}(H_{u'})}\|D_4 D^{\mu}L^{a}\|_{L^4_{(sc)}(H_{u'})} \| D_{\mu}F_{a 4}\|_{L^4_{(sc)}(H_{u'})}\\
&\lesssim   C\delta^{-\frac{1}{2}}+ \delta^{\frac{1}{2}} \int_{0}^u C\delta^{-\frac{1}{2}}\cdot C\delta^{-\frac{1}{4}}\cdot C\delta^{-\frac{1}{4}} \lesssim C\delta^{-\frac{1}{2}}.
\end{align*}

For $S_2$, we have $S_2 =C\delta^{-\frac{1}{2}}+ \delta^{3} \doubleint_{\D(u,\ub)} \nabla_4 \nabla_4 \alpha_F ( D_4 D^{\mu}L^{\nu}D_{\mu}F_{\nu a}  + W_{4}{}^{\mu}{}_{a}{}^{\nu}D_4 F_{\mu\nu})$. By $L^4_{(sc)}$ estimates, we can bound $S_{2}$ exactly in the same way as we bound $S_1$. Putting all the estimates together, we have
\begin{align*}
\|\nabla_4 \nabla_4 \alpha_F\|^2_{L^2_{(sc)}(H_u^{(0,\ub)})} &\lesssim  C \delta^{-\frac{1}{2}}+\|\nabla_4 \nabla_4 \alpha_F\|^2_{L^2_{(sc)}(H_0)}.
\end{align*}
Multiplying by $\delta$, we derive the first anomalous estimate of this subsection
\begin{equation}\label{anomalous_D4D4_alpha_F}
 \Ftwo(\nabla_4 \nabla_4 \alpha_F) \lesssim \Izero + C \delta^{\frac{1}{4}}.
\end{equation}
\begin{remark}
As a byproduct, we have also showed
\begin{equation}\label{anomalous_D4D4_rho_sigma_F_Hb}
 \Ftwob(\nabla_4 \nabla_4 \rho_F, \nabla_4 \nabla_4 \sigma_F) \lesssim \Izero + C \delta^{\frac{1}{4}}.
\end{equation}
\end{remark}

We take $X=Y=L$ in \eqref{D_X D_Y of F} and we also take $X = \Lb$ in \eqref{stokes_energy_momentum} to derive
\begin{align*}
\|\nabla_4 \nabla_4 \rho_F,\nabla_4 \nabla_4 \sigma_F)\|^2_{L^2_{(sc)}(H)} &\lesssim \Izero \cdot\delta^{-1} + \delta^{3} \doubleint_{\D(u,\ub)} \!\!\!\!\!{^{(\Lb)}\!\pih \cdot T[D_4 D_4 F] + \Divergence T[D_4 D_4 F](\Lb)}\\
&=\Izero \cdot\delta^{-1} + T_1 + T_2.
\end{align*}

For $T_1$, its integrands can be written as $\psi^{s_1} \cdot \Upsilon(D_3 D_3 F)^{s_2} \cdot \Upsilon(D_3 D_3 F)^{s_3}$.
If $\psi \neq \chib$, we can easily bound those terms by $C\delta^{-\frac{1}{2}}$ as before. When $\chib$ appears, in view of Remark \ref{observation_2}, this term must be
$\tr\chib \cdot(|\nabla_4 \nabla_4 \rho_F|^2+ |\nabla_4 \nabla_4\sigma_F|^2)$. Thanks to Gronwall's inequality,  this term can be absorbed by the left hand side of the above equation.

For $T_2$, we can proceed exactly as before with the following exception: we can no longer ignore $D^{\mu}\Lb^{\nu} D_{\mu}D_4 F_{\nu \alpha} + D^{\mu}\Lb^{\nu} D_4 D_{\mu} F_{\nu \alpha}$ since $D^{\mu}\Lb^{\nu}$ may become $\tr\chib$. Once again, we can use Gronwall's inequality to overcome this difficulty. Thus, we have
\begin{equation}\label{anomalous_D4D4_rho_sigma_F_H}
 \Ftwo(\nabla_4 \nabla_4 \rho_F, \nabla_4 \nabla_4 \sigma_F) \lesssim \Izero + C \delta^{\frac{1}{4}}.
\end{equation}

We move on to other anomalous estimates. We now fix $X=Y=\Lb$ in \eqref{D_X D_Y of F}. We take $X = L$ or $\Lb$ in \eqref{stokes_energy_momentum} then add them together to derive
\begin{align*}
&\|(\nabla_3 \nabla_3 \alphab_F, \nabla_3 \nabla_3 \rho_F, \nabla_3 \nabla_3 \sigma_F)\|^2_{L^2_{(sc)}(\Hb_{\ub}^{(0,u)})}+ \|(\nabla_3 \nabla_3 \rho_F, \nabla_3 \nabla_3 \sigma_F)\|^2_{L^2_{(sc)}(H_{u}^{(0,\ub)})}\\
&\lesssim \Izero \delta^{-1}+  \doubleint_{\D(u,\ub)} \!\!\!\!{^{(N)}\!\pih \cdot T[D_3 D_3 F] + \Divergence T[D_3 D_3 F](N)}= \Izero \delta^{-1}+ T_1 + T_2.
\end{align*}
where $N= L$ or $\Lb$.

For $T_1$, its integrands can be written as $\psi^{s_1} \cdot \Upsilon(D_3 D_3 F)^{s_2} \cdot \Upsilon(D_3 D_3 F)^{s_3}$.  If $\psi \neq \chib$, we can easily bound those terms by $C\delta^{-\frac{1}{2}}$. When $\chib$ appears, as we just did, we could use Gronwall's inequality to remove this term. Thus, by ignoring the $\chib$ terms,
\begin{align*}
T_1 &\lesssim C \delta^{-\frac{1}{2}}.
\end{align*}

For $T_2$, we can also proceed exactly as before to bound it by $ C \delta^{-\frac{1}{2}}$.
Putting all the estimates of this subsection together, we derive
\begin{equation}\label{anomalous_D3D3_alphab_F}
\Ftwo(\nabla_3 \nabla_3 \rho_F, \nabla_3 \nabla_3 \sigma_F)+\Ftwob(\nabla_3 \nabla_3 \alphab_F, \nabla_3 \nabla_3 \rho_F, \nabla_3 \nabla_3 \sigma_F) \lesssim \Izero + C \delta^{\frac{1}{4}}.
\end{equation}

\subsubsection{Energy Estimates on Non-anomalies}
We take $O$ as some angular momentum and $(X,Y) \in \{(L, \Lb), (\Lb,L)\}$ to derive
\begin{align*}
  &\int_{H_{u}^{(0,\ub)}} |\Upsilon(\Lh_O D_Y F)^{(s)}|^2 +\int_{\Hb_{\ub}^{(0,u)}} |\Upsilon(\Lh_O D_Y F)^{(s-\frac{1}{2})}|^2 \lesssim  \int_{H_{0}^{(0,\ub)}} |\Upsilon(\Lh_O D_Y F)^{(s)}|^2  + \\
  & \quad  \doubleint_{\D(u,\ub)} | \Divergence T[\Lh_O D_Y F](X)|  + |{}^{(X)}\pi\cdot T[\Lh_O D_Y F]| = \Izero + T_1 +T_2.
\end{align*}
We take all the possible combination of $X$, $Y$ and $O$ and add the scale invariant version of the above estimates together to derive estimates. We need the following schematic expression \footnote{\, We use $\pi$ to denote the deformation tensor of $O$.}
\begin{align}\label{divergence_formula_two_derivatives}
D^{\mu}\Lh_O D_Y F_{\mu\nu} &= D^{\gamma}Y^{\delta} \cdot \Lh_O D_\gamma F_{\delta \nu} +  D\pi \cdot \Upsilon(D_Y F) \\
& \quad +\Psi(\Lh_O W) \cdot \Upsilon  + \Psi \cdot \Upsilon(\Lh_O F) + \pi \cdot \Upsilon(D_\mu D_Y F) + l.o.t.\notag,
\end{align}
where $l.o.t.$ stands for trilinear terms. They are at least quadratic in $\Upsilon$. When one derives estimates, they are extremely easy to handle since we can use twice $L^{\infty}_{(sc)}$ norms on $\Upsilon$. In this way, we can gain a whole $\delta$ which is good enough to compensate all kinds of anomalies. It is precisely for this reason that we shall omit these terms in sequel. We have a similar formula for $\Lh_O D_Y {}^*\!	F_{\mu\nu}$. By duality, we shall not consider those terms.

We first control $T_1$. In view of \eqref{divergence_formula_two_derivatives}, we claim that, thanks to Gronwall's inequality, we can ignore curvature terms $\Psi(\Lh_O W) \cdot \Upsilon$ and $\Psi \cdot \Upsilon(\Lh_O F)$. In fact, for $\Psi(\Lh_O W) \cdot \Upsilon$, because there is no anomaly in $\Psi(\Lh_O W)$, we use $L^2_{(sc)}$ estimates on $\Psi(\Lh_O W)$ and $L^\infty_{(sc)}$ estimates on $\Upsilon$, it contributes at most $C\delta^{\frac{1}{2}}$ to $T_1$ which is acceptable; for $\Psi \cdot \Upsilon(\Lh_O F)$, although we may encounter anomaly from curvature component $\Psi$, since $\Psi(\Lh_O F)$ is never anomalies, we can use $L^4_{(sc)}$ estimates on both terms, thus, it contributes at most $C\delta^{\frac{1}{4}}$  to $T_1$.

Without loss of generality, we can then rewrite \eqref{divergence_formula_two_derivatives} as
\begin{equation}\label{divergence_formula_two_derivatives_revisted}
D^{\mu}\Lh_O D_Y F_{\mu\nu} = D\pi \cdot \Upsilon(D_Y F)+ D^{\gamma}Y^{\delta} \cdot \Lh_O D_\gamma F_{\delta \nu}  + \pi \cdot \Upsilon(D_\mu D_Y F).	
\end{equation}
and we split $T_1 = T_{11}+T_{12}+T_{13}$ according to \eqref{divergence_formula_two_derivatives_revisted}. We control them one by one.

For $T_{11} = \int\!\!\!\int \Upsilon(\Lh_O D_Y F) \cdot  D\pi \cdot \Upsilon(D_Y F)$, we claim that $T_{11} \lesssim C\delta^{\frac{1}{4}}$. We observe that $\Upsilon(\Lh_O D_Y F)$ is never anomalous. If $D\pi \neq D_3 \pi_{3a}$,  we can use $L^4_{(sc)}$ estimates on both $D\pi$ and $\Upsilon(D_Y F)$ to bound the whole thing by $C \delta^{\frac{1}{4}}$; when $D\pi = D_3 \pi_{3a}$, we use trace estimates,
\begin{align*}
 &\quad \doubleint_{\D(u,\ub )} \Upsilon(\Lh_O D_Y F) \cdot  \nabla_3 Z \cdot \Upsilon(D_Y F)\\
&\lesssim \|\nabla_3 Z\|_{L^2_{(sc)}(S_{u,0})L^{\infty}([0,\ub])} \sup_{\ub}\|\Upsilon(D_Y F)\|_{Tr_{(sc)}(H_u)} \delta^{-\frac{1}{2}}\int_0^{\ub}  \|\Upsilon(\Lh_O D_Y F)\|_{L^2_{(sc)}(\Hb_{\ub'})} \lesssim C\delta^{\frac{1}{4}}.
\end{align*}

For $T_{12} = \int\!\!\!\int \Upsilon(\Lh_O D_Y F) \cdot  D^{\gamma}Y^{\delta} \cdot \Lh_O D_\gamma F_{\delta \nu}$, the worst case happens when $Y =\Lb$ and $\gamma=\delta =a$ (this brings in $\tr\chib_0$); other terms are easily bounded by $C\delta^{\frac{1}{2}}$. Thus, the worst case gives
\begin{equation*}
  \tr\chib_0 \cdot  \sum_{a=1}^{2}\doubleint_{\D(u,\ub )}\Upsilon(\Lh_O D_Y F) \cdot  \Lh_O D_a F_{a \nu}.	
\end{equation*}
Thus, according to Maxwell equation $D^{\mu}F_{\mu\nu} =0$, we replace $D_a F_{a \nu}$ by $D_4 F_{3\nu}$ and $D_3 F_{4\nu}$. Then, thanks to Gronwall's inequality, this term can be absorbed by the left hand side.

For $T_{13} = \int\!\!\!\int  \Upsilon(\Lh_O D_Y F) \cdot  \pi \cdot \Upsilon(D_\mu D_Y F)$. A careful computation shows that the integrand is of the form
\begin{equation*}
 \Upsilon(\Lh_O D_Y F) \cdot \pi^{\gamma\delta} \cdot D_{\gamma} D_Y F_{\delta \mu}.
\end{equation*}
with $\nu = 3$ or $4$. If $D_{\gamma} D_Y F_{\delta \mu}$ is not anomalous, we can easily bound these terms by $C \delta^{\frac{1}{2}}$. Thus, we only concentrate on anomalies. According to Lemma \ref{comparison_maxwell_2}, \ref{comparison_maxwell_2_D_a_D_null}, \ref{two more 1} and \ref{two more 2}, we essentially have two such terms,
\begin{equation}\label{lalala}
D_4 D_4 F_{a4} , \quad D_3 D_3 F_{a3}.
\end{equation}
Of course, there are other anomalous terms. But their anomalies are from lower order terms which allow us to use $L^4_{(sc)}$ estimates. To illustrate the idea, in view of Lemma \ref{comparison_maxwell_2_D_a_D_null}, we pick $\sigma(D_a D_3 F)$ as one example. The anomaly comes from $\alpha_F$. We can control it by
\begin{align*}
|\doubleint_{\D(u,\ub)}\Upsilon(\Lh_O D_Y F) \cdot \pi\cdot  \alpha_F| \lesssim |\int_{0}^u \|\Upsilon(\Lh_O D_Y F)\|_{L^2_{(sc)}(H_u)} \|\pi\|_{L^4_{(sc)}(H_u)} \|\alpha_F\|_{L^4_{(sc)}(H_u)} \lesssim C \delta^{\frac{1}{4}}.
\end{align*}

We turn to the terms in \eqref{lalala}. Since $\pi^{3a} = 0$, we can forget $D_3 D_3 F_{a3}$. Thus, it remains to control a bulk integral with integrand
\begin{equation*}
 \Upsilon(\Lh_O D_Y F) \cdot \pi^{4a} \cdot D_{4} D_Y F_{a 4} \sim \Upsilon(\Lh_O D_4 F) \cdot Z \cdot \nabla_{4} \nabla_4 \alpha_F.
\end{equation*}

The idea is to move bad derivative $\nabla_4$ to good component to avoid the anomalies. As we noticed in $T_{11}$, $\Upsilon \neq \alpha$. First of all, we have
\begin{align*}
\doubleint\nabla_O \nabla_4 \Upsilon \cdot Z \cdot \nabla_{4} \nabla_4 \alpha_F = -\doubleint\nabla_4 \Upsilon \cdot \nabla_O Z \cdot \nabla_{4} \nabla_4 \alpha_F -\doubleint\nabla_4 \Upsilon \cdot  Z \cdot \nabla_O \nabla_{4} \nabla_4 \alpha_F.
\end{align*}

We bound the first integral by $\Izero + C\delta^{\frac{1}{4}}$: we use $L^4_{(sc)}$ norm on $\nabla_4 \Upsilon$ and $\nabla Z$; we use $L^2_{(sc)}$ norm on $\nabla_4 \nabla_4 \alpha_F$ to trace the anomaly back to initial data.

To control the second one, we integrate by parts again. Notice that we can replace $\nabla_O \nabla_{4} \nabla_4 \alpha_F$ by $\nabla_{4} \nabla_O \nabla_4 \alpha_F$ since the commutator enjoys better estimates then can be ignored. Thus,
\begin{align*}
\doubleint \nabla_4 \Upsilon \cdot  Z &\cdot \nabla_4 \nabla_{O} \nabla_4 \alpha_F =\int_{H_u} \nabla_4 \Upsilon \cdot  Z \cdot \nabla_{O} \nabla_4 \alpha_F - \int_{H_0} \nabla_4 \Upsilon \cdot  Z \cdot \nabla_{O} \nabla_4 \alpha_F\\
&-\doubleint\nabla_4 \nabla_4 \Upsilon \cdot  Z \cdot \nabla_{O} \nabla_4 \alpha_F-\doubleint\nabla_4 \Upsilon \cdot  \nabla_4 Z \cdot \nabla_{O} \nabla_4 \alpha_F
\end{align*}
In this form, all of the anomalies can be traced back to initial data. Thus, we have control on $T_{13}$ as well as $T_1$.

The control of $T_2$ is much easier. In fact, according to the form of the integrand, we use Gronwall's inequality and anomalous estimates derived so far to show  $|T_{2}| \lesssim \Izero + C \delta^{\frac{1}{4}}$.

We collect all those estimates, together with the anomalous ones, we finally conclude
\begin{equation}\label{energy_estimates_Ftwo}
 \Ftwo + \Ftwob \lesssim \Izero + C\delta^{\frac{1}{8}}.
\end{equation}

\subsection{Energy Estimates on One Derivative of Curvature}

\subsubsection{Preliminaries}
We use $\R_u$ and $\Rb_{\ub}$ to denote the restriction of $\R$ and $\Rb$ to the interval $[0,u]$ and $[0,\ub]$ respectively. We recall following lemmas from Section 15.2 of \cite{K-R-09}.
\begin{lemma} \label{comparison_D_curvature} Let $H=H_u^{(0,\ub)}$ and $\Hb = \Hb_{\ub}^{(0,u)}$. If $\delta$ is sufficiently small, then
\begin{align*}
 &\delta^{\frac{1}{2}}\|\alpha(D_3 W)\|_{L^2_{(sc)}(H)} + \delta^{\frac{1}{2}}\|\beta(D_a W)\|_{L^2_{(sc)}(H)} \lesssim \Izero + \delta^{\frac{1}{4}}C,\\
 \|\alpha(D_a W)\|_{L^2_{(sc)}(H)} &+ \|(\beta(D_3 W, D_4 W))\|_{L^2_{(sc)}(H)} + \|(\betab(D_4 W,D_a W))\|_{L^2_{(sc)}(H)}\\
&\quad +\|(\rho,\sigma)(D_4 W, D_3 W, D_a W)\|_{L^2_{(sc)}(H)} +\|\alphab(D_4 W)\|_{L^2_{(sc)}(H)} \lesssim \R_u + \delta^{\frac{1}{4}}C,\\
\|(\rho, \sigma)(D_4 W, D_3 W, &D_a W)\|_{L^2_{(sc)}(\Hb)} +\|\betab(D_4 W, D_3 W, D_a W)\|_{L^2_{(sc)}(\Hb)}\\
&\quad +\|\beta(D_3 W)\|_{L^2_{(sc)}(\Hb)}+\|\alphab(D_4 W, D_a W)\|_{L^2_{(sc)}(\Hb)}\lesssim \Rb_{\ub} + \delta^{\frac{1}{4}}C.
\end{align*}
\end{lemma}

\begin{lemma}\label{comparison_lemma_2} Let $H=H_u^{(0,\ub)}$, $\Hb = \Hb_{\ub}^{(0,u)}$ and $O$ is an angular momentum. Let $1\leq s_1 \leq \frac{5}{2}$, $1\leq s_2 \leq \frac{3}{2}$, $ s_3 \leq \frac{1}{2}$ and $\frac{1}{2} \leq s_4 \leq 2$. If $\delta$ is sufficiently small, we have
\begin{align*}
 &\|\alpha(\Lh_L R) -\nabla_L \alpha\|_{L^2_{(sc)}(H_u^{(0,\ub)})} \lesssim C, \quad \delta^{\frac{1}{2}}\|\alpha(\Lh_{\Lb} R) -\nabla_{\Lb} \alpha\|_{L^2_{(sc)}(H_u^{(0,\ub)})} \lesssim \Rzero + C \delta^{\frac{3}{4}},\\
 &\|(\Psi^{s_1}(\Lh_L R) -(\nabla_L \Psi)^{s_1},\Psi^{s_1}(\Lh_O R) -(\nabla_O \Psi)^{s_1})\|_{L^2_{(sc)}(H)}+ \|\Psi^{s_4}(\Lh_O R) -(\nabla_O \Psi)^{s_4}\|_{L^2_{(sc)}(\Hb)}\lesssim C \delta^{\frac{1}{4}}, \\
 &\|\Psi^{s_2}(\Lh_{\Lb} R) -(\nabla_{\Lb} \Psi)^{s_2}\|_{L^2_{(sc)}(H)} + \|\Psi^{s_3}(\Lh_{\Lb} R) -(\nabla_{\Lb} \Psi)^{s_3}\|_{L^2_{(sc)}(\Hb)} \lesssim \Rzero + C\delta^{\frac{1}{4}}.
\end{align*}
\end{lemma}

\subsubsection{Energy Estimates on Anomalies}

We take $O = L$ and $X = Y = Z = L$ in \eqref{stokes_Bel_Robinson} to derive
\begin{align*}
 \int_{H_u^{(0,\ub)}} |\alpha(\Lh_L W)|^2 &\lesssim \int_{H_0^{(0,\ub)}} |\alpha(\Lh_L W)|^2 + |\doubleint_{\D(u,\ub)} \Divergence Q[\Lh_L W]_{444}  + (\pih \cdot Q[\Lh_L W])_{444}|
\end{align*}
We classify the terms in the bulk integral into five types of expressions with $s_1 + s_2 + s_3 = 6$,
\begin{align*}
(I)\, \psi^{s_1}\cdot\Psi^{s_2}[\Lh_4 W]\cdot\Psi^{s_3}[\Lh_4 W], \, (II)\, \psi^{s_1}\cdot &\Psi^{s_2}[\Lh_4 W]\cdot (D\Psi)^{s_3},\,(III)(D\,^{(4)}\!\pih)^{s_1}\cdot\Psi^{s_2}[\Lh_4 W]\cdot\Psi^{s_3}, \\
(IV)\, \Upsilon^{s_1}\cdot (D_L D\Upsilon)^{s_2}\cdot \Psi^{s_3}[\Lh_4 W],\quad &(V) (D_L \Upsilon)^{s_1}\cdot (D\Upsilon)^{s_2}\cdot \Psi^{s_3}[\Lh_4 W].
\end{align*}

For type $(I)$, $(II)$ and $(IV)$ terms, we can use $L^\infty_{(sc)}$ norm on $\psi$ or $\Upsilon$. Thus, it is easy to bound all of them by $C\delta^{-\frac{1}{2}}$.


For type $(V)$ terms, we can proceed as follows,
\begin{align*}
 (V)  \lesssim   \delta^{\frac{1}{2}} \int_{0}^{u} \|\Psi(\Lh_L W)^{s_2}\|_{L^2_{(sc)}(H_{u'}^{(0,\ub)})}  \|D \Upsilon\|_{L^4_{(sc)}(H_{u'}^{(0,\ub)})} \|D \Upsilon\|_{L^4_{(sc)}(H_{u'}^{(0,\ub)})} \lesssim C \delta^{-\frac{1}{2}}.
\end{align*}

For type $(III)$ terms, we rewrite $(D\,^{(4)}\!\pih)^{s_1}$ as
\begin{align*}
 (D\,^{(4)}\!\pih)^{s_1} = (\Dslash \,^{(4)}\!\pih)^{s_1} + \tr\chib_0 \cdot \psi^{s_1} + \psi\cdot \psi,
\end{align*}
where $\Dslash \,^{(4)}\!\pih$ is the tangential part of $D\,^{(4)}\!\pih$. The last two terms can be ignored since they enjoy better estimates. We observe that $(\Dslash \,^{(4)}\!\pih)^{s_1}$ can not be $\nabla_4 \omega$ or $\nabla_3 \omegab$ (which we do not have estimates). In fact, since $\omega$ does not appear as a component of $\,^{(4)}\!\pih$, so $\nabla_4 \omega$ can not occur; we can also rule out $\nabla_3 \omegab$ by signature considerations. Thus,
\begin{align*}
 (III) &\lesssim  \delta^{\frac{1}{2}} \int_{0}^{u} \|\Psi(\Lh_L W)^{s_2}\|_{L^2_{(sc)}(H_{u'}^{(0,\ub)})}  \|(\Dslash\,^{(4)}\!\pih)^{s_1}\|_{L^4_{(sc)}(H_{u'}^{(0,\ub)})}\|\Psi^{s_3}\|_{L^4_{(sc)}(H_{u'}^{(0,\ub)})}\lesssim  C\delta^{-\frac{1}{2}}.
\end{align*}

Putting all these estimates together, we have $ \|\alpha(\Lh_L W)^{s_2}\|^2_{L^2_{(sc)}(H_{u}^{(0,\ub)})} \lesssim \delta^{-1}\Izero  + C \delta^{-\frac{1}{2}}$.

In view of Lemma \ref{comparison_lemma_2}, we conclude
\begin{equation*}
\delta^{\frac12}\|\nabla_4 \alpha\|_{L^2_{(sc)}(H_u^{(0,\ub)})} \lesssim \Izero + C\delta^{\frac{1}{8}}.
\end{equation*}

We then take $O = \Lb$ and $X = Y = Z = \Lb$ in \eqref{stokes_Bel_Robinson} to derive
\begin{align*}
 \int_{\Hb_{\ub}^{(0,u)}} |\alphab(\Lh_{\Lb} W)|^2 &\lesssim  \delta^{-1}\Izero + |\doubleint_{\D(u,\ub)} \Divergence Q[\Lh_{\Lb} W]_{333}  + (\,^{(3)}\!\pih \cdot Q[\Lh_{\Lb} W])_{333}|.
\end{align*}
We classify the terms in the bulk integral into five types of expressions with $s_1 + s_2 + s_3 = 1$,
\begin{align*}
(I)\, \psi^{s_1}\cdot\Psi^{s_2}[\Lh_3 W]\cdot\Psi^{s_3}[\Lh_3 W], \, (II)\, \psi^{s_1}\cdot&\Psi^{s_2}[\Lh_3 W]\cdot (D\Psi)^{s_3},\,(III)\,(D\,^{(3)}\!\pih)^{s_1}\cdot\Psi^{s_2}[\Lh_3 W]\cdot\Psi^{s_3}, \\
(IV)\,\Upsilon^{s_1}\cdot (D_3 D\Upsilon)^{s_2}\cdot \Psi^{s_3}[\Lh_3 W],\quad &(V)\, (D_3 \Upsilon)^{s_1}\cdot (D\Upsilon)^{s_2}\cdot \Psi^{s_3}[\Lh_3 W].
\end{align*}

For type $(I)$ terms, if $\psi^{s_1} \neq \tr\chib$, we can still use $L^\infty_{(sc)}$ norm on $\psi$ or $\Upsilon$ to bound them by $C\delta^{-\frac{1}{2}}$.
We now estimate the terms of type $(I)$. If $\tr\chib$ appears in $(I)$, it must come from $\,^{(3)}\!\pi^{ab} \cdot Q[\Lh_{\Lb} W])_{ab33} = \tr\chib |\betab(\Lh_{\Lb} W)|^2 $.  Thus, we can estimate $(I)$ as follows
\begin{align*}
 (I) & \lesssim C\delta^{-\frac{1}{2}} + \delta^{-1}\int_{0}^{\ub}  \|\betab(\Lh_{\Lb} W)\|^2_{L^2_{(sc)}(\Hb_{\ub'}^{(0,u)})}.
\end{align*}
In view Lemma \ref{comparison_lemma_2} and \eqref{NBE_Lb_betab},
\begin{align*}
  \|\betab(\Lh_{\Lb} W)\|_{L^2_{(sc)}(\Hb_{\ub}^{(0,u)})}&\lesssim \|\nabla_3\betab \|_{L^2_{(sc)}(\Hb_{\ub}^{(0,u)})} + C\\
&\lesssim \|\divergence \alphab+\tr \chib_0 \cdot \betab  + \psi\cdot\Psi + \Upsilon\cdot \nabla \Upsilon-D_3 R_{3b} \|_{L^2_{(sc)}(\Hb_{\ub}^{(0,u)})} + C\lesssim C.
\end{align*}
Hence, $ (I) \lesssim C\delta^{-\frac{1}{2}}$.

For type $(II)$ terms, since $s_3 \leq 1$,  we can integrate $(D\Psi)^{s_3}$ along $\Hb_{\ub}$. Thus,
\begin{align*}
 (II) & \lesssim  \delta^{-1}\int_{0}^{\ub} \|\Psi[\Lh_3 W]\|_{L^2_{(sc)}(\Hb_{\ub'}^{(0,u)})} \|(D\Psi)^{s_3}\|_{L^2_{(sc)}(\Hb_{\ub'}^{(0,u)})} \\
&\lesssim  \delta^{-1}\int_{0}^{\ub} \|\alphab(\Lh_{\Lb} W)\|^2_{L^2_{(sc)}(\Hb_{\ub'}^{(0,u)})}  +  \delta^{-1}\int_{0}^{\ub}  |\Rb(u,\ub')|^2 + C\delta^{\frac{1}{4}}.
\end{align*}

For type $(III)$, $(IV)$ and $V$ terms, they can be treated exactly as before (the $(III)$ terms in for $X=Y=Z=O=L$), thus $(III)+(IV)+(V)  \lesssim C \delta^{-\frac{1}{2}}$.

Putting all estimates together, after a standard use of Gronwall's inequality, we derive
\begin{equation*}
 \delta^{\frac{1}{2}}\|\nabla_3 \alphab\|_{L^2_{(sc)}(\Hb_{\ub}^{(0,u)})} \lesssim \Izero + C \delta^{\frac{1}{8}}.
\end{equation*}

\subsubsection{Energy Estimates on Non-anomalies}
We take an angular momentum $O$ and $X , Y, Z \in \{L,\Lb\}$ in \eqref{stokes_Bel_Robinson} to derive
\begin{align*}
  & \quad \int_{H_{u}^{(0,\ub)}} |\Psi^s(\Lh_O W)|^2 +\int_{\Hb_{\ub}^{(0,u)}} |\Psi^{s-\frac{1}{2}}(\Lh_O W)|^2 \\
& \lesssim  \int_{H_{0}^{(0,\ub)}} |\Psi^s(\Lh_O W)|^2 + |\doubleint_{\D(u,\ub)} \Divergence Q[\Lh_{O} W](X,Y,Z)  + (\pih \cdot Q[\Lh_{O} W])(X,Y,Z)|
\end{align*}
As before, we sum all the possible choice for $O$, $X$, $Y$ and $Z$. The integrands of the bulk integral can be classified into five types as before,
\begin{align*}
(I)\,\psi^{s_1} \Psi^{s_2}[\Lh_O W] \Psi^{s_3}[\Lh_O W], \, (II)\, (\,^{(O)}\!\pih)^{s_1}& \Psi^{s_2}[\Lh_O W] (D\Psi)^{s_3},\,(III)\, (D\,^{(O)}\!\pih)^{s_1} \Psi^{s_2}[\Lh_O W] \Psi^{s_3},\\
(IV)\, \Upsilon^{s_1} (D_O D\Upsilon)^{s_2} \Psi^{s_3}[\Lh_O W],\quad& (V)\,(D_O \Upsilon)^{s_1} (D\Upsilon)^{s_2}  \Psi^{s_3}[\Lh_O W],
\end{align*}
where $s_1 + s_2 + s_3 = 2s \geq 2$. We shall bound them one by one.

For type $(I)$ terms,  if $\psi^{s_1} \neq \tr\chib$, those terms are obvious bounded by $C\delta^{\frac{1}{2}}$. For the worst $\psi^{s_1} = \tr\chib$, we can bound those terms by
\begin{equation*}
\sum_{s\geq 1} \int_0^{u}  \|\Psi^s(\Lh_O W)\|^2_{L^2_{(sc)}(H_{u'}^{(0,\ub)})} + \sum_{s\leq 2} \delta^{-1} \int_0^{\ub}  \|\Psi^s(\Lh_O W)\|^2_{L^2_{(sc)}(\Hb_{\ub'}^{(0,u)})}.
\end{equation*}
Thanks to Gronwall's inequality, this quantity can be absorbed by the left hand side. Hence, we can say that $(I) \lesssim C \delta^{\frac{1}{2}}$.

For type $(II)$ terms, if $(D\Phi)^{s_3}$ is not anomalous, those terms are obvious bounded by $C\delta^{\frac{1}{2}}$. When $(D\Phi)^{s_3}$ is anomalous, we observe that $D\Phi  \in \{\alpha(D_3 W), \beta(D_a W), \alpha(D_4 W), \alphab(D_3 W)\}$. We bound these four possibilities one by one.

When $D\Phi ^{s_3} = \alpha(D_3 W)$, according to \eqref{NBE_Lb_alpha}, we have $\alpha(D_3 W) = -\frac{1}{2}\tr\chib_0 \cdot \alpha + E_1$ where $E_1$ is not anomalous, i.e. we have
$\|E_1\|_{L^2_{(sc)}(H_{u}^{(0,\ub)})} \lesssim C$ or $\|E_1\|_{L^2_{(sc)}(\Hb_{\ub}^{(0,u)})} \lesssim C$. Thus, we can regard $D\Phi ^{s_3} = \alpha(D_3 W)$ as $\alpha$. We then use $L^4_{sc}$ norms on $\alpha$ and $(\,^{(O)}\!\pih)^{s_1}$ so that the whole thing is bounded by $C\delta^{\frac{1}{4}}$.

When $D\Phi ^{s_3} = \beta(D_a W)$, we can proceed exactly before to locate the anomaly of $D\Phi ^{s_3} = \beta(D_a W)$ to $\alpha$. Thus, those terms are also bounded by $C\delta^{\frac{1}{4}}$.

When $D\Phi ^{s_3} = \alphab(D_3 W)$,  according to $\alphab(D_3 R) = \nabla_3 \alpha + 4\omegab\cdot \alphab$, we can replace $\alphab(D_3 W)$ by $\nabla_3 \alphab$ and ignore other terms (which are bounded by $C\delta^{\frac{1}{4}}$); similarly, we replace $\Psi^{s_2}[\Lh_O W]$ by $(\nabla_O \Psi)^{s_2}$. Thus, it suffices to bound the following integral,
\begin{align*}
 \doubleint_{D(u,\ub)}(\,^{(O)}\!\pih)^{s_1}\cdot&( \nabla_O \Psi)^{s_2} \cdot \nabla_3 \alphab = \int_{H_u}(\,^{(O)}\!\pih)^{s_1}\cdot (\nabla_O \Psi)^{s_2} \cdot \alphab -\int_{H_0}(\,^{(O)}\!\pih)^{s_1}\cdot(\nabla_O \Psi)^{s_2} \cdot \alphab\\
& -\doubleint_{D(u,\ub)} (\nabla_3 \,^{(O)}\!\pih)^{s_1}\cdot(\nabla_O \Psi)^{s_2}\cdot  \alphab -\doubleint_{D(u,\ub)}(\,^{(O)}\!\pih)^{s_1}\cdot\nabla_3 (\nabla_O \Psi)^{s_2}\cdot \alphab.
\end{align*}

We ignore two boundary terms since they are by $C\delta^{\frac{1}{2}}$ in an obvious way.

For first bulk integral, if $(\nabla_3 \,^{(O)}\!\pih)^{s_1} \neq (D_3 \,^{(O)}\!\pih)_{3a}$, we use $L^4_{(sc)}$ norms on $\alphab$ and $(\nabla_3 \,^{(O)}\!\pih)^{s_1}$ to bound the whole thing by $C\delta^{\frac{1}{2}}$. If $(\nabla_3 \,^{(O)}\!\pih)^{s_1} = \nabla_3 Z$, we use trace estimates,
\begin{align*}
|\doubleint\nabla_3 Z \cdot(\nabla_O \Psi)^{s_2} \cdot\alphab |\lesssim \|\nabla_3 Z\|_{L^2_{(sc)}(S_{u,0})L^{\infty}[0,\ub]} \sup_{\ub}\|\alphab\|_{Tr_{(sc)}(H_u)} \delta^{-\frac{1}{2}}\int_0^{\ub}  \|(\nabla_O \Psi)^{s_2}\|_{L^2_{(sc)}(\Hb_{\ub'})}.
\end{align*}
By signature considerations, we know that $(\nabla_O \Psi)^{s_2}$ can be integrated along $\Hb_{\ub}$. In fact,  since $s_2 =2s \in \mathbb{Z}$ and $s_2 < 3$, we know $s_2 \leq 2$. Thus, we can bound first bulk integral by $C\delta^{\frac{1}{2}}$.

For second bulk integral term, we replace $\nabla_3 (\nabla_O \Psi)^{s_2}$ by $\nabla_O [(\nabla_3 \Psi)^{s_2-\frac{1}{2}}]$ and we ignore the commutator which can be bounded by $C\delta^{\frac{1}{2}}$ in an obvious way. We integrate by parts again,
\begin{align*}
\doubleint_{D(u,\ub)} & (\,^{(O)}\!\pih)^{s_1} \cdot \nabla_O [(\nabla_3 \Psi)^{s_2-\frac{1}{2}}] \cdot \alphab =-\doubleint_{D(u,\ub)}\nabla_O [(\,^{(O)}\!\pih)^{s_1}] \cdot (\nabla_3 \Psi)^{s_2-\frac{1}{2}} \cdot \alphab\\
  &-\doubleint_{D(u,\ub)}(\,^{(O)}\!\pih)^{s_1} \cdot (\nabla_3 \Psi)^{s_2-\frac{1}{2}} \cdot \nabla_O \alphab -\doubleint_{D(u,\ub)}(\,^{(O)}\!\pih)^{s_1} \cdot (\nabla_3 \Psi)^{s_2-\frac{1}{2}} \cdot (\nabla^a O_a)\alphab
\end{align*}
By signature considerations, $\Psi \neq \alphab$, thus $\nabla_3 \Psi$ is not anomalous. We can easily bound these three integrals. Thus, second bulk integral can be bounded by $C\delta{^\frac{1}{2}}$. We can also conclude that when $D\Phi ^{s_3} = \alphab(D_3 W)$, all the terms are bounded by $C\delta^{\frac{1}{4}}$.

When $D\Phi ^{s_3} = \alpha(D_4 W)$, similarly, we can bound those terms by $C\delta^{\frac{1}{4}}$. The details can be found in Section 15.10 of \cite{K-R-09}. Thus, Hence, we can say that $(II) \lesssim C \delta^{\frac{1}{4}}$.

For type $(III)$ terms, if $(D\,^{(O)}\!\pi)^{s_1}\neq D_3\,^{(O)}\!\pi_{3a}$, we use $L^4_{(sc)}$ norms on $(D\,^{(O)}\!\pi)^{s_1}$ and $\Psi^{s_3}$ to bound those terms by $C\delta^{\frac{1}{4}}$. If $(D\,^{(O)}\!\pi)^{s_1}\neq D_3\,^{(O)}\!\pi_{3a}$, we replace $D_3\,^{(O)}\!\pi_{3a}$ by $\nabla_3 Z$ and bound it as follows,
\begin{align*}
\doubleint\nabla_3 Z\cdot\Psi^{s_2}[\Lh_O W]\cdot  \Psi^{s_3}\lesssim\|\nabla_3 Z\|_{L^2_{(sc)}(S_{u,0})L^{\infty}[0,\ub]} \sup_{\ub}\|\Psi^{s_3}\|_{Tr_{(sc)}} \delta^{-\frac{1}{2}}\int_0^{\ub}  \|(\nabla_O \Psi)^{s_2}\|_{L^2_{(sc)}(\Hb_{\ub'})}.
\end{align*}
Those bounds yield $(III) \lesssim C\delta{^\frac{1}{4}}$.

For type $(IV)$ terms, we bound them as follows,
\begin{align*}
(IV) 
&\lesssim \delta^{-1}\int_{0}^{\ub} \|\Psi(\Lh_O W)^{s_2}\|_{L^2_{(sc)}(\Hb_{\ub'}^{(0,u)})} \delta^{\frac{1}{2}}\|\Upsilon^{s_1}\|_{L^{\infty}_{(sc)}(\Hb_{\ub'}^{(0,u)})}\|D_O D\Upsilon^{s_2}\|_{L^2_{(sc)}(\Hb_{\ub'}^{(0,u)})} \lesssim C \delta^{\frac{1}{4}}.
\end{align*}

For type $(V)$ terms, we bound them as follows,
\begin{align*}
 (V) 
& \lesssim \delta^{-1}\int_{0}^{\ub} \|\Psi(\Lh_O W)^{s_2}\|_{L^2_{(sc)}(\Hb_{\ub'}^{(0,u)})}  \delta^{\frac{1}{2}}\| (D\Upsilon)^{s_2}\|_{L^4_{(sc)}(\Hb_{\ub'}^{(0,u)})} \| (\nabla_O \Upsilon)^{s_2}\|_{L^4_{(sc)}(\Hb_{\ub'}^{(0,u)})}\lesssim C\delta^{\frac{1}{4}}.
\end{align*}

Putting everything together, we summarize the estimates in this subsection in the following inequalities:
\begin{align*}
 \|\nabla \alpha \|_{L^2_{(sc)}(H_{u}^{(0,\ub)})} +\|\nabla \beta \|_{L^2_{(sc)}(\Hb_{\ub}^{(0,u)})} &\lesssim \|\nabla \alpha \|_{L^2_{(sc)}(H_{0}^{(0,\ub)})} + C \delta^{\frac{1}{4}},\\
\sum_{s\geq 2}\|(\nabla \Psi)^{s}\|_{L^2_{(sc)}(H_{u}^{(0,\ub)})} +\|(\nabla \Psi)^{s-\frac{1}{2}} \|_{L^2_{(sc)}(\Hb_{\ub}^{(0,u)})} &\lesssim \sum_{s\geq 2}\|(\nabla \Psi)^{s}\|_{L^2_{(sc)}(H_{0}^{(0,\ub)})} + C \delta^{\frac{1}{4}}.
\end{align*}

Combining all the energy estimates, for sufficiently small $\delta$, we establish {\bf{Theorem C}}. We remark that the constant $C$ from now on will depend only on the size of initial data $\Izero$.


\section{Formation of Trapped Surfaces}
Based estimates derived in previous sections, we demonstrate how a trapped surface forms. More precisely, we show that $S_{1,\delta}$ is trapped. The following equations are responsible for the formation,
\begin{equation}\label{Final_Equation_1}
 \nabla_4 \tr \chi + \frac{1}{2}(\tr \chi)^2 = -|\chih|^2-2\omega \cdot \tr \chi - |\alpha_F|^2,
\end{equation}
\begin{equation}\label{Final_Equation_2}
 \nabla_3 \chih + \frac{1}{2} \tr \chib \cdot\chih = \nabla \tensor \eta +2\omegab \cdot \chih -\frac{1}{2}\tr \chi \cdot \chibh +\eta \tensor \eta + \hat{T}_{ab} = E_1,
\end{equation}
\begin{align}\label{Final_Equation_3}
\nabla_3 {\alpha_F}  + \frac{1}{2}\tr\chib \cdot \alpha_F &= -\nabla {\rho_F} + ^*\!\nabla {\sigma_F}-2\,^*\!\etab \cdot {\sigma_F} + 2\etab \cdot {\rho_F} + 2\omegab {\alpha_F}-\chih \cdot {\alphab_F} = E_2.
\end{align}
We also recall that $r = \ub-u+ r_0$ and $\tr\chib_0 = -\frac{2}{r} = -\frac{2}{\ub-u+ r_0}$. We assume  where $r_0 \sim 10$.

\subsection{Formation Mechanism}
We rewrite \eqref{Final_Equation_1} as $\frac{d}{d\ub} \tr \chi  = -\frac{1}{2\Omega} (\tr\chi)^2-\frac{1}{\Omega}(|\chih|^2+ |\alpha_F|^2) \leq -\frac{1}{\Omega}(|\chih|^2+ |\alpha_F|^2)$, therefore,
\begin{equation}\label{Equation_a}
 \tr\chi(u, \delta) \leq \tr\chi(u,0) - \int_0^\delta \frac{1}{\Omega}(|\chih|^2+ |\alpha_F|^2)(u,\ub) d\ub
\end{equation}
In view of \eqref{Final_Equation_3}, we have $\nabla_3 (r^2 |\alpha_F|^2)=r^2|\alpha_F|^2[-\trchibt + \frac{2}{r}(\Omega-1)]+r^2 (E_2 \cdot \alpha_F)$, thus,
\begin{align*}
 \frac{d}{du} (r^2 |\alpha_F|^2) &=\frac{r^2|\alpha_F|^2}{\Omega}[-\trchibt + \frac{2}{r}(\Omega-1)]+\frac{r^2 (E_2\cdot \alpha_F)}{\Omega}=F_2,
\end{align*}
which implies
\begin{equation*}
 |\alpha_F|^2(u,\ub) = \frac{r(0,\ub)^2}{r(u,\ub)^2}|\alpha_F(0,\ub)|^2 +  \frac{r(0,\ub)^2}{r(u,\ub)^2}\int_0^{u}F_2(u',\ub)=\frac{(\ub+r_0)^2}{(\ub - u + r_0)^2}|\alpha_F(0,\ub)|^2 +G_2,
\end{equation*}
with $G_2 = \frac{r(0,\ub)^2}{r(u,\ub)^2}\int_0^{u}F_2(u',\ub)$. Similarly, we have
\begin{equation*}
 |\chih|^2(u,\ub) =\frac{(\ub+r_0)^2}{(\ub - u + r_0)^2}|\chih(0,\ub)|^2 +G_1,
\end{equation*}
with $ G_1 = \frac{r(0,\ub)^2}{r(u,\ub)^2}\int_0^{u}(\frac{r^2|\chih|^2}{\Omega}[-\trchibt + \frac{2}{r}(\Omega-1)]+\frac{r^2 (E_1\cdot \chih)}{\Omega})(u',\ub)$.
Base on the following lemma whose proof is deferred to next subsection, we can ignore two error terms $G_1$ and $G_2$.
\begin{lemma}\label{Smallness_Lemma}If $\delta$ is sufficiently small, then
\begin{equation}\label{Smallness_Estimates}
 \|\int_{0}^{\delta} G_1(u,\ub') d\ub'\|_{L_{u,\theta}^{\infty}}+\|\int_{0}^{\delta} G_2(u,\ub') d\ub'\|_{L_{u,\theta}^{\infty}}\leq C(\Izero) \delta^{\frac{1}{2}} .
\end{equation}
\end{lemma}
The proof is deferred to the next subsection. Back to equation \eqref{Equation_a}, we have
\begin{align*}
 \tr\chi(u, \delta) &\leq \tr\chi(u,0) - \int_0^\delta \frac{(\ub+r_0)^2}{\Omega(\ub - u + r_0)^2}(|\chih(0,\ub)|^2 + |\alpha_F(0,\ub)|^2) d\ub + C\delta^{\frac{1}{2}}\\
&\leq \tr\chi(u,0) - \frac{r_0^2}{(r_0-u)^2}\int_0^\delta \frac{1}{\Omega(u,\ub)}(|\chih(0,\ub)|^2 + |\alpha_F(0,\ub)|^2) d\ub + C\delta^{\frac{1}{2}}.
\end{align*}
Since $|\Omega(u,\ub) - 1 | \leq C \delta^{\frac{1}{2}}$, we have
\begin{align*}
 \tr\chi(u, \delta) &\leq \frac{2}{r_0-u} - \frac{r_0^2}{(r_0-u)^2}\int_0^\delta(|\chih(0,\ub)|^2 + |\alpha_F(0,\ub)|^2) d\ub + C\delta^{\frac{1}{2}}.
\end{align*}
Thus, if the left-hand side is negative, this would be a sufficient condition for $S_{u,\delta}$ to be trapped. This is equivalent to
\begin{equation}\label{Formation_Condition_1}
\int_0^\delta|\chih(0,\ub)|^2 + |\alpha_F(0,\ub)|^2 d\ub > (1+C \delta^{\frac{1}{2}})\frac{2(r_0-u)}{r_0^2}
\end{equation}

We have to make sure that on the initial surface $H_0$, there is no trapped surface. According to \eqref{Final_Equation_1} (notice that $\Omega =1$ and $\omega=0$ on $H_0$), we have $\nabla_4 \tr \chi =- \frac{1}{2}(\tr \chi)^2  -|\chih|^2 - |\alpha_F|^2 \geq -\frac{2}{r_0^2} -(|\chih|^2 + |\alpha_F|^2)$, this implies
\begin{align*}
 \tr\chi(0,\ub) &\geq \tr\chi(0,0) -\frac{2\delta}{r_0^2} - \int_0^\delta|\chih(0,\ub)|^2 + |\alpha_F(0,\ub)|^2 =\frac{2(r_0-\delta)}{r_0^2} - \int_0^\delta |\chih(0,\ub)|^2 + |\alpha_F(0,\ub)|^2.
\end{align*}
So a sufficient condition, that initial hypersurface is free of trapped surfaces, is as follows,
\begin{equation}\label{Formation_Condition_2}
 \int_0^\delta(|\chih(0,\ub)|^2 + |\alpha_F(0,\ub)|^2) d\ub < \frac{2(r_0-\delta)}{r_0^2}.
\end{equation}
Together with \eqref{Formation_Condition_1}, if $\delta$ is sufficiently small, we require the data satisfies
\begin{equation}\label{Initial_Condition_1}
 (1+C \delta^{\frac{1}{2}})\frac{2(r_0-u)}{r_0^2}<\int_0^\delta(|\chih(0,\ub)|^2 + |\alpha_F(0,\ub)|^2) d\ub < \frac{2(r_0-\delta)}{r_0^2}.
\end{equation}
This implies the formation of trapped surfaces. We complete the proof of {\bf Main Theorem}.
\subsection{Verification of Smallness}
We need a refined estimate on $\phit$ defined in Subsection \ref{Tr_estimates_for_nabla_eta_etab}. This is needed for the formation of trapped surfaces.
\begin{proposition} If $\delta$ is sufficiently small, we have
\begin{equation}
 \|\phit\|_{L^{\infty}_{(sc)}(u, \ub)} \lesssim C\varepsilon^{\frac{1}{4}} + C \delta^{\frac{1}{8}}.
\end{equation}
\end{proposition}
\begin{proof}
Recall the definition of $\phit$: $\nabla_3 \phit = \nabla \eta$ on $\Hb_{\ub}$ with $\phit(0,\ub)=0$. Combined with the estimates derived so far, one can easily deduce $\|\phit\|_{L^2_{(sc)}{(u,\ub)}} \lesssim C$. By commuting derivatives, we have
\begin{equation*}
 \nabla_3 \nabla \phit = (\tr\chib_0 + \phi) \cdot \nabla \phit + (\psi \cdot \psi + \betab + \Upsilon \cdot \Upsilon)\cdot \phit + \psi \cdot \nabla \eta + \nabla^2 \eta.
\end{equation*}
In view of the triviality of $\nabla \phit$ on $H_0$ and Proposition \ref{refined estimates on nabla 2 eta}, by Gronwall's inequality, we have
\begin{equation*}
 \|\nabla \phit\|_{L^2_{(sc)}{(u,\ub)}} \lesssim C\delta^{\frac{1}{2}} + \|\nabla^2 \eta\|_{L^2_{(sc)}(\Hb_{\ub})} \lesssim C\delta^{\frac{1}{4}} +\|\nabla \rho\|_{L^2_{(sc)}(\Hb_{\ub})}+\|\nabla \sigma\|_{L^2_{(sc)}(\Hb_{\ub})}.
\end{equation*}
In view of ansatz \eqref{initial_ansatz_2} and {\bf{Theorem C}}, we know $\|(\nabla \rho, \nabla \sigma)\|_{L^2(\Hb_{\ub})} \lesssim C \varepsilon$. Hence, we end the proof by a direct use of Lemma \ref{sobolev_trace_estimates}.
\end{proof}

We turn to the proof of Lemma \ref{Smallness_Lemma}. According to the $L^{\infty}_{(sc)}$ estimates on connection coefficients and Maxwell field, we have
\begin{equation}
 \delta^{\frac{1}{2}} \|(\chi, \omega, \alpha_F)\|_{L^{\infty}} + \|(\eta, \etab, \rho_F,\sigma_F)\|_{L^{\infty}}+\delta^{\frac{1}{2}}\|(\chibh,\trchibt,\omegab, \alphab_F)\|_{L^{\infty}} \lesssim 1.
\end{equation}
In view of the definition of $E_1$ and $E_2$, all those quadratic terms are bounded in $L^{\infty}$ norm by a constant depending only on the initial data. Together with the definition of $F_1$ and $F_2$, it is easy to show the contribution from these terms to $G_1$ and $G_2$ verify the estimates \eqref{Smallness_Estimates} (we bound $r$ and $\Omega$ by $1$ and bound $\Omega-1$ by $C\delta^{\frac{1}{2}}$). It remains to show that for
\begin{align*}
 H_1(u,\ub, \theta) = \int_0^{u}|\nabla \tensor \eta||\chih| du' ,\,\, H_2(u,\ub, \theta) = \int_0^{u}(|\nabla\rho_F|+|\nabla \sigma_F|)|\alpha_F| du'.
\end{align*}
In fact, for $H_1$, we can use $\phit$ to renormalize $\chibh$ to be $\phit - \chibh$, see Section 2.4 of \cite{K-R-09} for details; for $H_2$, it comes directly from {\bf{Theorem C}} and initial ansatz \eqref{initial_ansatz_2}. This completes the proof.

\appendix

\section{Hodge Operators}\label{HodgeOperators}
We review four Hodge operators for horizontal tensor fields. In this subsection, all the functions or tensors are defined on $S_{u,\ub}$. For more detailed account on the subject, we refer the readers to \cite{B-Z} or \cite{Ch-K}.

For one form $F$, we have $\Done (F) = (\divergence F, \curl F)$; For a pair of functions $(F_1,F_2)$, we have $\Donestar$ (which is the dual of $\Done$) as follows, $\Donestar (F_1,F_2) = -\nabla F_1 + \nablastar F_2$; For a traceless symmetric two tensor $F$, we have $\Dtwo F = \divergence F$; For an one form $F$, we have $\Dtwostar$ (which is the dual of $\Dtwo$) acting as $\Dtwostar F = -\frac{1}{2} \widehat{\L_F \gamma}_{ab}= -\frac{1}{2}(\nabla_a F_b +\nabla_b F_a -(\divergence F)\gamma_{ab})$ where $\gamma$ is the induced metric on $S_{u,\ub}$.

We also have the standard Bochner formulas
\begin{align*}
\Donestar \cdot \Done &= -\triangle + K, \quad \Done \cdot \Donestar = -\triangle, \quad \Dtwostar \cdot \Dtwo = -\frac{1}{2}\triangle + K, \quad \Dtwo \cdot \Dtwostar = \frac{1}{2}(\triangle + K).
\end{align*}
Those formulas lead to the standard elliptic estimates for Hodge systems.

\section{Energy Estimates Scheme}
\subsection{Energy Identities for Weyl Fields}\label{Preliminary_Energy_Weyl}
We refer the reader to \cite{Ch-K} for the basic definitions. Let $W_{\alpha\beta\gamma\delta}$ be a Weyl field 
satisfying the following divergence equation, $\Divergence W = J$ where the source term $J_{\alpha\beta\gamma}$ is a Weyl current.
The Hodge dual $\Wstar$ also satisfies a divergence equation $\Divergence \Wstar = J^{*}$ with source term (which is also a Weyl current) $J^{*}{}_{\alpha\beta\gamma}= \frac{1}{2}J_{\alpha}{}^{\mu\nu}\cdot \epsilon_{\mu\nu\beta\gamma}$. In the case when there is a electromagnetic field coupled to the space-time and $W$ is taken to be the Weyl curvature tensor, these divergence identities read as,
\begin{align*}
 D^{\alpha}W_{\alpha\beta\gamma\delta} &= \frac{1}{2}(D_{\gamma}R_{\beta\delta}-D_{\delta}R_{\beta\gamma}),\quad D^{\alpha}{}\Wstar_{\alpha\beta\gamma\delta} =\frac{1}{2}(D_{\mu}R_{\beta\nu}-D_{\nu}R_{\beta\mu})\epsilon^{\mu\nu}{}_{\gamma\delta}.
\end{align*}

Let $Q[W]$ be the Bel-Robinson tensor of $W$, 
it also satisfies a divergence equation,
\begin{equation*}
D^\alpha  Q[W]_{\alpha\beta\gamma\delta}= W_{\beta}{}^{\mu}{}_{\delta}{}^{\nu}J_{\mu\gamma\nu} + W_{\beta}{}^{\mu}{}_{\gamma}{}^{\nu}J_{\mu\delta\nu} + \Wstar_{\beta}{}^{\mu}{}_{\delta}{}^{\nu}J^*{}_{\mu\gamma\nu} + \Wstar_{\beta}{}^{\mu}{}_{\gamma}{}^{\nu}J^*{}_{\mu\delta\nu}.
\end{equation*}
The modified Lie derivative of $\Lh_O W$ satisfies $\Divergence (\Lh_O W)_{\beta\gamma\delta} = J(O,W)_{\beta\gamma\delta} =  \sum_{i=0}^3 J_i(O,W)_{\beta\gamma\delta}$ where
\begin{align*}
 J_0(O,W) = \Lh_O J_{\beta\gamma\delta}&=\L_O J_{\beta\gamma\delta}-(\pih_{\beta}{}^{\mu} J_{\nu\gamma\delta}+\pih_{\gamma}{}^{\mu} J_{\beta\mu\delta} + \pih_{\delta}{}^{\mu} J_{\beta\gamma\mu})+\frac{1}{4}\text{Tr} \pi J_{\beta\gamma\delta},\\
 J_1(O,W)_{\beta\gamma\delta}&= \pih^{\mu\nu}D_\mu W_{\nu\beta\gamma\delta}, \quad J_2(O,W)_{\beta\gamma\delta}= p^\mu W_{\mu\beta\gamma\delta},\\
J_3(O,W)_{\beta\gamma\delta}&=q_{\mu\beta\nu}W^{\mu\nu}{}_{\gamma\delta}+q_{\mu \gamma \nu} W^{\mu}{}_{\beta}{}^{\nu}{}_{\delta}+q_{\mu \delta \nu} W^{\mu}{}_{\beta\gamma}{}^{\nu},
\end{align*}
and
\begin{equation*}
 p_\alpha = D^\mu \pih_{\mu\alpha},\quad q_{\alpha\beta\gamma}= D_\beta \pih_{\gamma\alpha}-D_\gamma \pih_{\beta\alpha}-\frac{1}{3}(p_\gamma g_{\beta\alpha}-p_\beta g_{\gamma\alpha}).
\end{equation*}

Given vector fields $X$, $Y$ and $Z$, we define the current associated to  $X$, $Y$, $Z$ and $W$ to be $P[W](X,Y,Z)_{\alpha} = Q[W]_{\alpha\beta\gamma\delta}X^\beta Y^\gamma Z^\delta$. The space-time divergence of $P[W]$ is $\Divergence P[W](X,Y,Z) = \Divergence Q[W] (X,Y,Z) + (\pi \cdot Q[W])(X,Y,Z)$ where
\begin{equation*}
(\pi \cdot Q[W])(X,Y,Z)=Q[W]_{\alpha\beta\gamma\delta}{}^{(X)}\!\pi^{\alpha\beta}Y^{\gamma}Z^{\delta}+Q[W]_{\alpha\beta\gamma\delta}{}^{(Y)}\!\pi^{\alpha\beta}Z^{\gamma}X^{\delta}+Q[W]_{\alpha\beta\gamma\delta}{}^{(Z)}\!\pi^{\alpha\beta}X^{\gamma}Y^{\delta}.
\end{equation*}
We integrate this identity on domain $\D(u,\ub)$ to derive

\begin{minipage}[!t]{0.2\textwidth}
  \includegraphics[width = 2 in]{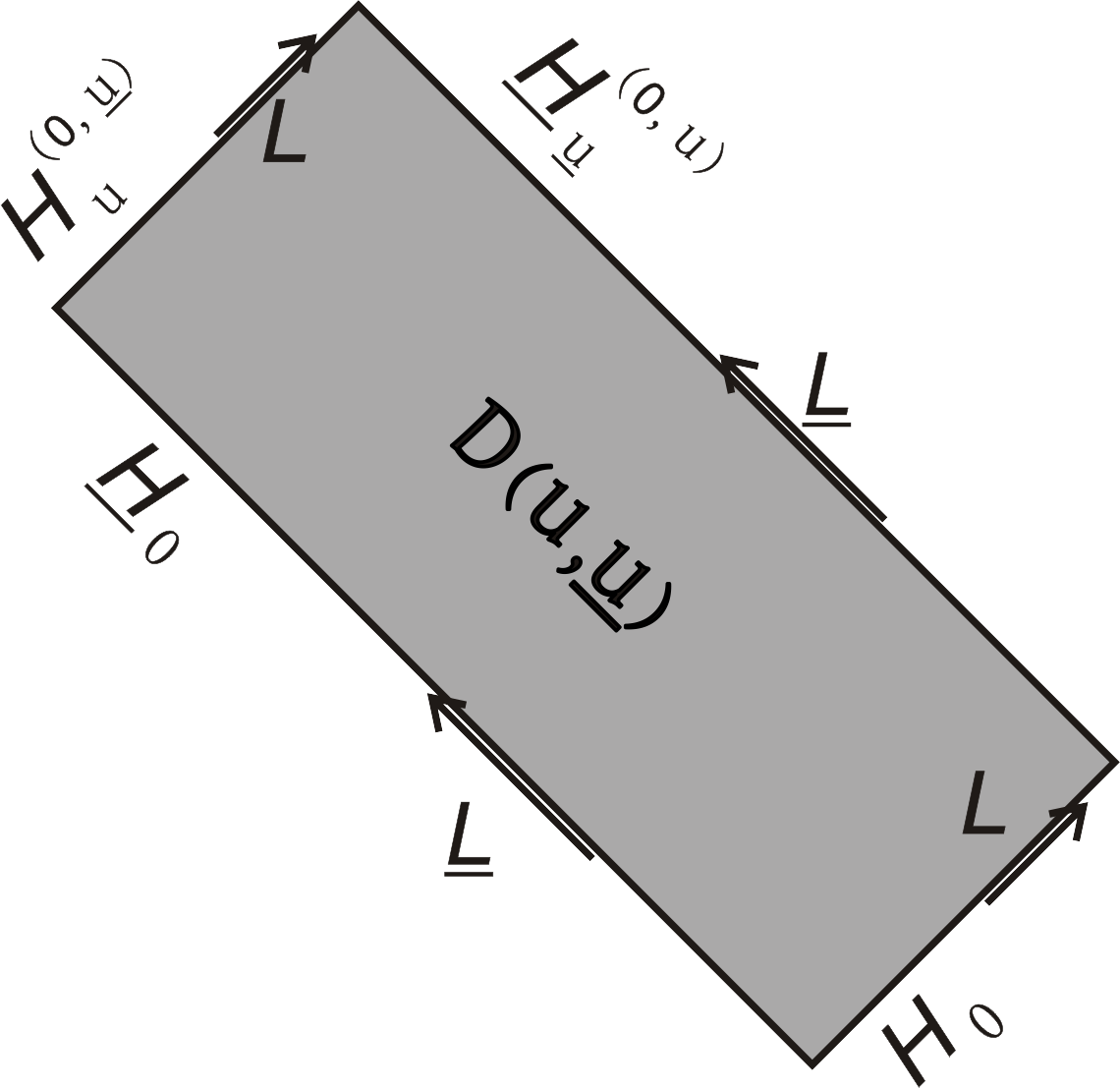}
\end{minipage}
\hspace{0.1\textwidth}
\begin{minipage}[!t]{0.6\textwidth}
\begin{align}\label{stokes_Bel_Robinson}
 &\quad\int_{H_u} Q(X,Y,Z,L)+\int_{\Hb_{\ub}} Q(X,Y,Z,\Lb) \notag\\
 &= \int_{H_0} Q(X,Y,Z,L)+\int_{\Hb_{0}} Q(X,Y,Z,\Lb)\\
&+\doubleint_{\D(u,\ub)} \Divergence Q (X,Y,Z) + \doubleint_{\D(u,\ub)} (\pi \cdot Q)(X,Y,Z).\notag
\end{align}
\end{minipage}

Given a vector field $O$, we use $\pi$ to denote its deformation tensor $^{(O)}\!\pi$.  


\subsection{Energy Identities for Maxwell Fields}\label{Preliminary_Energy_Maxwell}
We refer the reader to the thesis of Zipser in \cite{B-Z} for more details. Let $F_{\alpha\beta}$ be a 2-form. Its Hodge dual $\Fstar_{\alpha\beta}$ is defined as $\Fstar_{\alpha\beta} = \frac{1}{2}\epsilon_{\alpha\beta\mu\nu} F^{\mu\nu}$. The energy-momentum tensor $T[F]_{\alpha\beta}$ associated to $F_{\alpha\beta}$ is defined as
\begin{equation*}
 T[F]_{\alpha\beta} = F_{\alpha\mu}F_{\beta}{}^{\mu}-\frac{1}{4}(F\cdot F)g_{\alpha \beta} = F_{\alpha\mu}F_{\beta}{}^{\mu} + \Fstar{}_{\alpha\mu}{}\Fstar_{\beta}{}^{\mu}.
\end{equation*}
If $(F_{\alpha\beta},\Fstar{}_{\alpha\beta})$ satisfies divergence identities $D^\mu F_{\mu\alpha} = J_\alpha$ and $D^\mu {}\Fstar_{\mu\alpha} = J'_\alpha$, then $T[F]_{\alpha\beta}$ satisfies,
\begin{equation}\label{divergence_of_T[F]}
 D^\mu T[F]_{\mu\alpha} = F_{\alpha}{}_{\mu}J_{\mu} + {}\Fstar_{\alpha}{}_{\mu}J'_{\mu}.
\end{equation}
 Given a vector field $O$, we define the modified Lie derivative of $F$ along $O$ as
\begin{align*}
 \Lh_O F_{\alpha\beta} = \L_O F_{\alpha\beta}-(\pih_{\mu\alpha}\,F^{\mu}{}_{\beta} -\pih_{\mu\beta}\,F^{\mu}{}_{\alpha}).
\end{align*}
We remark that $\Lh_O$ commutes with Hodge star operator.
If $F_{\alpha\beta}$ solves the Maxwell equations, then the divergence of $T[\Lh_O F]_{\alpha\beta}$ is
\begin{equation*}
 \Divergence T[\Lh_O F]_\alpha = \Lh_{O} F_{\alpha}{}^{\mu}J(O,F)_{\mu}+{}^*\!\Lh_{O} F_{\alpha}{}^{\mu}J'(O,F)_{\mu},
\end{equation*}
where
\begin{align*}
 J(O,F)_{\mu} &= \pih^{\alpha\beta}D_{\beta}F_{\alpha\mu} + D^\alpha \pih_{\alpha\beta} F^{\beta}{}_{\mu} + (D_{\mu} \pih_{\alpha\beta}-D_{\beta}\pih_{\mu\alpha})F^{\alpha\beta},\\
 J'(O,F)_{\mu} &= \pih^{\alpha\beta}D_{\beta}{}\Fstar{}_{\alpha\mu} + D^\alpha \pih_{\alpha\beta} {}\Fstar{}^{\beta}{}_{\mu} + (D_{\mu} \pih_{\alpha\beta}-D_{\beta}\pih_{\mu\alpha}){}\Fstar{}^{\alpha\beta}.
\end{align*}

Given a vector field $X$, the current associated to $X$ and $F$ is defined to be $P[F](X)_{\alpha} = T[F]_{\alpha\mu}X^\mu$. The space-time divergence of $P[F]$ is $\Divergence P[F](X) =\Divergence T[F] (X) + \pih \cdot T[F]$. We integrate this identity on the domain $\D(u,\ub)$ to derive
\begin{align}\label{stokes_energy_momentum}
 \int_{H_u} T[F](X,L)+ \int_{\Hb_{\ub}} T[F](X,\Lb) &=\int_{H_0} T[F](X,L)+\int_{\Hb_{0}} T[F](X,\Lb)\\
& +\doubleint_{\D(u,\ub)} \Divergence T[F] (X) + \doubleint_{\D(u,\ub)} \pih \cdot T[F].\notag
\end{align}

\end{document}